\documentclass[10pt, leqno]{article}

\usepackage{amsmath,amssymb,amsthm, epsfig}

\usepackage{hyperref}

\usepackage{amsmath}

\usepackage[pdftex]{color}

\usepackage{color}

\usepackage{ulem}

\usepackage{mathrsfs}

\usepackage{wasysym}

\usepackage{stackrel}

\usepackage{cancel}

\title{\large{\bf Weighted Orlicz-Sobolev and variable exponent Morrey regularity for fully nonlinear parabolic PDEs with oblique boundary conditions and applications}}
\author{ \it by \smallskip \\ Junior da S. Bessa \footnote{\noindent Universidade Estadual de Campinas - UNICAMP.\\ Instituto de Matem\'{a}tica, Estat\'{i}stica e Computa\c{c}\~{a}o Cient\'{i}fica - IMECC.\\ Departamento  de Matemática. Bar\~{a}o Geraldo, Campinas - SP, Brazil. \noindent \texttt{E-mail address: \url{jbessa@unicamp.br}}},\,\,
Jo\~{a}o Vitor  da Silva
\footnote{\noindent Universidade Estadual de Campinas - UNICAMP.\\ Instituto de Matem\'{a}tica, Estat\'{i}stica e Computa\c{c}\~{a}o Cient\'{i}fica - IMECC.\\ Departamento  de Matemática. Bar\~{a}o Geraldo, Campinas - SP, Brazil. \noindent \texttt{E-mail address: \url{jdasilva@unicamp.br}}}, \,\, Maria N.B. Frederico \footnote{\noindent Universidade Federal do Cear\'{a}- UFC. Campus de Russas. Russas - CE, Brazil. \noindent \texttt{E-mail address: \url{nildebarreto@ufc.br}}}\\ $\&$\\ Gleydson C. Ricarte \footnote{\noindent Universidade Federal do Cear\'{a} - UFC. Departamento  de Matem\'{a}tica. Fortaleza - CE, Brazil. \noindent \texttt{E-mail address: \url{ricarte@mat.ufc.br}}}}

\newlength{\hchng}
\newlength{\vchng}
\def \div {\mathrm{div}}

\newcommand{\defeq}{\mathrel{\mathop:}=}

\newtheorem{theorem}{Theorem}[section]

\newtheorem{lemma}[theorem]{Lemma}

\newtheorem{proposition}[theorem]{Proposition}

\newtheorem{corollary}[theorem]{Corollary}

\theoremstyle{definition}

\newtheorem{definition}[theorem]{Definition}

\newtheorem{example}[theorem]{Example}

\theoremstyle{remark}

\newtheorem{remark}[theorem]{Remark}

\numberwithin{equation}{section}

\newcommand{\intav}[1]{\mathchoice {\mathop{\vrule width 6pt height 3 pt depth  -2.5pt
\kern -8pt \intop}\nolimits_{\kern -6pt#1}} {\mathop{\vrule width
5pt height 3  pt depth -2.6pt \kern -6pt \intop}\nolimits_{#1}}
{\mathop{\vrule width 5pt height 3 pt depth -2.6pt \kern -6pt
\intop}\nolimits_{#1}} {\mathop{\vrule width 5pt height 3 pt depth
-2.6pt \kern -6pt \intop}\nolimits_{#1}}}

\begin{document}
\maketitle

\begin{abstract}
In this manuscript, we establish global weighted Orlicz-Sobolev and variable exponent Morrey–Sobolev estimates for viscosity solutions to fully nonlinear parabolic equations subject to oblique boundary conditions on a portion of the boundary, within the following framework:
\[
\left\{
\begin{array}{rclcl}
F(D^2u,Du,u,x,t) - u_{t} &=& f(x,t) & \text{in} & \Omega_{\mathrm{T}}, \\
\beta \cdot Du + \gamma u &=& g(x,t) & \text{on} & \mathrm{S}_{\mathrm{T}}, \\
u(x,0) &=& 0 & \text{on} & \Omega_{0},
\end{array}
\right.
\]
where \(\Omega_{\mathrm{T}} = \Omega \times (0,\mathrm{T})\) denotes the parabolic cylinder with spatial base \(\Omega\) (a bounded domain in \(\mathbb{R}^{n}\), \(n \geq 2\)) and temporal height \(\mathrm{T} > 0\), \(\mathrm{S}_{\mathrm{T}} = \partial \Omega \times (0,\mathrm{T})\), and \(\Omega_{0} = \Omega \times \{0\}\). Additionally, \(f\) represents the source term of the parabolic equation, while the boundary data are given by \(\beta\), \(\gamma\), and \(g\). Our first main result is a global weighted Orlicz–Sobolev estimate for the solution, obtained under asymptotic structural conditions on the differential operator and appropriate assumptions on the boundary data, assuming that the source term belongs to the corresponding weighted Orlicz space. Leveraging these estimates, we demonstrate several applications, including a density result within the fundamental class of parabolic equations, regularity results for the related obstacle problem, and weighted Orlicz–BMO estimates for both the Hessian and the time derivative of the solution. Lastly, we derive variable exponent Morrey–Sobolev estimates for the problem via an extrapolation technique, which are of
independent mathematical interest.

\vspace{0.2cm}

\end{abstract}


\section{Introduction}

\hspace{0.4cm} This paper investigates Hessian and time derivative estimates for viscosity solutions to the following fully nonlinear parabolic mixed boundary value problem:
\begin{equation}\label{1.1}
\left\{
\begin{array}{rclcl}
F(D^2u,Du,u,x,t) - u_{t} &=& f(x,t) & \text{in} & \Omega_{\mathrm{T}}, \\
\beta \cdot Du(x, t) + \gamma u(x, t) &=& g(x,t) & \text{on} & \mathrm{S}_{\mathrm{T}}, \\
u(x, 0) &=& 0 & \text{on} & \Omega_{0},
\end{array}
\right.
\end{equation}
where \(\Omega \subset \mathbb{R}^n\) (\(n \geq 2\)) is a bounded domain with smooth boundary, \(\mathrm{T} > 0\), and the data \(f\), \(\gamma\), \(g\), and \(\beta\) satisfy appropriate regularity assumptions. The nonlinear operator \(F: \mathrm{Sym}(n) \times \mathbb{R}^n \times \mathbb{R} \times \Omega \times \mathbb{R} \to \mathbb{R}\), where \(\mathrm{Sym}(n)\) denotes the space of real symmetric \(n \times n\) matrices, is a uniformly parabolic second-order operator, meaning that there exist constants \(0 < \lambda \leq \Lambda < \infty\), referred to as the \textit{parabolicity constants}, such that
\begin{equation}\label{Unif.Ellip.}
\lambda\|\mathrm{N}\| \leq F(\mathrm{M} + \mathrm{N}, \varsigma, s, x, t) - F(\mathrm{M}, \varsigma, s, x, t) \leq \Lambda \|\mathrm{N}\|
\end{equation}
for all \(\mathrm{M}, \mathrm{N} \in \mathrm{Sym}(n)\) with \(\mathrm{N} \geq 0\) (in the sense of symmetric matrices), and for all \((\varsigma, s, x, t) \in \mathbb{R}^n \times \mathbb{R} \times \Omega \times \mathbb{R}\). The vector field \(\beta : \mathrm{S}_{\mathrm{T}} \to \mathbb{R}^n\) is assumed to be of unit length, and \(\gamma, g : \mathrm{S}_{\mathrm{T}} \to \mathbb{R}\) are given real-valued functions.

Hence, under suitable regularity assumptions on the boundary data \(\beta\), \(\gamma\), and \(g\), we establish global weighted Orlicz and variable exponent Morrey estimates for the Hessian and the temporal derivative of viscosity solutions to problem \eqref{1.1}, assuming that the source term \(f\) belongs to the corresponding function space, and under conditions weaker than the convexity of the second-order operator \(F\). 

Specifically and under suitable assumptions, we obtain the global weighted Orlicz regularity estimates (see Theorem \ref{T1} for further details)  
$$
\|u\|_{W^{2,\Upsilon}_{\omega}(\Omega_{\mathrm{T}})}\leq\mathrm{C}(\text{universal})\left(\|u\|_{L^{\infty}(\Omega_{\mathrm{T}})}^{n+1}+\|f\|_{L^{\Upsilon}_{\omega}(\Omega_{\mathrm{T}})}+\|g\|_{C^{1,\alpha}(\mathrm{S}_\mathrm{T})}\right),
$$
as well as global variable exponent Morrey regularity estimates (see Theorem \ref{T2} for further details)
$$
\|u\|_{W^{2,\varsigma(\cdot),\varrho(\cdot)}(\Omega_{\mathrm{T}})}\leq\mathrm{C}(\text{universal})\|f\|_{L^{\varsigma(\cdot),\varrho(\cdot)}(\Omega_{\mathrm{T}})},
$$

The condition imposed on the operator \( F \) is rooted in tangential analysis, specifically in the concept of the \textit{recession of an operator}. We adopt the terminology \textit{Recession operator}, following the framework introduced by Giga and Sato in the context of Hamilton-Jacobi equations \cite{GS01}:

\begin{definition}[{\bf Recession operator}]\label{DefAC}
	We say that \( F: \mathrm{Sym}(n) \times \mathbb{R}^n \times \mathbb{R} \times \Omega \times \mathbb{R} \to \mathbb{R} \) is an \emph{asymptotically fully nonlinear parabolic operator} if there exists a uniformly parabolic operator \( F^{\sharp}: \mathrm{Sym}(n) \times \mathbb{R}^n \times \mathbb{R} \times \Omega \times \mathbb{R} \to \mathbb{R} \), referred to as the \textit{Recession operator}, such that
	\begin{equation}\label{Reces}
		F^{\sharp}(\mathrm{X}, \varsigma, s, x, t) \defeq \lim_{\tau \to 0^{+}} \tau \cdot F\left(\frac{1}{\tau}\mathrm{X}, \varsigma, s, x, t\right),\tag{Rec}
	\end{equation}
	for all \( \mathrm{X} \in \mathrm{Sym}(n) \), \( \varsigma \in \mathbb{R}^n \), \( s, t \in \mathbb{R} \), and \( x \in \Omega \). For convenience, we introduce the shorthand notation \( F_{\tau}(\mathrm{X}, \varsigma, s, x, t) = \tau \cdot F\left(\frac{1}{\tau}\mathrm{X}, \varsigma, s, x, t\right) \).
\end{definition}

\bigskip

By way of illustration, a limiting profile such as \eqref{Reces} naturally emerges in singularly perturbed free boundary problems governed by fully nonlinear equations, in which the Hessian of solutions blows up along the phase transition interface, i.e., $\partial\{u^{\varepsilon} > \varepsilon\}$, where $u^{\varepsilon}$ satisfies in the viscosity sense:
\[
F(D^2 u^{\varepsilon}, x) = \mathcal{Q}_0(x)\frac{1}{\varepsilon} \zeta \left(\frac{u^{\varepsilon}}{\varepsilon}\right).
\]
In these approximations, we assume $0 < \mathcal{Q}_0 \in C^0(\overline{\Omega})$ and $0 \leq \zeta \in C^{\infty}(\mathbb{R})$ with $\operatorname{supp} \zeta = [0,1]$. Consequently, in this model, the limiting free boundary condition is governed by the operator $F^{\sharp}$ rather than $F$, i.e.,
\[
F^{\sharp}(D u(z_0) \otimes D u(z_0), z_0) = 2\mathrm{T}_0, \quad z_0 \in \partial\{u_0 > 0\},
\]
in an appropriate viscosity framework, for a certain total mass $\mathrm{T}_0 > 0$ (see \cite[Section 6]{RT} for illustrative examples and further details).

\medskip

Moreover, limit profiles such as \eqref{Reces} also arise in the context of higher-order convergence rates in the periodic homogenization of fully nonlinear uniformly parabolic Cauchy problems with rapidly oscillating initial data, as demonstrated below:
\[
\left\{
\begin{array}{rclcl}
  \frac{d}{dt}u^{\varepsilon}(x, t) & = & \frac{1}{\varepsilon^2}F(\varepsilon^2D^2 u^{\varepsilon}, x, t, \frac{x}{\varepsilon}, \frac{t}{\varepsilon}) & \text{in} & \mathbb{R}^n \times (0, T), \\
  u^{\varepsilon}(x, 0) & = & g\left(x, \frac{x}{\varepsilon}\right) & \text{on} & \mathbb{R}^n.
\end{array}
\right.
\]
In this setting, we have the asymptotic behavior:
\[
\displaystyle \lim_{\varepsilon \to 0^+} \frac{1}{\varepsilon^2}F(\varepsilon^2 \mathrm{X}, x, t, y, s) = F^{\sharp}( \mathrm{X}, x, t, y, s) ,
\]
uniformly for all $( \mathrm{X}, x, t, y, s) \in (\operatorname{Sym}(n) \setminus \{\mathcal{O}_{n\times n}\}) \times \mathbb{R}^n \times [0, T] \times \mathbb{T}^n \times \mathbb{T}$ (see \cite{KL20}). Consequently, there exists a unique function $v: \mathbb{R}^n \times [0, T] \times \mathbb{T}^n \times [0, \infty) \to \mathbb{R}$ such that $v(x, t, \cdot, \cdot)$ is a viscosity solution of
\[
\left\{
\begin{array}{rclcl}
  \frac{d}{ds}v(y, s) & = & F^{\sharp}(D_y^2 v, x, t, y, s) & \text{in} & \mathbb{T}^n \times (0, \infty), \\
  v(x, t, y,  0) & = & g(y, x) & \text{on} & \mathbb{T}^n.
\end{array}
\right.
\]


\bigskip

The cornerstone of our approach to establishing the results to be presented lies in the assumption that the recession operator \( F^{\sharp} \) satisfies certain structural properties (e.g., convexity/concavity or appropriate \textit{a priori} estimates). Through tangential analysis techniques, we can derive regularity results for the solutions corresponding to the initial data of problem \eqref{1.1}.

\medskip

Additionally, the vector field \(\beta\) and the function \(\gamma\) define the boundary operator:
\begin{equation*}
\mathcal{B}(\vec{v}, s, x, t) = \beta(x, t) \cdot \vec{v} + \gamma(x, t)s, \quad (v, s, x, t) \in \mathbb{R}^n \times \mathbb{R} \times \mathrm{S}_{\mathrm{T}}.
\end{equation*}

Throughout this paper, we assume the existence of a positive constant \(\mu_0 > 0\) such that \(\beta \cdot \vec{\mathbf{n}} \geq \mu_0\) on \(\mathrm{S}_{\mathrm{T}}\), where \(\vec{\mathbf{n}}\) denotes the unit outward normal vector to \(\Omega\). Geometrically, this means that \(\beta\) is not tangential to the lateral boundary \(\mathrm{S}_{\mathrm{T}}\) of the parabolic cylinder \(\Omega_{\mathrm{T}}\). This obliqueness condition ensures that problem \eqref{1.1} is well-posed because of the Shapiro–Lopatinskii compatibility condition (cf. \cite{MaPaVi}).

The analysis of models of the type \eqref{1.1} is motivated by their wide range of applications, including the study of Brownian motion, reflected shock waves in transonic flow, and the generalization of problems with Neumann and Robin boundary conditions (see, e.g., \cite{BessaOrlicz}, \cite{Ch}, \cite{Sato}, and \cite{Leiberman} for further references).

As an application of our results, we establish a density result for viscosity solutions of the problem
\begin{eqnarray}\label{problemaflat'}
\left\{
\begin{array}{rclcl}
F(D^{2}u,x,t)-u_{t} & = & f(x,t) & \text{in} & \mathrm{Q}^{+}_{1}, \\
\beta \cdot Du + \gamma u & = & g(x,t) & \text{on} & \mathrm{Q}^{*}_{1},
\end{array}
\right.    
\end{eqnarray}
in weighted Orlicz–Sobolev spaces, within the fundamental class of solutions \(\mathcal{S}\) to parabolic equations (see Definition \ref{calssefundamental} and Theorem \ref{Thm5.1-Density} for more details).

Subsequently, we address the existence/uniqueness and Calder\'{o}n–Zygmund type estimates for viscosity solutions of obstacle problems involving oblique tangential derivatives of the form
\begin{equation*}
\left\{
\begin{array}{rclcl}
F(D^2 u,Du,x,t)-\frac{\partial u}{\partial t} &\le& f(x,t)& \text{in} & \Omega_{\mathrm{T}}, \\
(F(D^2 u, Du,x,t)-\frac{\partial u}{\partial t}- f)(u-\phi) &=& 0 &\text{in}& \Omega_{\mathrm{T}},\\
u(x, t) &\ge& \phi(x, t) &\text{in}& \Omega_{\mathrm{T}},\\
\beta \cdot Du(x, t)+\gamma u(x, t)&=&g(x,t) &\text{on}& \mathrm{S}_{T},\\
u(x,0)&=&0 &\text{in}& \overline{\Omega},
\end{array}
\right.
\end{equation*}
for appropriate data \(f\), \(\beta\), \(\gamma\), and \(g\), and an obstacle \(\phi\) (see Theorem \ref{T3} and Corollary \ref{Uniqueness-result} for such results). Such free boundary problems have attracted significant interest over recent decades due to their connections with extensions of the classical theory for the heat operator and their non-variational counterparts (see, e.g., \cite{Geo.Mil} for related results).

A noteworthy byproduct of our analysis is the investigation of problem \eqref{problemaflat'} when the source term \(f\) belongs to \textit{weighted Orlicz–BMO spaces} \(L^{\Upsilon}_{\omega}\text{-}\mathrm{BMO}\) (see Definition \ref{deforlicbmospace}). In this setting, we demonstrate that, under suitable assumptions, both the Hessian \(D^{2}u\) and the time derivative \(u_{t}\) possess \(L^{\Upsilon}_{\omega}\)-BMO regularity (see Theorem \ref{BMO} for such a result).

Finally, to connect the variable exponent Morrey regularity for problem \eqref{1.1} (see Theorem \ref{T2}) with variable exponent H\"{o}lder spaces, we establish a Campanato-type theorem for these spaces—a result which, to the best of our knowledge, has not yet been available in the parabolic context. More precisely, under appropriate conditions on the data, we show that
\[
\mathfrak{L}^{\varsigma(\cdot),\varrho(\cdot)}(\Omega_{\mathrm{T}}) \cong C^{0,\alpha(\cdot)}(\overline{\Omega_{\mathrm{T}}}),
\]
for some function \(\alpha=\alpha(\cdot)\) (see Theorem \ref{campanato} for further details).

\subsection{Structural assumptions and further information}

\hspace{0.4cm} We begin this section by introducing some notations and definitions that will be used throughout the manuscript:

\begin{itemize}
    \item[\checkmark] For any point \( x = (x_{1}, \ldots, x_{n-1}, x_{n}) \in \mathbb{R}^{n} \), we write \( x = (x', x_{n}) \), where \( x' = (x_{1}, \ldots, x_{n-1}) \);
    
    \item[\checkmark] Given a set \( U \subset \mathbb{R}^{n} \times \mathbb{R} \) and \( r > 0 \), we denote \( rU = \{ (rx, r^{2}t) \in \mathbb{R}^{n} \times \mathbb{R} \; ; \; (x,t) \in U \} \);

    \item[\checkmark] \( \mathrm{B}_{r}(x) \) denotes the open ball of radius \( r > 0 \) centered at \( x \in \mathbb{R}^{n} \). In particular, \( \mathrm{B}_{r} := \mathrm{B}_{r}(0) \);
    
    \item[\checkmark] \( \mathrm{B}_{r}^{+} := \mathrm{B}_{r} \cap \mathbb{R}^{n}_{+} \). Additionally, we define \( \mathrm{T}_{r} := \{ (x', 0) \in \mathbb{R}^{n-1} \; ; \; |x'| < r \} \), and \( \mathrm{T}_{r}(x_{0}) := \mathrm{T}_{r} + x_{0}' \), where \( x_{0} = (x_{0}', (x_{0})_{n}) \);
    
    \item[\checkmark] \( \mathrm{B}_{r}^{+}(x) := \mathrm{B}_{r}^{+} + x \) denotes the upper half-ball of radius \( r \) centered at \( x \);
    
    \item[\checkmark] The parabolic cylinder centered at \( (x, t) \in \mathbb{R}^{n} \times \mathbb{R} \) with radius \( r > 0 \) is defined as \( \mathrm{Q}_{r}(x, t) := \mathrm{B}_{r}(x) \times (t - r^{2}, t) \). In particular, \( \mathrm{Q}_{r} := \mathrm{Q}_{r}(0,0) \);
    
    \item[\checkmark] We define \( \mathrm{Q}^{+}_{r}(x,t) := \mathrm{B}^{+}_{r}(x) \times (t - r^{2}, t) \), \( \mathrm{Q}^{+}_{r} := \mathrm{Q}^{+}_{r}(0,0) \), \( \mathrm{Q}_{r}^{*}(x,t) := \mathrm{T}_{r}(x) \times (t - r^{2}, t) \), and \( \mathrm{Q}_{r}^{*} := \mathrm{Q}_{r}^{*}(0,0) \);
   \item[\checkmark] For \(|\nu|\leq r\), we define \(\mathrm{Q}_{r}^{\nu}=\mathrm{Q}_{r}\cap\{x_{n}>-\nu\}\) and  \(\mathrm{Q}_{r}^{\nu}(x_{0},t_{0})=\mathrm{Q}_{r}^{\nu}+(x_{0},t_{0})\);
    \item[\checkmark] The parabolic distance between two points \(X=(x,t)\) and \(Y=(y,s)\) in \(\mathbb{R}^{n+1}\) is denoted by \(d_{p}(X,Y)=\max\{|x-y|,|t-s|^{1/2}\}\); 
    \item[\checkmark] For \( n \geq 2 \) and \( r > 0 \), the open cube of side length \( r \) in \( \mathbb{R}^{n} \) is denoted by
    \[
        K^{n}_{r} := \underbrace{\left(-\frac{r}{2}, \frac{r}{2}\right) \times \cdots \times \left(-\frac{r}{2}, \frac{r}{2}\right)}_{n\text{ factors}};
    \]
    
    \item[\checkmark] For a point \( (x_{0}, t_{0}) \in \Omega_{\mathrm{T}} \) and \( r > 0 \), we define the parabolic neighborhood \( \Omega_{\mathrm{T}}(x_{0}, t_{0}; r) := \Omega_{\mathrm{T}} \cap \mathrm{Q}_{r}(x_{0}, t_{0}) \);
    
    \item[\checkmark] For a function \( u = u(x,t) \), we denote its time derivative by \( u_{t} \) (or \( \frac{\partial u}{\partial t} \)), its spatial gradient by \( Du = (u_{x_{1}}, \ldots, u_{x_{n}}) \), and its Hessian matrix by \( D^{2}u = (u_{x_{i}x_{j}})_{n \times n} \).
\end{itemize}

\medskip

On the other hand, recall that a function \( \Phi: [0, +\infty) \to [0, +\infty) \) is called an \(\mathrm{N}\)-function if it is convex, increasing, continuous, satisfying \( \Phi(0) = 0 \), and \( \Phi(s) > 0 \) for all \( s > 0 \), and
\[
\lim_{s \to 0^{+}} \frac{\Phi(s)}{s} = 0 \quad \text{and} \quad \lim_{s \to +\infty} \frac{\Phi(s)}{s} = +\infty.
\]

We say that an \(\mathrm{N}\)-function \( \Phi \) satisfies the \( \Delta_{2} \) condition (respectively, the \( \nabla_{2} \) condition) if there exists a constant \( \mathrm{k}_{1} > 1 \) (respectively, \( \mathrm{k}_{2} > 1 \)) such that
\[
\Phi(2s) \leq \mathrm{k}_{1} \Phi(s) \quad \left( \text{respectively, } \Phi(s) \leq \frac{1}{2\mathrm{k}_{2}} \Phi(\mathrm{k}_{2}s) \right), \quad \forall s > 0.
\]
Moreover, we write \( \Phi \in \Delta_{2} \) (respectively, \( \Phi \in \nabla_{2} \)) to indicate that \( \Phi \) satisfies the \( \Delta_{2} \) condition (respectively, the \( \nabla_{2} \) condition). When both conditions hold, we write \( \Phi \in \Delta_{2} \cap \nabla_{2} \).

For any function \( \Phi \in \Delta_{2} \cap \nabla_{2} \), its lower index is defined by
\[
i(\Phi) := \lim_{s \to 0^{+}} \frac{\log(h_{\Phi}(s))}{\log s} = \sup_{0 < s < 1} \frac{\log(h_{\Phi}(s))}{\log s},
\]
where
\[
h_{\Phi}(s) := \sup_{t > 0} \frac{\Phi(ts)}{\Phi(t)}, \quad s > 0.
\]

\begin{remark}
The functions \(\Phi(s) = s^p\) and \(\overline{\Phi}(s) = s^p \log(s + 1)\), for \(p > 1\), are examples of \(\mathrm{N}\)-functions that satisfy the condition \(\Delta_2 \cap \nabla_2\), with \(i(\Phi) = i(\overline{\Phi}) = p\). Moreover, if \(\Phi \in \Delta_2\), then \(i(\Phi) > 1\) (cf. \cite{fioreza} for more details).
\end{remark}

We now recall the notion of \textit{weights}. A function \(\omega \in L^1_{\text{loc}}(\mathbb{R}^{n+1})\) is called a \textit{weight} if it takes values in the interval \((0, +\infty)\) almost everywhere. In this case, we identify \(\omega\) with the measure
\begin{eqnarray*}
\omega(U) = \int_{U} \omega(x,t)\,dx\,dt,
\end{eqnarray*}
for every Lebesgue measurable set \(U \subset \mathbb{R}^{n+1}\). We say that a weight \(\omega\) belongs to the \textit{Muckenhoupt class} \(\mathfrak{A}_q\), for some \(q \in (1, \infty)\), and write \(\omega \in \mathfrak{A}_q\), if
\begin{eqnarray*}
[\omega]_{q, \mathrm{Q}} \defeq \sup_{\mathrm{Q} \subset \mathbb{R}^{n+1}} \left(\intav{\mathrm{Q}} \omega(x,t)\,dx\,dt\right) \left(\intav{\mathrm{Q}} \omega(x,t)^{\frac{-1}{q-1}}\,dx\,dt\right)^{q-1} < \infty,
\end{eqnarray*}
where the supremum is taken over all parabolic cubes \(\mathrm{Q} \subset \mathbb{R}^{n+1}\).

We are now in a position to define one of the principal functional spaces that will be of interest in this work.

\begin{definition}\label{weightedOrliczspacesdef}
Let \(\Phi \in \Delta_2 \cap \nabla_2\) be an \(\mathrm{N}\)-function, \(\omega\) a weight, and \(U \subset \mathbb{R}^{n+1} = \mathbb{R}^n \times \mathbb{R}\) a measurable set. The \textit{weighted Orlicz space} \(L^\Phi_{\omega}(U)\) is defined as the space of all measurable functions \(f: U \to \mathbb{R}\) such that
\begin{eqnarray*}
\rho_{\Phi,\omega}(f) \defeq \int_{U} \Phi(|f(x,t)|)\omega(x,t)\,dx\,dt < \infty,
\end{eqnarray*}
where \(\rho_{\Phi,\omega}(f)\) is referred to as the modular. Owing to the condition \(\Phi \in \Delta_2 \cap \nabla_2\), the Luxemburg norm
\begin{eqnarray*}
\|f\|_{L^{\Phi}_{\omega}(U)} = \inf\left\{s > 0 : \rho_{\Phi,\omega}\left(\frac{f}{s}\right) \leq 1\right\}
\end{eqnarray*}
renders \(L^\Phi_{\omega}(U)\) a reflexive Banach space. Furthermore, the \textit{weighted Orlicz–Sobolev space} \(W^{k, \Phi}_{\omega}(U)\) consists of all measurable functions \(f: U \to \mathbb{R}\) such that all distributional derivatives \(D^{r}_{t}D_{x}^{s}f\), with \(0 \leq 2r + s \leq k\), belong to \(L^\Phi_{\omega}(U)\). This space is equipped with the norm
\begin{eqnarray*}
\|f\|_{W^{k,\Phi}_{\omega}(U)} = \sum_{j=0}^{k} \sum_{\genfrac{}{}{0pt}{}{r,s \geq 0}{2r + s = j}} \|D_{t}^{r}D_{x}^{s}f\|_{L^{\Phi}_{\omega}(U)}.
\end{eqnarray*}
\end{definition}

\begin{remark}
Regarding Definition~\ref{weightedOrliczspacesdef}:
\begin{itemize}
\item[\checkmark] It is worth noting that if \(\Phi(s) = s^p\) for \(p > 1\), then \(L^\Phi_{\omega}(U)\) coincides with the classical weighted Lebesgue space \(L^p_{\omega}(U)\), and \(W^{k, \Phi}_{\omega}(U)\) coincides with the weighted Sobolev space \(W^{k, p}_{\omega}(U)\).
\item[\checkmark] The spaces \(L^\Phi_{\omega}(U)\) can be seen as intermediate between Lebesgue spaces. More precisely, there exist constants \(1 < p_1 \leq p_2 < \infty\) such that
\begin{eqnarray*}
L^{\infty}(U) \subset L^{p_2}_{\omega}(U) \subset L^{\Phi}_{\omega}(U) \subset L^{p_1}_{\omega}(U) \subset L^{1}(U),
\end{eqnarray*}
see \cite{BessaOrlicz} and \cite{KokiMiro} for further details.
\item[\checkmark] The modular in the definition of \(L^{\Phi}_{\omega}(U)\) can be estimated by
\begin{eqnarray}\label{modularestimate}
\rho_{\Phi,\omega}(g) \leq \mathrm{C}\left(\|g\|_{L^{\Phi}_{\omega}(U)}^{p_2} + 1\right),
\end{eqnarray}
where \(\mathrm{C} > 0\) is a constant independent of \(g\) (cf. \cite{BLOK}).
\end{itemize}
\end{remark}

To state one of our main results, we require the following lemma, established by Byun et al. in \cite[Lemma 5]{BLOK}, which ensures, under suitable conditions, that \(L^\Phi_{\omega}(U)\) can be continuously embedded into a Lebesgue space.

\begin{lemma}\label{mergulhoorliczlebesgue}
Let \(\Phi\) be an \(\mathrm{N}\)-function satisfying \(\Phi \in \Delta_2 \cap \nabla_2\), let \(\omega \in \mathfrak{A}_{i(\Phi)}\), and suppose \(\Omega \subset \mathbb{R}^n\) is bounded. Then, there exists a constant \(p_0 \in (1, i(\Phi))\), depending only on \(i(\Phi)\) and \(\omega\), such that \(L^\Phi_{\omega}(U)\) is continuously embedded in \(L^{p_0}(U)\). Moreover, the following estimate holds:
\begin{eqnarray*}
\|g\|_{L^{p_0}(U)} \leq \mathrm{C}' \|g\|_{L^{\Phi}_{\omega}(U)}, \quad \forall g \in L^{\Phi}_{\omega}(U),
\end{eqnarray*}
where \(\mathrm{C}' = \mathrm{C}'(n, i(\Phi), \omega) > 0\) is a constant independent of \(g\).
\end{lemma}

Throughout this manuscript, we shall adopt the following assumptions:

\begin{enumerate}
\item[(H1)] \textbf{(Structural conditions)} We assume that the operator \(F: \text{Sym}(n)\times \mathbb{R}^n \times\mathbb{R} \times \Omega \times \mathbb{R} \to \mathbb{R}\) is continuous in each of its variables. Moreover, there exist constants \(0 < \lambda \le \Lambda\), \(\sigma \geq 0\), and \(\xi \geq 0\) such that
\begin{equation*}
\begin{aligned}
\mathcal{M}^{-}_{\lambda,\Lambda}(\mathrm{X}-\mathrm{Y}) - \sigma |\zeta-\eta| -\xi|r-s| &\le F(\mathrm{X}, \zeta, r, x, t)-F(\mathrm{Y}, \eta, s, x, t) \\
&\le \mathcal{M}^{+}_{\lambda, \Lambda}(\mathrm{X}-\mathrm{Y}) + \sigma |\zeta-\eta| + \xi|r-s|
\end{aligned}
\end{equation*}
for all \(\mathrm{X}, \mathrm{Y} \in \text{Sym}(n)\), \(\zeta,\eta \in \mathbb{R}^n\), \(r, s \in \mathbb{R}\), and \((x,t) \in \Omega \times \mathbb{R}\), where
\[
\mathcal{M}^{+}_{\lambda,\Lambda}(\mathrm{X}) \defeq \Lambda \sum_{e_i > 0} e_i + \lambda \sum_{e_i < 0} e_i, \quad \mathcal{M}^{-}_{\lambda,\Lambda}(\mathrm{X}) \defeq \Lambda \sum_{e_i < 0} e_i + \lambda \sum_{e_i > 0} e_i
\]
are the \emph{Pucci extremal operators}, and \(e_i = e_i(\mathrm{X})\) (\(1 \leq i \leq n\)) denote the eigenvalues of \(\mathrm{X}\).

\begin{remark}
In this context, we refer to \(F\) as a \((\lambda, \Lambda, \sigma, \xi)\)-parabolic operator. For normalization purposes, we assume \(F(0, 0, 0, x, t) = 0\) for all \((x, t) \in \Omega \times \mathbb{R}\). This assumption entails no loss of generality, since given any \(F\), the modified operator \(G(\mathrm{X}, p, r, x, t) = F(\mathrm{X}, p, r, x, t) - F(0, 0, 0, x, t)\) still satisfies the same structural condition and remains a \((\lambda, \Lambda, \sigma, \xi)\)-parabolic operator.
\end{remark}

\item[(H2)] \textbf{(Regularity of the data)} The source term satisfies \(|f|^{n+1} \in L^{\Phi}_{\omega}(\Omega_{\mathrm{T}})\) for some \(\Phi \in \Delta_{2} \cap \nabla_{2}\) and \(\omega \in \mathfrak{A}_{i(\Phi)}\). The boundary data \(\gamma, g \in C^{1,\alpha}(\partial \Omega \times (0, \mathrm{T}))\) with \(\gamma \leq 0\), and the vector field \(\beta \in C^{1,\alpha}(\partial \Omega \times (0, \mathrm{T}))\), for some \(\alpha \in (0,1)\).

\item[(H3)] \textbf{(Continuity of the coefficients)} For each fixed point \((x_0, t_0) \in \Omega_{\mathrm{T}}\), we define the oscillation function
\[
\psi_{F}((x,t), (x_0,t_0)) \defeq \sup_{\mathrm{X} \in \mathrm{Sym}(n)} \frac{|F(\mathrm{X}, 0, 0, x, t) - F(\mathrm{X}, 0, 0, x_0, t_0)|}{\|\mathrm{X}\| + 1},
\]
which quantifies the local variation of the coefficients of \(F\) around \((x_0, t_0)\) (cf. \cite{CCS}). When \((x_0,t_0) = (0,0)\), we simply write \(\psi_F(x,t)\). We assume that the map \((x,t) \mapsto F^{\sharp}(\mathrm{X}, 0, 0, x, t)\) is H\"{o}lder continuous in the \(L^p\)-average sense for every \(\mathrm{X} \in \mathrm{Sym}(n)\) and \(p\geq n+1\). More precisely, there exist universal constants\footnote{Throughout this paper, a constant is said to be \textit{universal} if it depends only on \(n, \lambda, \Lambda, p, \mu_0, \|\gamma\|_{C^{1,\alpha}(\partial \Omega)}\), and \(\|\beta\|_{C^{1,\alpha}(\partial \Omega)}\)} \(\hat{\alpha} \in (0,1)\), \(\theta_0 > 0\), and \(0 < r_0 \leq 1\) such that
\[
\left( \intav{\Omega_{\mathrm{T}}(x_0,t_0;r)} \psi_{F^{\sharp}}((x,x_0),(t,t_0))^{p} \, dxdt \right)^{1/p} \le \theta_0 r^{\hat{\alpha}}
\]
for all \((x_0,t_0) \in \overline{\Omega} \times (0, \mathrm{T})\) and \(0 < r \le r_0\).

\item[(H4)] \textbf{(\(C^{2,\alpha}\) interior estimates)} We assume that solutions to the homogeneous problem
\[
F^{\sharp}(D^2 \mathfrak{h}) - \mathfrak{h}_t = 0 \quad \text{in} \quad \mathrm{Q}_1
\]
admit a priori interior estimates in \(C^{2,\alpha}_{\mathrm{loc}}\), that is,
\[
\|\mathfrak{h}\|_{C^{2,\alpha}(\mathrm{Q}_{1/2})} \leq \mathrm{c}_1 \|\mathfrak{h}\|_{L^{\infty}(\mathrm{Q}_1)}
\]
for some constant \(\mathrm{c}_1 > 0\). 

\item[(H5)] \textbf{(\(C^{2,\alpha}\) boundary estimates)} We further assume that the recession operator \(F^{\sharp}\) exists and satisfies boundary a priori estimates up to the boundary. More precisely, for any boundary datum \(g_0 \in C^{1,\alpha}(\mathrm{Q}^{*}_1)\) (for some \(\alpha \in (0, 1)\)), solutions to
\[
\begin{cases}
F^{\sharp}(D^2 \mathrm{h}) - \mathrm{h}_t = 0 & \text{in } \mathrm{Q}^{+}_1 \\
\beta \cdot D\mathrm{h} + \gamma \mathrm{h} = g_0(x,t) & \text{on } \mathrm{Q}^{*}_1
\end{cases}
\]
belong to \(C^{2,\alpha}(\mathrm{Q}^{+}_{1/2})\), and satisfy the estimate
\[
\|\mathrm{h}\|_{C^{2,\alpha}(\mathrm{Q}^{+}_{1/2})} \le \mathrm{c}_2 \left( \|\mathrm{h}\|_{L^{\infty}(\mathrm{Q}^{+}_1)} + \|g_0\|_{C^{1,\alpha}(\mathrm{Q}^{*}_1)} \right)
\]
for some constant \(\mathrm{c}_2 > 0\).
\end{enumerate}

By way of explanation, we discuss the topics mentioned earlier concerning assumptions (H4)-(H5): the regularity assumptions on the governing operator in the problem. A central question in the regularity theory of partial differential equations is identifying the weakest possible conditions on a parabolic operator that still ensure optimal estimates for the second-order $D^{2}u$ and time $u_t$ derivatives of its viscosity solutions. This question, however, is highly nontrivial, and the available answers are only partial.

For instance, Krylov, in \cite{Kry83}, demonstrated that under the assumption of convexity or concavity of the operator, solutions to 
\[
F(D^{2}u) - u_{t} = 0
\]
are of class $C^{2,\alpha}$ (see also \cite{WangII} for related results). On the other hand, Caffarelli and Stefanelli, in \cite{CS}, provided examples of uniformly parabolic equations whose solutions fail to be of class $C^{2,1}$. This illustrates the general impossibility of establishing a classical theory of existence for smooth solutions to such parabolic problems.

More recently, Goffi introduced in \cite{Goffi24} a novel class of operators whose associated solutions admit higher-order regularity estimates in both the elliptic and parabolic settings. Specifically, the author proved that if the governing operator is quasi-convex or quasi-concave, then the classical higher-order H\"{o}lder estimates originally obtained by Krylov are recovered.

In \cite{DaSildosP19}, da Silva and dos Prazeres investigated non-convex, fully nonlinear, second-order parabolic equations of the form
\begin{equation}\label{Eq-DaSdosP}
\frac{\partial u}{\partial t} - F(x,t,D^2u) = f(x,t).
\end{equation} 
The authors assume that \( F: \mathrm{Q}_1 \times \text{Sym}(n) \to \mathbb{R} \) satisfies a uniform ellipticity condition, is differentiable with respect to \( \mathrm{X} \), and possesses a uniformly continuous differential. The primary objective is to analyze the regularity properties of flat viscosity solutions to equation \eqref{Eq-DaSdosP}. In this context, the main  findings of the manuscript are as follows:
\begin{itemize}
    \item[($\mathcal{I}$)] If \( F(\cdot,\mathrm{X}) \) and \( f(\cdot) \) are Dini continuous, then flat solutions of \eqref{Eq-DaSdosP} belong to the class \( C^{2,1,\psi} \), for some modulus of continuity \( \psi \) determined by the Dini character of the data.
    \item[($\mathcal{II}$)] If \( F(\cdot,\mathrm{X}) \) and \( f(\cdot) \) are merely continuous, then flat solutions of \eqref{Eq-DaSdosP} are locally parabolically \( C^{1,\log\text{-}\mathrm{Lip}} \).
\end{itemize}

In conclusion, in a related direction, da Silva and Santos~\cite{DM} studied the parabolic problem
\[
u_{t} - F(D^{2}u,x,t) = f(x,t) \quad \text{in } \mathrm{Q}_{1},
\]
under the assumption that the operator \(F\) has a ``small parabolic aperture''. In such a context, they established Schauder and  $W^{2,p}$ estimates for the corresponding viscosity solutions.

\bigskip

\subsection{Main Theorems}

Our first main result establishes global regularity in weighted Orlicz spaces for viscosity solutions to \eqref{1.1} under the asymptotic regime.

\vspace{0.4cm}

\begin{theorem}[{\bf Global Weighted Orlicz Regularity}]\label{T1}
Let \(\Omega \subset \mathbb{R}^{n}\) be a bounded domain with \(\partial \Omega \in C^{2,\alpha}\) for some \(\alpha \in (0,1)\), and let \(\mathrm{T} > 0\). Suppose the structural conditions \(\mathrm{(H1)-(H5)}\) are satisfied, and let \(u\) be an \(L^{p}\)-viscosity solution of \eqref{1.1}, where \(p = p_{0}(n+1)\) for some constant \(p_{0} > 1\) as in Lemma \ref{mergulhoorliczlebesgue}. Then, \(u \in W^{2,\Upsilon}_{\omega}(\Omega_{\mathrm{T}})\), with \(\Upsilon(s) = \Phi(s^{n+1})\), and the following estimate holds:
\begin{eqnarray*}
\|u\|_{W^{2,\Upsilon}_{\omega}(\Omega_{\mathrm{T}})}\leq\mathrm{C}\left(\|u\|_{L^{\infty}(\Omega_{\mathrm{T}})}^{n+1}+\|f\|_{L^{\Upsilon}_{\omega}(\Omega_{\mathrm{T}})}+\|g\|_{C^{1,\alpha}(\mathrm{S}_\mathrm{T})}\right),
\end{eqnarray*} 
where $\mathrm{C}>0$ depends only on the parameters $n$, $\mathrm{T}$, $\lambda$, $\Lambda$, $\xi$, $\sigma$, $\mu_{0}$, $p_{0}$, $\Phi$, $\omega$, $\mathrm{c}_{1}$, $\mathrm{c}_{2}$,  $\theta_0$, $\|\beta\|_{C^{1,\alpha}(\mathrm{S}_{\mathrm{T}})}$, $\|\gamma\|_{C^{1,\alpha}(\mathrm{S}_{\mathrm{T}})}$, and $\|\partial \Omega\|_{C^{2,\alpha}}$.
\end{theorem}

\vspace{0.4cm}

On the other hand, we also derive a sharp regularity estimate in the framework of variable exponent Morrey spaces. To this end, let \(\varsigma, \varrho \in C^{0}(\Omega_{\mathrm{T}})\) be functions satisfying \(0 \leq \varrho(x,t) \leq \varrho_{0} < n+2\) and \(n+2 < \varsigma_{1} \leq \varsigma(x,t) \leq \varsigma_{2} < \infty\) for all \((x,t) \in \Omega_{\mathrm{T}}\), where \(\varsigma_{1}\) and \(\varsigma_{2}\) are positive constants.

\begin{definition}\label{defmor}
The \textit{variable exponent Morrey space} \(L^{\varsigma(\cdot),\varrho(\cdot)}(\Omega_{\mathrm{T}})\) consists of measurable functions \(h: \Omega_{\mathrm{T}} \to \mathbb{R}\) such that
\begin{eqnarray*}
\rho_{\varsigma(\cdot),\varrho(\cdot)}(h)\defeq\sup_{\genfrac{}{}{0pt}{}{(x,t)\in \Omega_{\mathrm{T}}}{r>0}}\left(\frac{1}{r^{\varrho(x,t)}}\int_{\Omega_{\mathrm{T}}(x,t;r)}|h(y,s)|^{\varsigma(y,s)}\,dyds\right)<\infty,
\end{eqnarray*} 
endowed with the Luxemburg norm
\begin{eqnarray*}
\|h\|_{L^{\varsigma(\cdot),\varrho(\cdot)}(\Omega_{\mathrm{T}})}\defeq \inf\left\{t>0;\, \rho_{\varsigma(\cdot),\varrho(\cdot)}\left(\frac{h}{t}\right)\le 1\right\}.
\end{eqnarray*}

Moreover, the \textit{variable exponent Morrey-Sobolev space}, denoted \(W^{1,\varsigma(\cdot),\varrho(\cdot)}(\Omega_{\mathrm{T}})\) (respectively \(W^{2,\varsigma(\cdot),\varrho(\cdot)}(\Omega_{\mathrm{T}})\)), consists of all measurable functions \(f\) such that \(f\) and \(Df\) (resp. \(f\), \(f_{t}\), \(Df\), and \(D^{2}f\)) belong to \(L^{\varsigma(\cdot),\varrho(\cdot)}(\Omega_{\mathrm{T}})\), with norm defined by
\begin{eqnarray*}
\|f\|_{W^{1,\varsigma(\cdot),\varrho(\cdot)}(\Omega_{\mathrm{T}})}=\|f\|_{L^{\varsigma(\cdot),\varrho(\cdot)}(\Omega_{\mathrm{T}})}+\|Df\|_{L^{\varsigma(\cdot),\varrho(\cdot)}(\Omega_{\mathrm{T}})},
\end{eqnarray*}
\begin{eqnarray*}
\left(\text{resp. }\, \|f\|_{W^{2,\varsigma(\cdot),\varrho(\cdot)}(\Omega_{\mathrm{T}})}=\|f\|_{W^{1,\varsigma(\cdot),\varrho(\cdot)}(\Omega_{\mathrm{T}})}+\|f_{t}\|_{L^{\varsigma(\cdot),\varrho(\cdot)}(\Omega_{\mathrm{T}})}+\|D^{2}f\|_{L^{\varsigma(\cdot),\varrho(\cdot)}(\Omega_{\mathrm{T}})}\right).
\end{eqnarray*}
\end{definition}

To obtain regularity estimates in these spaces, we assume that \(\varsigma\) is \textit{log-H\"{o}lder continuous}. That is, there exists a constant \(\mathrm{C}_{\varsigma} > 0\) such that
\begin{eqnarray*}
|\varsigma(x,t)-\varsigma(y,s)|\leq\frac{-\mathrm{C}_{\varsigma}}{\log d_{p}((x,t),(y,s))}, 
\end{eqnarray*}
for all \((x,t),(y,s)\in \Omega_{\mathrm{T}}\) with \(0<d_{p}((x,t),(y,s))\leq \frac{1}{2}\). 

\begin{remark}\label{remark}
A necessary and sufficient condition for \(\varsigma\) to be log-H\"{o}lder continuous is the existence of a modulus of continuity \(\pi: [0,\infty) \to [0,\infty)\) satisfying
\begin{eqnarray*}
\sup_{0<r<\frac{1}{2}}\left[\pi(r)\log \left(\frac{1}{r}\right)\right]\leq \mathrm{C}_{\varsigma}.
\end{eqnarray*}
\end{remark}

Our second main result is stated in the theorem below:

\vspace{0.4cm}

\begin{theorem}[{\bf Global Variable Exponent Morrey Regularity}]\label{T2}
Let \(\Omega \subset \mathbb{R}^{n}\) be a bounded domain with \(\partial \Omega \in C^{2,\alpha}\) for some \(\alpha \in (0,1)\), and let \(\mathrm{T} > 0\). Assume that the structural hypotheses \(\mathrm{(H1)}\), \(\mathrm{(H3)}\), and \(\mathrm{(H5)}\) hold, with \(\beta \in C^{1,\alpha}(\mathrm{S}_{\mathrm{T}})\), \(\gamma = g = 0\), and \(f \in L^{\varsigma(\cdot),\varrho(\cdot)}(\Omega_{\mathrm{T}})\). Let \(u\) be an \(L^{\varsigma_{1}}\)-viscosity solution of \eqref{1.1}. Then, \(u \in W^{2,\varsigma(\cdot),\varrho(\cdot)}(\Omega_{\mathrm{T}})\), and the following estimate holds:
\begin{eqnarray*}
\|u\|_{W^{2,\varsigma(\cdot),\varrho(\cdot)}(\Omega_{\mathrm{T}})}\leq\mathrm{C}\|f\|_{L^{\varsigma(\cdot),\varrho(\cdot)}(\Omega_{\mathrm{T}})},
\end{eqnarray*} 
where \(\mathrm{C}>0\) depends only on the parameters \(n\), \(\mathrm{T}\), \(\lambda\), \(\Lambda\), \(\xi\), \(\sigma\), \(\mu_{0}\), \(\varsigma_{1}\), \(\varsigma_{2}\), \(\varrho_{0}\), \(\mathrm{c}_{1}\), \(\mathrm{c}_{2}\), \(\theta_0\), \(\|\beta\|_{C^{1,\alpha}(\mathrm{S}_{\mathrm{T}})}\), and \(\|\partial \Omega\|_{C^{2,\alpha}}\).
\end{theorem}

We emphasize that, although our manuscript is strongly influenced by recent developments in~\cite{BessaOrlicz},~\cite{BessaRicarte}, and~\cite{BHLp}, our approach necessitates several nontrivial adaptations due to the presence of a non-homogeneous oblique boundary condition in the parabolic setting and the asymptotic behavior of the limiting operator. Furthermore, in contrast to~\cite{BHLp} and~\cite{ZZZ21}, our results yield additional quantitative applications, including density results for oblique-type problems via tangential methods, global \( W^{2,\Upsilon} \)-regularity for obstacle-type problems under oblique boundary conditions, BMO-type estimates (see Section~\ref{Section5}), and global Morrey estimates with variable exponents (see Section~\ref{Section6}). Additionally, our recession profiles (see ~\ref{Reces}) encompass a broader class of operators than the linear ones considered in~\cite{ZZZ21} (see also~\cite{BJ0}). In particular, within the oblique boundary framework, our results extend previous contributions from~\cite{CP},~\cite{DM}, and~\cite{Goffi24}, and, to some extent, those from~\cite{BJ0},~\cite{daST},~\cite{Kry83}, and~\cite{ZZZ21}, by employing techniques specifically tailored to fully nonlinear models with relaxed convexity assumptions and nonstandard boundary data.

Additionally, even in the linear setting, it is worth emphasizing that our results are remarkable within the current literature on Orlicz-Sobolev estimates with oblique boundary conditions, since, to the best of our knowledge, only results for problems with Dirichlet boundary conditions are available (see \cite{ByunLee15} and \cite{Yao16} for related contributions). 

Furthermore, our results should be regarded as a natural extension of the interior estimates established in the previous works \cite{BHLp}, \cite{CP}, \cite{ZZF}, and \cite{ZZZ21} in the parabolic setting, as well as in \cite{Bessa}, \cite{BLOK}, \cite{BJ}, \cite{BJ0}, and \cite{Lee} for the corresponding elliptic estimates.

\subsection*{Structure of manuscript}

Our paper is organized as follows: Section~\ref{Section2} provides preliminary material on viscosity solutions and the functional spaces discussed above, which underpin the subsequent analysis. In Section~\ref{Section3}, we introduce the main contributions of our approach, namely a caloric approximation result and decay estimates for the sets where the Hessian fails to be controlled, emphasizing the derivation of estimates from the original model \( F \) and their correspondence to those associated with the limiting profile \( F^{\sharp} \). Section~\ref{Section4} is devoted to the proof of Theorem~\ref{T1}, relying on the tools developed in the preceding sections. In Section~\ref{Section5}, we present consequences of Theorem~\ref{T1}, including density results in \( L^{\Upsilon}_{\omega} \) spaces, identification of the fundamental class for model~\eqref{1.1}, global weighted Orlicz–Sobolev estimates for the obstacle problem governed by fully nonlinear parabolic models with oblique boundary conditions, as well as Orlicz-type weighted estimates for the Hessian and time derivative. Finally, Section~\ref{Section6} contains the proof of Theorem~\ref{T2}, which illustrates the applicability of our gradient estimates in the context of H\"{o}lder spaces with variable exponents.


\subsection{State-of-the-Art: regularity theory for nonlinear parabolic models}

In this part, we review some relevant contributions from the existing literature that pertain to our problem. 

Let \( \mathrm{Q}_{R}(z_0) \subset \Omega_{\mathrm{T}} \) denote a parabolic cylinder centered at \( z_0 = (x_0, t_0) \) with radius \( R >0\), i.e.,
\[
\mathrm{Q}_{R}(z_0) := \mathrm{B}_{R}(x_0) \times (t_0 - R^2, t_0).
\]

In the modern regularity theory of second-order parabolic partial differential equations, one of the fundamental \textit{a priori} estimates is the nowadays well-known Calder\'{o}n–Zygmund estimate (see \cite{CZ} for the original estimates in the elliptic scenario).

Its local estimates state that if $u$ solves the heat equation in a parabolic cylinder $\mathrm{Q}_R(z_0)$, then under appropriate conditions on $f$ (the estimates require $f \in L^p(\Omega_{\mathrm{T}})$ for some $1 < p < \infty$ - the case $p=2$ is classical and follows from energy methods), the solution gains regularity.

\begin{theorem}[{\bf Local $L^p$ Estimates - \cite[Theorem 7.22]{Lieberman1996} and \cite[Chapter 4]{Kry08}}]
Let $u \in W^{1,2}_p(\mathrm{Q}_R(z_0))$ be a strong solution of 
$$
\partial_t u - \Delta u = f(x, t) \quad \text{in} \quad  \mathrm{Q}_R(z_0).
$$
Then, for any $1 < p < \infty$, there exists a constant $\mathrm{C} = \mathrm{C}(n,p)$ such that
\[
\| \partial_t u \|_{L^p(\mathrm{Q}_{R/2}(z_0))} + \| D^2 u \|_{L^p(\mathrm{Q}_{R/2}(z_0))} \leq \mathrm{C} \left( \| f \|_{L^p(\mathrm{Q}_R(z_0))} + \| u \|_{L^p(\mathrm{Q}_R(z_0))} \right).
\]
\end{theorem}

Additionally, global estimates require boundary conditions and regularity of the domain (the global estimate assumes a $C^{1,1}$ boundary regularity - for less regular boundaries, e.g., Lipschitz, the results may fail). For instance, for the Dirichlet problem:
\[
\begin{cases}
\partial_t u - \Delta u = f & \text{in} \quad \Omega_{\mathrm{T}}, \\
u = 0 & \text{on} \quad \partial_p \Omega_{\mathrm{T}},
\end{cases}
\]
where $\partial_p \Omega_{\mathrm{T}}$ is the parabolic boundary, we have the following result.

\begin{theorem}[{\bf Global $L^p$ Estimates - \cite[Corollary 7.31]{Lieberman1996}}]
Let $\Omega$ be a bounded domain with $C^{1,1}$ boundary, $f \in L^p(\Omega_{\mathrm{T}})$, and $u \in W^{1,2}_p(\Omega_{\mathrm{T}})$ solve the Dirichlet problem. Then, for $1 < p < \infty$,
\[
\| \partial_t u \|_{L^p(\Omega_{\mathrm{T}})} + \| D^2 u \|_{L^p(\Omega_T)} \leq \mathrm{C} \| f \|_{L^p(\Omega_{\mathrm{T}})},
\]
where $\mathrm{C} = \mathrm{C}(n,p,\Omega,T)$.
\end{theorem}

It is well-established that this estimate continues to hold when the heat operator is replaced by any constant-coefficient, second-order parabolic operator, or even by a second-order parabolic operator whose principal part has continuous coefficients and whose lower-order coefficients lie in appropriate Lebesgue spaces (see, for example, \cite{Lieberman1996}, which covers both local and global estimates in detail for parabolic equations in divergence and non-divergence form).

Concerning problem \eqref{1.1}, the body of literature addressing estimates for \(D^{2}u\) and \(u_{t}\) is not as extensive as in the elliptic case. This is largely due to the inherent difficulties in analyzing parabolic equations, where the geometry of the domains involved is more intricate than in the elliptic setting, and the regularity theory must treat the temporal derivative \(u_{t}\) separately from the spatial derivatives \(u_{x_{i}}\). Nevertheless, the regularity theory for viscosity solutions of fully nonlinear parabolic equations has become an active and important area of research. Significant advances have been made in this direction. For instance, Lieberman in \cite{Lieb01} established H\"{o}lder continuity of the gradient for the homogeneous version of problem \eqref{1.1}, under suitable regularity assumptions on the operator \(F\). 

Subsequently, Nazarov and Ural’tseva in \cite{NazUral92} developed \(C^{1,\alpha}\) regularity results for solutions of the following quasilinear parabolic problem:
\begin{equation*}
\left\{
\begin{array}{rclcl}
\displaystyle u_{t} - \sum_{i,j=1}^{n} a_{ij}(Du,u,x,t)D_{ij}u &=& f(Du,u,x,t) & \text{in} & \Omega_\mathrm{T}, \\
\beta \cdot Du + u &=& g(x,t) & \text{on} & \mathrm{S}_{\mathrm{T}}, \\
u(x, 0) &=& u_{0}(x) & \text{on} & \Omega,
\end{array}
\right.
\end{equation*}
which involves a class of quasilinear operators with nonlinear first-order terms and non-degenerate oblique boundary conditions.

In this direction, for oblique problems with non-divergence form operators, Lieberman in \cite{Lieberman1996} established Sobolev estimates for the problem
\begin{equation*}
\left\{
\begin{array}{rclcl}
\displaystyle u_{t} - \sum_{i,j=1}^{n} a_{ij}(x,t)D_{ij}u +\sum_{i=1}^{n}b^{i}(x,t)D_{i}u+c(x,t)u(x,t)&=& f(x,t) & \text{in} & \Omega\subset\mathbb{R}^{n+1}, \\
\beta \cdot Du + \beta_{0}u &=& 0 & \text{on} & \mathrm{S}\subset \partial\Omega,
\end{array}
\right.
\end{equation*}
where \(\Omega\) is a bounded domain with the boundary portion \(\mathrm{S}\in C^{1,\alpha}\), \(f\in L^{p}(\Omega)\) and \(\beta,\beta_{0}\in C^{0,\alpha}(\mathrm{S})\) such that \(p\) and \(\alpha\) satisfy the following compatibility condition \(p(1-\alpha)<1\). Here, the coefficients \(a_{ij},b^{i},c\) are bounded, \((a_{ij})_{ij}\)  is a matrix uniformly parabolic with modulus of continuity. 

Subsequently, Softova, in \cite{Sof}, proved \(W^{2,p}\) regularity for the problem
\begin{equation*}
\left\{
\begin{array}{rclcl}
\displaystyle u_{t} - \sum_{i,j=1}^{n} a_{ij}(x,t)D_{ij}u &=& f(x,t) & \text{in} &\Omega_\mathrm{T}, \\
\beta \cdot Du &=& \varphi(x,t) & \text{on} & \mathrm{S}_{\mathrm{T}}, \\
u(x, 0) &=& \psi(x) & \text{on} & \Omega,
\end{array}
\right.
\end{equation*}
where the coefficients \(a_{ij}\) belong to VMO (vanishing mean oscillation) spaces, and the functions \(\psi\) and \(\varphi\) are elements of suitable Besov spaces.

Within the framework of \(W^{2,p}\) estimates, we also emphasize the contributions by Zhang et al.~\cite{ZZZ21} and Byun and Han~\cite{BHLp}. These authors established \(W^{2,p}\) regularity for problem~\eqref{1.1} under the conditions \(\gamma = g = 0\), assuming that the operator \(F\) satisfies an asymptotic convexity condition in~\cite{ZZZ21}, and exhibits a convex structural form in~\cite{BHLp}.

Still within the context of regularity theory, it is also worth highlighting the work of Chatzigeorgiou and Milakis~\cite{Ch}, who investigated H\"{o}lder continuity estimates of the forms \(C^\alpha\), \(C^{1,\alpha}\), and \(C^{2,\alpha}\) for operators with constant coefficients.

Additionally, we must quote \cite{CP}, where the authors derive sharp Sobolev estimates for solutions to fully nonlinear parabolic equations under minimal and asymptotic assumptions on the governing operator. Specifically, they prove that solutions belong to the Sobolev space \( W^{2,1;p}_{\text{loc}} \). Their approach proceeds by transferring improved regularity from a limiting configuration. In this setting, they make use of the recession profile associated with \( F \). This framework enables them to impose structural conditions exclusively on the original operator and in its asymptotic behavior at \( \text{Sym}(n) \). As its elliptic counterpart (see \cite{PT}), the regularity is governed by the asymptotic behavior of \( F \) at the ends of such space.

In the elliptic setting, the authors in~\cite{Bessa} developed \(W^{2,p}\) regularity for the problem
\begin{eqnarray}\label{problemaw2p}
\left\{
\begin{array}{rclcl}
F(D^2u,Du,u,x) &=& f(x) & \text{in} & \Omega, \\
\beta \cdot Du + \gamma u &=& g(x) & \text{on} & \partial\Omega,
\end{array}
\right.
\end{eqnarray}
for \(f \in L^p(\Omega)\). Under suitable assumptions on the data \(\beta, \gamma, g\), and the domain \(\Omega \subset \mathbb{R}^n\), the authors employed an asymptotic methodology based on the recession operator associated with the original second-order operator, leveraging compactness and stability arguments. Consequently, several applications were derived, including BMO-type estimates for the Hessian, density results for solutions, and global regularity estimates for the corresponding obstacle problem.

Last, but not least, more recently, Bessa extended these estimates to the setting of weighted Orlicz spaces in~\cite{BessaOrlicz}. Specifically, under the same asymptotic conditions on the governing operator \( F \) and the source term, in the setting of weighted Orlicz spaces, the author obtains weighted Orlicz-Sobolev estimates for viscosity solutions to problem \eqref{problemaw2p}. In particular, regularity results were established for the obstacle problem, along with Orlicz-BMO estimates for the Hessian \(D^2u\). In a complementary development, Bessa and Ricarte obtained analogous regularity results for solutions to problem \eqref{problemaw2p} in~\cite{BessaRicarte} within the framework of weighted Lorentz spaces, thereby establishing Morrey-type estimates with variable exponents as a consequence.

\section{Preliminaries} \label{Section2}

\hspace{0.4cm} In this section, we introduce some fundamental definitions and properties related to H\"{o}lder spaces and weighted Orlicz spaces. We conclude by recalling a few standard results and definitions concerning viscosity solutions of fully nonlinear parabolic equations with oblique boundary conditions.

\subsection{Some Basic Functional Spaces}

We begin by introducing some fundamental functional spaces that are essential for the development of this work. Throughout the following, we assume that $U \subset \mathbb{R}^{n} \times \mathbb{R}$ denotes a bounded domain.

\begin{definition}
Let $C(U)$ denote the space of continuous functions defined on $U$. We define the space $C^{1}(U)$ (respectively, $C^{2}(U)$) as the set of functions $u \in C(U)$ (resp. $u \in C^{1}(U)$) such that $u_{x_{i}} \in C(U)$ for all $1 \leq i \leq n$ (resp. $u_{t}$ and $u_{x_{i}x_{j}} \in C(U)$ for all $1 \leq i,j \leq n$). These spaces are equipped with the following norms:
\begin{eqnarray*}
\|u\|_{C^{1}(U)} &=& \|u\|_{L^{\infty}(U)} + \|Du\|_{L^{\infty}(U)}, \\
\|u\|_{C^{2}(U)} &=& \|u\|_{C^{1}(U)} + \|u_{t}\|_{L^{\infty}(U)} + \|D^{2}u\|_{L^{\infty}(U)}.
\end{eqnarray*}
\end{definition}

We now recall the definition of H\"{o}lder spaces.

\begin{definition}
A function $u$ defined on $U \subset \mathbb{R}^{n} \times \mathbb{R}$ is said to be {\emph{$\alpha$-H\"{o}lder continuous}} for some $\alpha \in (0,1]$ (\emph{Lipschitz continuous} when $\alpha = 1$) if
\begin{eqnarray*}
[u]_{\alpha,U} = \sup_{\substack{(x,t),(y,s) \in U\\ (x,t) \neq (y,s)}} \frac{|u(x,t) - u(y,s)|}{d_{p}((x,t),(y,s))^{\alpha}} < \infty.
\end{eqnarray*}
The space of such functions is denoted by $C^{0,\alpha}(U)$, equipped with the norm
\begin{eqnarray*}
\|u\|_{C^{0,\alpha}(U)} = \|u\|_{L^{\infty}(U)} + [u]_{\alpha,U}.
\end{eqnarray*}
\end{definition}

We also introduce the notion of the \textit{temporal $\alpha$-H\"{o}lder seminorm}, defined as
\begin{eqnarray*}
[u]_{\alpha,U;t} = \sup_{\substack{(x,t),(x,s) \in U\\ t \neq s}} \frac{|u(x,t) - u(x,s)|}{|t-s|^{\alpha}}.
\end{eqnarray*}

\begin{definition}
We define the space $C^{1,\alpha}(U)$ (respectively, $C^{2,\alpha}(U)$) as the set of functions $u \in C^{1}(U)$ (resp. $u \in C^{2}(U)$) such that
\begin{eqnarray*}
\|u\|_{C^{1,\alpha}(U)} &=& \|u\|_{C^{1}(U)} + [u]_{\frac{1+\alpha}{2},U;t} + \sum_{i=1}^{n} [u_{x_{i}}]_{\alpha,U} < \infty, \\
\|u\|_{C^{2,\alpha}(U)} &=& \|u\|_{C^{2}(U)} + \sum_{i=1}^{n} [u_{x_{i}}]_{\frac{1+\alpha}{2},U;t} + [u_{t}]_{\alpha,U} + \sum_{i,j=1}^{n} [u_{x_{i}x_{j}}]_{\alpha,U} < \infty.
\end{eqnarray*}
\end{definition}

Next, we present several additional properties of weights and weighted Orlicz spaces that will be employed throughout the remainder of this work. We begin with a result concerning weights, the proof of which can be found in the book by Kokilashvili and Krbec~\cite{KokiMiro}.

\begin{lemma}\label{Strongdoubling}
Let $\omega$ be an $\mathfrak{A}_{p}$ weight for some $1<p<\infty$. Then:
\begin{itemize}
\item[(a)] ({\bf Monotonicity}) If $p' \ge p$, then $\omega$ belongs to the class $\mathfrak{A}_{p'}$ and satisfies $[\omega]_{p'} \leq [\omega]_{p}$.
\item[(b)] ({\bf Strong Doubling Property}) There exist positive constants $\kappa_{1}$ and $\theta_0$, depending only on $n$, $p$, and $[\omega]_{p}$, such that
\begin{eqnarray*}
\frac{1}{[\omega]_{p}}\left(\frac{|\mathrm{E}|}{|\Omega|}\right)^{p}\leq\frac{\omega(\mathrm{E})}{\omega(\Omega)}\leq \kappa_{1}\left(\frac{|\mathrm{E}|}{|\Omega|}\right)^{\theta_0},
\end{eqnarray*}
for all Lebesgue measurable sets $\mathrm{E}\subset \Omega$.
\end{itemize}
\end{lemma}

Next, we recall the definition of the Hardy-Littlewood maximal operator. For $f \in L^1_{\textrm{loc}}(\mathbb{R}^{n+1})$, the \textit{Hardy-Littlewood maximal operator} is defined by
$$
\mathcal{M}(f)(x,t) \defeq \sup_{\rho>0} \intav{\mathrm{Q}_{\rho}(x,t)} |f(y,s)| \, dyds.
$$

We will later use the following weighted version of the classical Hardy-Littlewood-Wiener theorem for Orlicz spaces (cf.~\cite[Theorem 2.1.1]{KokiMiro}).

\begin{lemma}\label{maximalorlicz}
Let $\Phi$ be an $\mathrm{N}$-function satisfying $\Phi\in \Delta_{2} \cap \nabla_{2}$, and let $\omega\in \mathfrak{A}_{i(\Phi)}$. Then, for all $g \in L^{\Phi}_{\omega}(\mathbb{R}^{n+1})$, we have
\begin{eqnarray*}
\rho_{\Phi,\omega}(g)\leq \rho_{\Phi,\omega}\left(\mathcal{M}(g)\right)\leq \mathrm{C}\rho_{\Phi,\omega}(g),
\end{eqnarray*}
where the constant $\mathrm{C}>0$ is independent of $g$.
\end{lemma}

The next result provides a sufficient condition for the Hessian and the temporal derivative (in the sense of distributions) to belong to a weighted Orlicz space (cf.~\cite[Lemma 3.4]{BLOK} and \cite[Proposition 1.1]{CC}).

\begin{lemma}\label{caracterizationofhessian}
Let $\Phi$ be an $\mathrm{N}$-function satisfying the $\Delta_{2} \cap \nabla_{2}$ condition, and let $\omega\in \mathfrak{A}_{i(\Phi)}$ be a weight. Suppose $u \in C^{0}(U)$ for a bounded domain $U \subset \mathbb{R}^{n+1}$, and define, for $r>0$,
\begin{eqnarray*}
\Theta(u,r)(x,t)\defeq  \Theta(u,\mathrm{Q}_{r}(x,t)\cap U)(x,t), \quad (x,t)\in U.
\end{eqnarray*}
If $\Theta(u,r) \in L^{\Phi}_{\omega}(U)$, then the Hessian $D^{2}u$ and the time derivative $u_{t}$ belong to $L^{\Phi}_{\omega}(U)$, and we have the estimate
\begin{eqnarray*}
\|u_{t}\|_{L^{\Phi}_{\omega}(U)}+\|D^{2}u\|_{L^{\Phi}_{\omega}(U)}\leq 9\|\Theta(u,r)\|_{L^{\Phi}_{\omega}(U)}.
\end{eqnarray*}
\end{lemma}

We will also require a characterization of functions in weighted Orlicz spaces via their distribution functions concerning the weight. The proof relies on standard arguments from measure theory (cf.~\cite[Lemma 4.6]{BOKPS}).

\begin{proposition}\label{caracterizacaodosespacosdeorliczcompeso}
Let $\Phi \in \Delta_{2} \cap \nabla_{2}$ be an $\mathrm{N}$-function, and let $\omega$ be an $\mathfrak{A}_{s}$-weight for some $s\in(1,\infty)$. Let $g: U \to \mathbb{R}$ be a nonnegative measurable function on a bounded domain $U \subset \mathbb{R}^{n+1}$. Given constants $\eta>0$ and $\mathrm{M}>1$, we have:
$$
g \in L^{\Phi}_{\omega}(U) \Longleftrightarrow  \sum_{j=1}^{\infty} \Phi(\mathrm{M}^{j}) \omega(\{(x,t)\in U : g(x,t)>\eta \mathrm{M}^{j}\}) \defeq \mathscr{S}< \infty,
$$
and, moreover,
\begin{eqnarray*}
\mathrm{C}^{-1} \mathscr{S} \le \rho_{\Phi,\omega}(g) \le \mathrm{C}(\omega(U) + \mathscr{S}),
\end{eqnarray*}
where $\mathrm{C} = \mathrm{C}(\eta, M, \Phi, \omega)$ is a positive constant.
\end{proposition}

\subsection{Key tools in viscosity solutions theory} 
Now, we introduce the appropriate notion of viscosity solutions to equation~\eqref{1.1}. For simplicity, we adopt the following notation: $\Omega_{\mathrm{T}} = \Omega \times (0,\mathrm{T})$ and $\Gamma_{\mathrm{I}} = \Gamma \times \mathrm{I}$, where $\Gamma \subset \partial \Omega$ is a relatively open subset and $\mathrm{I}$ is a fixed interval in $(0,\mathrm{T})$.

\begin{definition}[{\bf $C^{0}$-viscosity solutions}]\label{VSC_0}
Let $F$ be a $(\lambda, \Lambda, \sigma, \xi)$-parabolic operator and $\Gamma \subset \partial \Omega$ a relatively open set. A function $u \in C^{0}(\Omega_{\mathrm{T}} \cup \Gamma_{\mathrm{I}})$ is called a \emph{$C^{0}$-viscosity solution} if the following conditions are satisfied:
\begin{enumerate}
    \item[a)] For every $\varphi \in C^{2}(\Omega_{\mathrm{T}} \cup \Gamma_{\mathrm{I}})$ that touches $u$ from above at a point $(x_0, t_0) \in \Omega_{\mathrm{T}} \cup \Gamma_{\mathrm{I}}$,
    $$
    \left\{
    \begin{array}{rcl}
    F\left(D^2 \varphi(x_0,t_0), D \varphi(x_0,t_0), \varphi(x_0,t_0), x_0, t_0\right) - \varphi_{t}(x_0,t_0) &\ge& f(x_0,t_0), \quad \text{if } (x_0, t_0) \in \Omega_{\mathrm{T}}, \\
    \beta(x_0,t_0) \cdot D\varphi(x_0,t_0)+\gamma(x_{0},t_{0})\varphi(x_{0},t_{0}) &\ge& g(x_0,t_0), \quad \text{if } (x_0, t_0) \in \Gamma_{\mathrm{I}}.
    \end{array}
    \right.
    $$

    \item[b)] For every $\varphi \in C^{2}(\Omega_{\mathrm{T}} \cup \Gamma_{\mathrm{I}})$ that touches $u$ from below at a point $(x_0, t_0) \in \Omega_{\mathrm{T}} \cup \Gamma_{\mathrm{I}}$,
    $$
    \left\{
    \begin{array}{rcl}
    F\left(D^2 \varphi(x_0,t_0), D \varphi(x_0,t_0), \varphi(x_0,t_0), x_0, t_0\right) - \varphi_{t}(x_0,t_0) &\le& f(x_0,t_0), \quad \text{if } (x_0, t_0) \in \Omega_{\mathrm{T}}, \\
    \beta(x_0,t_0) \cdot D\varphi(x_0,t_0) +\gamma(x_{0},t_{0})\varphi(x_{0},t_{0})&\le& g(x_0,t_0), \quad \text{if } (x_0, t_0) \in \Gamma_{\mathrm{I}}.
    \end{array}
    \right.
    $$
\end{enumerate}
\end{definition}

\begin{remark}
If we replace the test functions in Definition~\ref{VSC_0} by functions in the Sobolev space $W^{2,1,p}$ and assume $f \in L^{p}(\Omega_{\mathrm{T}})$ for some $p > \frac{n+1}{2}$, the corresponding solution is referred to as an \emph{$L^{p}$-viscosity solution}.
\end{remark}

For notational convenience, we define
$$
\mathcal{L}^{\pm}(u) \defeq\mathcal{M}^{\pm}_{\lambda,\Lambda}(D^2 u) \pm \sigma |Du| - \partial_t u.
$$

\begin{definition}\label{calssefundamental}
We define the parabolic fundamental classes $\overline{\mathcal{S}}_{p}(\lambda, \Lambda, \sigma, f)$ and $\underline{\mathcal{S}}_{p}(\lambda, \Lambda, \sigma, f)$ as the sets of continuous functions $u$ satisfying $\mathcal{L}^{+}(u) \ge f$ and $\mathcal{L}^{-}(u) \le f$, respectively, in the viscosity sense (see Definition~\ref{VSC_0}). 

We further define:
\[
\mathcal{S}_{p}(\lambda, \Lambda, \sigma, f) \defeq \overline{\mathcal{S}}_{p}(\lambda, \Lambda, \sigma, f) \cap \underline{\mathcal{S}}_{p}(\lambda, \Lambda, \sigma, f),
\]
\[
\mathcal{S}_{p}^{\star}(\lambda, \Lambda, \sigma, f) \defeq \overline{\mathcal{S}}_{p}(\lambda, \Lambda, \sigma, |f|) \cap \underline{\mathcal{S}}_{p}(\lambda, \Lambda, \sigma, -|f|).
\]

Moreover, when $\sigma = 0$, we denote $\mathcal{S}_{p}^{\star}(\lambda, \Lambda, 0, f)$ simply by $\mathcal{S}_{p}^{\star}(\lambda, \Lambda, f)$ (respectively, $\underline{\mathcal{S}}_{p}, \overline{\mathcal{S}}_{p}, \mathcal{S}_{p}$).
\end{definition}

Below, we introduce the concept of concave/convex paraboloids, which will be of significant utility in the subsequent analysis of the Hessian of functions.

\begin{definition}
We say that $\mathrm{P}_{\mathrm{M}}$ is a paraboloid with \textit{opening} $\mathrm{M}>0$ if 
$$
\mathrm{P}_{\mathrm{M}}(x,t)= \pm \frac{\mathrm{M}}{2} (|x|^2 - t) + p_1 \cdot x + p_0,
$$
where $p_1 \in \mathbb{R}^n$ and $p_0 \in \mathbb{R}$. The paraboloid is said to be convex in the case of the ``+'' sign and concave otherwise.
\end{definition}

Now, for $u \in C^{0}(U)$, $U^{\prime} \subset \overline{U}$, and $\mathrm{M} > 0$, we define
$$
\underline{\mathrm{G}}_{\mathrm{M}}(u,U^{\prime}) \defeq \left\{(x_0,t_{0}) \in U^{\prime} \;:\; \exists \, \mathrm{P}_{\mathrm{M}} \text{ such that } \mathrm{P}_{\mathrm{M}}(x_0,t_{0}) = u(x_0,t_{0}) \text{ and } \mathrm{P}_{\mathrm{M}}(x,t) \le u(x,t)\,\, \forall (x,t) \in U^{\prime} \right\}
$$
and
$$
\underline{\mathrm{A}}_{\mathrm{M}}(u,U^{\prime}) \defeq U^{\prime} \setminus \underline{\mathrm{G}}_{\mathrm{M}}(u,U^{\prime}).
$$

Analogously, using convex paraboloids, we define $\overline{\mathrm{G}}_{\mathrm{M}}(u,U^{\prime})$ and $\overline{\mathrm{A}}_{\mathrm{M}}(u,U^{\prime})$, and set
$$
\mathrm{G}_{\mathrm{M}}(u,U^{\prime}) \defeq  \underline{\mathrm{G}}_{\mathrm{M}}(u,U^{\prime}) \cap \overline{\mathrm{G}}_{\mathrm{M}}(u,U^{\prime}) \quad \text{and} \quad \mathrm{A}_{\mathrm{M}}(u,U^{\prime}) \defeq \underline{\mathrm{A}}_{\mathrm{M}}(u,U^{\prime}) \cap \overline{\mathrm{A}}_{\mathrm{M}}(u,U^{\prime}).
$$

Associated with the sets $\overline{\mathrm{G}}_{\mathrm{M}}$, we define the following function:
$$
\overline{\Theta}(u,U^{\prime}, x,t) \defeq \inf\left\{\mathrm{M} > 0 \;:\; (x,t) \in \overline{\mathrm{G}}_{\mathrm{M}}(u,U^{\prime})\right\}.
$$
Similarly, one defines $\underline{\Theta}(u,U^{\prime}, x,t)$. Finally, we define
$$
\Theta(u,U^{\prime}, x,t) \defeq \sup\left\{\underline{\Theta}(u,U^{\prime}, x,t), \overline{\Theta}(u,U^{\prime}, x,t)\right\}.
$$

\begin{remark}
For further properties concerning the fundamental classes of viscosity solutions and paraboloids, we refer the reader to \cite{Imbert;Silvestre} and \cite{CC}.
\end{remark}

Additionally, we require a Maximum Principle for parabolic models with oblique tangential derivatives, stated as follows:

\begin{theorem}[{\bf A.B.P.T. Maximum Principle}]\label{ABP-fullversion}
Let $u \in C^{0}(\overline{\mathrm{Q}^{+}_{1}})$ satisfy
\begin{equation*}
\left\{
\begin{array}{rclcl}
u \in \mathcal{S}_{p}(\lambda,\Lambda,f) &\text{in}& \mathrm{Q}^{+}_{1}, \\
\beta \cdot Du + \gamma u = g(x,t) &\text{on}& \mathrm{Q}^{*}_{1}.
\end{array}
\right.
\end{equation*}
Assume that $\gamma \leq 0$ on $\mathrm{Q}^{*}_{1}$, and that there exists $\varsigma \in \mathrm{Q}^{*}_{1}$ such that $\beta \cdot \varsigma \geq \mu_0$ in $\mathrm{Q}^{*}_{1}$. Then,
\begin{eqnarray*}
\|u\|_{L^{\infty}(\mathrm{Q}^{+}_{1})} \leq \|u\|_{L^{\infty}(\partial_{p} \mathrm{Q}^{+}_{1} \setminus \mathrm{Q}^{*}_{1})} + C\left( \| f\|_{L^{n+1}(\mathrm{Q}^{+}_{1})} + \| g\|_{L^{\infty}(\mathrm{Q}^{*}_{1})} \right),
\end{eqnarray*}
where $C > 0$ depends only on $n$, $\lambda$, $\Lambda$, and $\mu_0$.
\end{theorem}

\begin{proof}
This result follows directly from \cite[Theorem 2.5]{Geo.Mil} (see also \cite[Theorem 2.1]{LiZhang}).
\end{proof}

With this more general version of the A.B.P.T. Maximum Principle, and following arguments similar to those in \cite[Theorem 3.1]{Geo.Mil}, we obtain local H\"{o}lder continuity for solutions of the class of models described above. This is the content of the following theorem:

\begin{theorem}[{\bf H\"{o}lder Regularity}]\label{Holder_Est}
Let $u \in C^{0}(\mathrm{Q}^{+}_{1} \cup \mathrm{Q}^{*}_{1})$ be a viscosity solution satisfying
\[
\left\{
\begin{array}{rclcl}
 u \in \mathcal{S}_{p}(\lambda, \Lambda, f) & \text{in} & \mathrm{Q}^{+}_{1}, \\
 \beta \cdot Du + \gamma u = g(x,t) & \text{on} & \mathrm{Q}^{*}_{1}.
\end{array}
\right.
\]
Then $u \in C^{0, \alpha^{\prime}}(\overline{\mathrm{Q}^{+}_{\frac{1}{2}}})$, and
\[
\|u\|_{C^{0, \alpha^{\prime}}(\overline{\mathrm{Q}^{+}_{\frac{1}{2}}})} \le \mathrm{C}(n, \lambda, \Lambda, \mu_0) \left( \|u\|_{L^{\infty}(\mathrm{Q}^{+}_{1})} + \|f\|_{L^{n+1}(\mathrm{Q}^{+}_{1})} + \|g\|_{L^{\infty}(\mathrm{Q}^{*}_{1})} \right),
\]
where $\alpha^{\prime} \in (0,1)$ depends only on $n$, $\lambda$, $\Lambda$, and $\mu_0$.
\end{theorem}

Next, we state the following stability result, whose proof follows along the same lines as in \cite[Theorem 3.8]{CCKS}.

\begin{lemma}[{\bf Stability Lemma}]\label{Est}
Consider $\{\Omega_k\}_{k \in \mathbb{N}}$ be an increasing sequence of open sets in $\mathbb{R}^{n} \times \mathbb{R}$ such that $\Omega_k \subset \Omega_{k+1}$ and define $\Omega \defeq \bigcup_{k=1}^{\infty} \Omega_k$. Let $p \ge n+1$ and suppose $F$, $F_k$ are $(\lambda, \Lambda, \sigma, \xi)$-parabolic  operators. Assume $f \in L^{p}(\Omega)$, $f_k \in L^p(\Omega_k)$, and let $u_k \in C^0(\Omega_k)$ be $L^{p}$-viscosity subsolutions (resp. supersolutions) of 
\[
F_k(D^2 u_k, Du_k, u_k, x, t) - (u_k)_t = f_k(x,t) \quad \text{in} \quad \Omega_k.
\]
Suppose that $u_k \to u_{\infty}$ locally uniformly in $\Omega$ and that, for every parabolic cylinder $\mathrm{Q}_r(x_0,t_0) \subset \Omega$ and test function $\varphi \in W^{2,p}(\mathrm{Q}_r(x_0,t_0))$, we have
\begin{equation} \label{Est1}
\|(\hat{g} - \hat{g}_k)^+\|_{L^p(\mathrm{Q}_r(x_0,t_0))} \to 0 \quad \text{(resp. } \|(\hat{g} - \hat{g}_k)^-\|_{L^p(\mathrm{Q}_r(x_0,t_0))} \to 0\text{)},
\end{equation}
where 
\[
\hat{g}(x,t) \defeq F(D^2 \varphi, D \varphi, u, x, t) - f(x,t), \quad \hat{g}_k(x,t) \defeq F_k(D^2 \varphi, D \varphi, u_k, x, t) - f_k(x,t).
\]
Then $u_\infty$ is an $L^{p}$-viscosity subsolution (resp. supersolution) of
\[
F(D^2 u, Du, u, x, t) - u_t = f(x,t) \quad \text{in} \quad \Omega.
\]
Moreover, if $F$ and $f$ are continuous, then $u_\infty$ is also a $C^0$-viscosity subsolution (resp. supersolution), provided that condition \eqref{Est1} holds for all test functions $\varphi \in C^2(\mathrm{Q}_r(x_0,t_0))$.
\end{lemma}

We now turn our attention to the existence and uniqueness of viscosity solutions under oblique boundary conditions. To this end, we impose the following assumption on the operator $F$:

\begin{enumerate}
\item [$(\bf E)$] There exists a modulus of continuity $\tilde{\omega}$, i.e., a non-decreasing function satisfying $\displaystyle \lim_{\theta \to 0} \tilde{\omega}(\theta) = 0$, such that
\[
\psi_{F}((x,t),(y,s)) \le \tilde{\omega}(|(x,t)-(y,s)|).
\]
\end{enumerate}

 Now, we establish the existence and uniqueness of solutions to the following problem:
\begin{eqnarray}\label{problemamisto}
\left\{
\begin{array}{rclcl}
	F(D^2u, x,t)-u_{t} &=& f(x,t) & \text{in} & \mathrm{Q}^{+}_{1},\\
	\beta\cdot Du+\gamma u&=& g(x,t) & \text{on} & \mathrm{Q}^{*}_{1},\\
	u&=&\varphi & \text{on} & \partial_{p} \mathrm{Q}^{+}_{1}\setminus \mathrm{Q}^{*}_{1},
\end{array}
\right.
\end{eqnarray}
where we employ the techniques developed in \cite{Bessa}. The proof of the next result follows the strategy of \cite[Theorem 2.7]{Bessa} with only minor adjustments. For this reason, we omit it here.

\begin{theorem}[{\bf Comparison Principle}]\label{comparation}
Assume that $\beta\in C^{2}(\mathrm{Q}^{*}_{1})$ and that $F$ satisfies assumptions \(\mathrm{(H1)}\) and $(\bf E)$. Let $u$ and $v$ be functions such that
$$
\left\{
\begin{array}{rclcl}
F(D^2u, x,t)-u_{t} &\geq& f_{1}(x,t) & \text{in} & \mathrm{Q}^{+}_{1},\\
\beta\cdot Du+\gamma u&\geq& g_{1}(x,t) & \text{on} & \mathrm{Q}^{*}_{1},
\end{array}
\right.
$$
and
$$
\left\{
\begin{array}{rclcl}
F(D^2v, x,t)-v_{t} &\leq& f_{2}(x,t) & \text{in} & \mathrm{Q}^{+}_{1},\\
\beta\cdot Dv+\gamma v&\leq& g_{2}(x,t) & \text{on} & \mathrm{Q}^{*}_{1}.
\end{array}
\right.
$$
Then,
$$
\left\{
\begin{array}{rclcl}
u-v\in \underline{\mathcal{S}}\left(\frac{\lambda}{n},\Lambda,f_{1}-f_{2}\right) & \text{in} & \mathrm{Q}^{+}_{1},\\
\beta \cdot D(u-v)+\gamma(u-v)\geq (g_{1}-g_{2})(x,t) & \text{on} & \mathrm{Q}^{*}_{1}.
\end{array}
\right.
$$
\end{theorem}

By combining Theorem~\ref{comparation} with the A.B.P.T. estimate~\ref{ABP-fullversion}, we obtain the following existence and uniqueness result for problem~\eqref{problemamisto}. The proof proceeds analogously to that in \cite{Bessa}, with minor adaptations.

\begin{theorem}[{\bf Existence and Uniqueness}]\label{Existencia}
Let $\beta\in C^{2}(\mathrm{Q}^{*}_{1})$ and $\varphi\in C^{0}(\partial_{p} \mathrm{Q}^{+}_{1}\setminus \mathrm{Q}^{*}_{1})$, and assume that $F$ satisfies condition $(\bf E)$. Suppose that there exists a vector field $\varsigma \in \mathrm{Q}^{*}_{1}$ such that $\beta\cdot \varsigma\ge \mu_0$ on $\mathrm{Q}^{*}_{1}$. Then, there exists a unique viscosity solution to problem~\eqref{problemamisto}.
\end{theorem}

\section{Caloric Approximation and Decay of the Sets \(\mathrm{A}_{t}\)} \label{Section3}

\hspace{0.4cm} In this section, we present a key tool that plays a central role in establishing the decay of the measure of the sets $\mathrm{A}_{t}$—where the ``Hessian and the temporal derivative behave poorly'' —with respect to powers of $t$. This tool is the \textit{Caloric Approximation Lemma}, which ensures that if our equation is sufficiently close to the homogeneous equation with constant coefficients, then the corresponding solution is also close to that of the homogeneous equation with frozen coefficients.

The following result is fundamental in our tangential approximation strategy. Specifically, it characterizes how the ``modulus of convergence'' of $F_{\tau}$ to $F^{\sharp}$ behaves.

\begin{lemma}\label{lemma3.1}
Let $F$ be a uniformly parabolic operator and assume that $F^{\sharp}$ exists. Then, given $\epsilon > 0$, there exists a constant $\tau_0 = \tau_0(\lambda, \Lambda, \epsilon, \psi_{F^{\sharp}}) > 0$ such that, for every $\tau \in (0, \tau_0)$, the following inequality holds:
\[
\frac{\left|\tau F \left(\tau^{-1} \mathrm{X}, 0, 0, x,t\right) - F^{\sharp}(\mathrm{X}, 0, 0, x,t) \right|}{1+\|\mathrm{X}\|} \le \epsilon,
\]
for every $\mathrm{X} \in \text{Sym}(n)$.
\end{lemma}

\begin{proof}
The proof of this lemma follows the same reasoning as that in \cite{ST} (see also \cite{CP} and \cite{PT}).
\end{proof}

\begin{remark}
It is important to highlight that the weighted Orlicz estimates in Theorem~\ref{T1} depend not only on universal constants, but also on the ``modulus of convergence'' $F_{\tau} \to F^{\sharp}$. More precisely, for \(\delta>0\) define
\[
\displaystyle  \rho(\delta) \defeq \sup_{0<\tau\leq\delta}\sup_{\mathrm{X} \in \text{Sym}(n) } \frac{\left|\tau F\left(\tau^{-1} \mathrm{X}, 0, 0, x,t\right) - F^{\sharp}(\mathrm{X}, 0, 0, x,t) \right|}{1+\|\mathrm{X}\|}.
\]
 By the Lemma \ref{lemma3.1} we have that \(\rho(\delta)\searrow 0\) as \(\delta\searrow 0\). Consequently, there exists \(\delta(\epsilon)>0\) such that \(\rho(\delta(\epsilon))\leq \epsilon\) (e.g., one may take \(\delta(\epsilon)=\tau_{0}\) in the Lemma \ref{lemma3.1}). Thus, the constant $\mathrm{C} > 0$ appearing in the global estimate of Theorem~\ref{T1} also depends on the function $\rho$.
\end{remark}

In what follows, the recession operator and the path $\tau \mapsto F_{\tau}$ are inserted in an approximating regime. For translating these ideas into a precise statement, we formulate the following result (which, at this stage, is independent of the regularity assumptions \(\mathrm{(H4)-(H5)}\))

\begin{lemma}[{\bf Caloric Approximation Lemma}] \label{Approx}
Let $n+1 \le p < \infty$, $0 \le \nu \le 1$, and assume that conditions \(\mathrm{(H1)-(H3)}\) are satisfied. Given $\delta > 0$, let $\varphi \in C^{0}(\partial_{p}\mathrm{Q}^{\nu}_{r}(0',\nu,0))$  with $\|\varphi\|_{L^{\infty}(\partial_{p}\mathrm{Q}^{\nu}_{r}(0',\nu,0))} \le \mathrm{C}_{1}$ for some $\mathrm{C}_{1}>0$, and let $g \in C^{0,\alpha}(\overline{\mathrm{Q}}^{*}_{2r})$ with $0 < \alpha < 1$ and $\|g\|_{C^{0,\alpha}(\overline{\mathrm{Q}}^{*}_{2r})} \le \mathrm{C}_{2}$ for some $\mathrm{C}_{2} > 0$. Then, there exist positive constants $\epsilon = \epsilon(\delta,n, \mu_0, p, \lambda, \Lambda, \mathrm{C}_{1}, \mathrm{C}_{2}) < 1$ and $\tau_0 = \tau_0(\delta, n, \lambda, \Lambda, \mu_0, \mathrm{C}_{1}, \mathrm{C}_{2}) > 0$ such that, if
\[
\max\left\{ \left|F_{\tau}(\mathrm{X},x,t) - F^{\sharp}(\mathrm{X},x,t)\right|, \, \|\psi_{F^{\sharp}}\|_{L^{p}(\mathrm{Q}_{2r}^{\nu}(0',\nu,0))},\,\|f\|_{L^{p}(\mathrm{Q}_{2r}^{\nu}(0',\nu,0))}  \right\} \le \epsilon \quad \text{and} \quad \tau \le \tau_0,
\]
then any two $L^p$-viscosity solutions $u$ (normalized so that $\|u\|_{L^{\infty}(\mathrm{Q}_{r}^{\nu}(0',\nu,0))} \leq 1$) and $\mathrm{h}$ of the problems
\[
\left\{
\begin{array}{rclcl}
F_{\tau}(D^2u,x,t) - u_{t} &=& f(x,t) & \text{in} & \mathrm{Q}^{\nu}_{r}(0',\nu,0), \\
\beta \cdot Du + \gamma u &=& g(x,t) & \text{on} & \mathrm{Q}_{r}^{*}, \\
u &=& \varphi & \text{on} & \partial_{p}\mathrm{Q}_{r}^{\nu}(0',\nu,0)\setminus \mathrm{Q}_{r}^{*}
\end{array}
\right.
\]
and
\[
\left\{
\begin{array}{rclcl}
F^{\sharp}(D^2 \mathrm{h},0,0) - \mathrm{h}_{t} &=& 0 & \text{in} & \mathrm{Q}^{\nu}_{\frac{3}{4}r}(0',\nu,0), \\
\beta \cdot D\mathrm{h} + \gamma \mathrm{h} &=& g(x,t) & \text{on} & \mathrm{Q}_{\frac{3}{4}r}^{*}, \\
\mathrm{h} &=& u & \text{on} & \partial_{p}\mathrm{Q}_{\frac{3}{4}r}^{\nu}(0',\nu,0)\setminus \mathrm{Q}_{\frac{3}{4}r}^{*}
\end{array}
\right.
\]
satisfy the estimate
\[
\|u - \mathrm{h}\|_{L^{\infty}(\mathrm{Q}_{\frac{3}{4}r}^{\nu}(0',\nu,0))} \le \delta.
\]
\end{lemma}

\begin{proof}
Without loss of generality, we assume that \( r = 1 \). We will prove the lemma by contradiction. Suppose the claim does not hold. Then, there exist \( \delta_0 > 0 \) and a sequence of functions \( (F_{\tau_j})_{j \in \mathbb{N}} \), \( (F^{\sharp}_j)_{j \in \mathbb{N}} \), \( (u_j)_{j \in \mathbb{N}} \), \( (f_j)_{j \in \mathbb{N}} \), \( (\varphi_j)_{j \in \mathbb{N}} \), \( (g_j)_{j \in \mathbb{N}} \), and \( (\mathrm{h}_j)_{j \in \mathbb{N}} \) related by the following system of equations:
$$
\left\{
\begin{array}{rclcl}
F_{\tau_j}(D^2 u_j, x, t) - (u_j)_t &=& f_j(x, t) & \text{in} & \mathrm{Q}^{\nu_j}_1(0', \nu_j, 0) \\
\beta \cdot D u_j + \gamma u_j &=& g_j(x, t) & \text{on} & \mathrm{Q}_1^* \\
u_j &=& \varphi_j & \text{on} & \partial_p \mathrm{Q}_1^{\nu_j}(0', \nu_j, 0) \setminus \mathrm{Q}_1^*
\end{array}
\right.
$$
and
$$
\left\{
\begin{array}{rclcl}
F^{\sharp}(D^2 \mathrm{h}_j, 0, 0) - (\mathrm{h}_j)_t &=& 0 & \text{in} & \mathrm{Q}^{\nu_j}_{\frac{3}{4}}(0', \nu_j, 0) \\
\beta \cdot D \mathrm{h}_j + \gamma \mathrm{h}_j &=& g_j(x, t) & \text{on} & \mathrm{Q}_{\frac{3}{4}}^* \\
\mathrm{h}_j &=& u_j & \text{on} & \partial_p \mathrm{Q}_{\frac{3}{4}}^{\nu_j}(0', \nu_j, 0) \setminus \mathrm{Q}_{\frac{3}{4}}^*
\end{array}
\right.
$$
where \( \tau_j \), \( \Vert \psi_{F_{\tau_j}^{\sharp}} \Vert_{L^p(\mathrm{Q}^{\nu_j}_2(0', \nu_j, 0))} \), and \( \|f_j\|_{L^p(\mathrm{Q}^{\nu_j}_2(0', \nu_j, 0))} \) tend to zero as \( j \to \infty \), and such that
\begin{equation} \label{1}
\|u_j - \mathrm{h}_j\|_{L^{\infty}(\mathrm{Q}^{\nu_j}_{\frac{3}{4}}(0', \nu_j, 0))} > \delta_0.
\end{equation}
Moreover, \( \varphi_j \in C^0(\partial_p \mathrm{Q}^{\nu_j}_1(0', \nu_j, 0)) \) and \( g_j \in C^{0, \alpha}(\overline{\mathrm{Q}^*}_2) \) satisfy \( \|\varphi_j\|_{L^{\infty}(\partial_p \mathrm{Q}^{\nu_j}_1(0', \nu_j, 0))} \leq \mathrm{C}_1 \) and \( \|g_j\|_{C^{0, \alpha}(\overline{\mathrm{Q}^*}_2)} \leq \mathrm{C}_2 \), respectively. From Theorem \ref{Holder_Est}, we have for all \( 0 < \rho < 1 \),
\begin{equation} \label{4.6}
\|u_j\|_{C^{0, \alpha'}(\mathrm{Q}^{\nu_j}_{1, \rho}(0', \nu_j, 0))} \leq \mathrm{C}(n, \lambda, \Lambda, \mathrm{C}_1, \mathrm{C}_2, \mu_0) \rho^{-\alpha'}
\end{equation}
for some \( \alpha' = \alpha'(n, \lambda, \Lambda, \mu_0) \in (0, 1) \). Suppose that there exists a number \( \nu_{\infty} \) and a subsequence \( \{\nu_{j_k}\} \) such that \( \nu_{j_k} \to \nu_{\infty} \) as \( k \to +\infty \). We can assume that such a subsequence is monotone. If \( \nu_{j_k} \) is decreasing, we can check that
$$
\mathrm{Q}^{\nu_{\infty}}_{\frac{7}{8}}(0', \nu_{\infty}, 0) \subset \mathrm{Q}^{\nu_{j_k}}_{\frac{7}{8}}(0', \nu_{j_k}, 0), \quad \forall k \in \mathbb{N}.
$$
Thus, by \eqref{4.6}, we observe that
\begin{equation} \label{4.7}
\|u_{j_k}\|_{C^{0, \alpha'}(\mathrm{Q}^{\nu_{\infty}}_{\frac{7}{8}}(0', \nu_{\infty}, 0))} \leq \mathrm{C}(n, \lambda, \Lambda, \mathrm{C}_1, \mathrm{C}_2, \mu_0).
\end{equation}
On the other hand, if \( \nu_{j_k} \) is increasing, there exists a number \( k_0 \) such that
$$
\mathrm{Q}^{\nu_{\infty}}_{\frac{7}{8}}(0', \nu_{\infty}, 0) \subset \mathrm{Q}^{\nu_{j_k}}_{\frac{15}{16}}(0', \nu_{j_k}, 0), \quad \text{for all} \, k \geq k_0.
$$
Then, once again by \eqref{4.6}, the estimate \eqref{4.7} is valid when \( (\nu_{j_k}) \) is increasing. Thus, we can apply the Arzelà-Ascoli compactness criterion, and there exist functions \( u_{\infty} \in C^{0, \alpha}(\mathrm{Q}^{\nu_{\infty}}_{\frac{7}{8}}(0', \nu_{\infty}, 0)) \), \( g_{\infty} \in C^{0, \alpha}(\overline{\mathrm{Q}^*}_1) \), and subsequences such that \( u_{j_k} \to u_{\infty} \) in \( C^0(\mathrm{Q}^{\nu_{\infty}}_{\frac{7}{8}}(0', \nu_{\infty}, 0)) \) and \( g_{j_k} \to g_{\infty} \) in \( C^{0, \alpha}(\overline{\mathrm{Q}^*}_1) \).

Since the functions \( F^{\sharp}_j(\cdot, 0, 0) \to F^{\sharp}_{\infty}(\cdot, 0, 0) \) uniformly in compact sets of \( \textit{Sym}(n) \), and for every \( \varphi \in C^2(\mathrm{Q}_r(x_0, t_0)) \) such that \( \mathrm{Q}_r(x_0, t_0) \subset \mathrm{Q}^{\nu_{\infty}}_{\frac{7}{8}}(0', \nu_{\infty}, 0) \), we have
\begin{eqnarray*}
|F_{\tau_{j_k}}(D^2 \varphi, x, t) - f_{j_k}(x, t) - F^{\sharp}_{\infty}(D^2 \varphi, 0, 0)| &\leq& |F_{\tau_{j_k}}(D^2 \varphi, x, t) - F^{\sharp}_{j_k}(D^2 \varphi, x, t)| + |f_{j_k}| + \\
& & |F^{\sharp}_{j_k}(D^2 \varphi, x, t) - F^{\sharp}_{j_k}(D^2 \varphi, 0, 0)| + \\
& & |F^{\sharp}_{j_k}(D^2 \varphi, 0, 0) - F^{\sharp}_{\infty}(D^2 \varphi, 0, 0)| \\
&\leq& |F_{\tau_{j_k}}(D^2 \varphi, x, t) - F^{\sharp}_{j_k}(D^2 \varphi, x, t)| + |f_{j_k}| + \\
& & \psi_{F_{\tau_{j_k}}^{\sharp}}((x, t), (0, 0))(1 + |D^2 \varphi|).
\end{eqnarray*}
Thus,
$$
\lim_{k \to +\infty} \| F_{\tau_{j_k}}(D^2 \varphi, x, t) - f_{j_k}(x, t) - F^{\sharp}_{\infty}(D^2 \varphi, 0, 0) \|_{L^p(\mathrm{Q}_r(x_0, t_0))} = 0.
$$
Therefore, the Stability Lemma \ref{Est} ensures that \( u_{\infty} \) satisfies in the viscosity sense
$$
\left\{
\begin{array}{rclcl}
F^{\sharp}_{\infty}(D^2 u_{\infty}, 0, 0) - (u_{\infty})_t &=& 0 & \text{in} & \mathrm{Q}^{\nu_{\infty}}_{\frac{7}{8}}(0', \nu_{\infty}, 0) \\
\beta \cdot D u_{\infty} + \gamma u_{\infty} &=& g_{\infty}(x, t) & \text{on} & \mathrm{Q}_{\frac{7}{8}}^*.
\end{array}
\right.
$$
Now, define \( w_{j_k} := u_{\infty} - \mathrm{h}_{j_k} \) for each \( k \). We observe that \( w_{j_k} \) satisfies by Theorem \ref{comparation}
$$
\left\{
\begin{array}{rclcl}
w_{j_k} &\in& \mathcal{S} \left( \frac{\lambda}{n}, \Lambda, 0 \right) & \text{in} & \mathrm{Q}^{\nu_{\infty}}_{\frac{3}{4}}(0', \nu_{\infty}, 0) \\
\beta \cdot D w_{j_k} + \gamma w_{j_k} &=& (g_{\infty} - g_{j_k})(x, t) & \text{on} & \mathrm{Q}_{\frac{3}{4}}^* \\
w_{j_k} &=& u_{\infty} - u_{j_k} & \text{on} & \partial_p \mathrm{Q}^{\nu_{\infty}}_{\frac{3}{4}}(0', \nu_{\infty}, 0) \setminus \mathrm{Q}_{\frac{3}{4}}^*.
\end{array}
\right.
$$
Thus, by Lemma \ref{ABP-fullversion}, we observe that
\begin{eqnarray*}
\|w_{j_k}\|_{L^{\infty}(\mathrm{Q}^{\nu_{\infty}}_{\frac{3}{4}}(0', \nu_{\infty}, 0))} &\leq& \|u_{\infty} - u_{j_k}\|_{L^{\infty}(\partial_p \mathrm{Q}^{\nu_{\infty}}_{\frac{3}{4}}(0', \nu_{\infty}, 0) \setminus \mathrm{Q}_{\frac{3}{4}}^*)} + \\
& & \mathrm{C}(n, \lambda, \Lambda, \mu_0) \|g_{\infty} - g_{j_k}\|_{L^{\infty}(\mathrm{Q}_{\frac{3}{4}}^*)} \to 0 \quad \text{as} \,\, k \to +\infty.
\end{eqnarray*}
Thus, \( w_{j_k} \) converges uniformly to zero. This implies that \( \mathrm{h}_{j_k} \) converges uniformly to \( u_{\infty} \) in \( \mathrm{Q}^{\nu_{\infty}}_{\frac{3}{4}}(0', \nu_{\infty}, 0) \), which contradicts \eqref{1} for \( k \gg 1 \).
\end{proof}

Following the well-established ideas from classical literature, our goal is to ensure, via such an approximation, a decay in the measure of the sets $\mathrm{A}_{t}$ as powers of $t$. These sets are characterized by those points where the Hessian and the temporal derivative are "bad" (in a suitable measure sense).  In this context, the next result ensures a decay in the fundamental class of solutions, where both the solution and the source term are small, which is a classical result. Therefore, the proof is omitted. For details, we recommend that the interested reader see \cite[Corollary 3.10]{BHLp} and \cite[Proposition 3.3]{Bessa}.

\begin{proposition}[{\bf Power Decay on the Boundary}]\label{Prop2.12}
Let $\Omega=\mathrm{B}^{+}_{12\sqrt{n}}\times (0,13]$, $r\in(0,1]$, and $(x_{0},t_{0})\in \mathrm{B}^{+}_{12\sqrt{n}}\times (0,13]$ such that $r\Omega(x_{0},t_{0})=\mathrm{B}^{+}_{12r\sqrt{n}}\times (t_{0},t_{0}+13r^{2}]\subset \Omega$. Assume that $u\in \mathcal{S}_{p}^{*}(\lambda,\Lambda,f)$ in $r\Omega(x_{0},t_{0})$, $u\in C^0(\Omega)$, and $\Vert u\Vert_{L^{\infty}(\Omega)}\leq 1$. Then, there exist universal constants $\mathrm{C}>0$ and $\delta>0$ such that if $\Vert f\Vert_{L^{n+1}(r\Omega(x_{0},t_{0}))}\le 1$, it follows that
$$
|\mathrm{A}_{s}(u,\Omega)\cap ((K^{n-1}_{r}\times (0,r)\times (0,r^{2}))+(x_{1},t_{1})))|\leq \mathrm{C}s^{-\mu}|K^{n-1}\times (0,r)\times (0,r^{2})|
$$
for any $(x_{1},t_{1}) \in (\mathrm{B}_{9\sqrt{n}}(x_{0}) \cap \overline{\mathbb{R}^n_+})\times [t_{0},t_{0}+5r^{2}]$ and $s>1$.
\end{proposition}

For clarity, we note that the forthcoming results will rely on the assumptions \(\mathrm{(H4)}\) and/or \(\mathrm{(H5)}\). In what follows, these hypotheses will be tacitly assumed whenever required, and their specific role will be made explicit in the corresponding proofs.

In light of this decay estimate, we will utilize the convergence module to examine the behavior of solutions associated with the operators of the continuous path $\tau\mapsto F_{\tau}$. 

\begin{proposition}\label{Prop.4.6}
Assume that the structural conditions \(\mathrm{(H1)-(H5)}\) are satisfied. Let $\Omega^{\ast}=\mathrm{B}^{+}_{14\sqrt{n}}\times(0,15]$, $0<r\leq1$, and let $u$ be a viscosity solution of
$$
\left\{
\begin{array}{rclcl}
F_{\tau}(D^2u,x,t)-u_{t}&=& f(x,t) & \mbox{in} & \Omega^{\ast},\\
\beta\cdot Du+\gamma u&=& g(x,t) & \mbox{on} & \mathrm{S}=: \mathrm{T}_{14\sqrt{n}}\times (0,15]
\end{array}
\right.
$$
Assume further that $\max \left\{ \|f\|_{L^{n+1}(\Omega)}, \,\,\tau \right\} \le \epsilon$ for some $0<\epsilon<1$ and consider a point $(x_{0},t_{0})\in \mathrm{S}$ such that $r\Omega(x_{0},t_{0})\subset \Omega$. Finally, assume that
$$
\mathrm{G}_1(u,\Omega^{\ast}) \cap \left((K^{n-1}_{3r} \times (0,3r)\times (r^{2},10r^{2})) + (\tilde{x}_1, \tilde{t}_{1})\right) \neq \emptyset
$$
for some $(\tilde{x}_1,\tilde{t}_{1}) \in (\mathrm{B}_{9r\sqrt{n}}(x_{0})\cap \{x_{n}\geq 0\})\times [t_{0}+2r^{2},t_{0}+5r^{2}]$. Then,
$$
|\mathrm{G}_{\mathrm{M}}(u,\Omega^{\ast}) \cap  \left((K^{n-1}_{r} \times (0,r)\times (0,r^{2})) + (x_1,t_{1})\right)| \ge (1-\epsilon_0)|K^{n-1}_{r}\times (0,r)\times (0,r^{2})|,
$$
where $(x_1,t_{1}) \in (\mathrm{B}_{9r\sqrt{n}}(x_{0})\cap\{x_{n}\geq 0\})\times[t_{0}+2r^{2},\tilde{t}_{1}]$, and $\mathrm{M}>1$ depends only on $n$, $\lambda$, $\Lambda$, $\mu_0$, $\|g\|_{C^{1,\alpha}(\mathrm{S})}$, and $\mathfrak{c}_{1}$, and $\epsilon_0 \in (0,1)$.
\end{proposition}

\begin{proof}
Consider a point $(x_{2},t_{2})$ in the set $\mathrm{G}_1(u,\Omega^{\ast}) \cap \left(K^{n-1}_{3r} \times (0,3r)\times (r^{2},10r^{2}) + (\tilde{x}_1, \tilde{t}_{1})\right)$. In particular, since $(x_{2},t_{2}) \in \mathrm{G}_1(u,\Omega^{\ast})$, there exist paraboloids with opening $s=1$ that touch $u$ at $(x_2,t_2)$ from above and below; that is,
\[
-\frac{1}{2}(|x-x_2|^2-(t-t_2)) \le u(x,t)-l(x) \le \frac{1}{2}(|x-x_2|^2-(t-t_2)),
\]
for every $(x,t) \in \mathrm{B}^{+}_{14\sqrt{n}}\times (0,t_2)$ and some affine function $l$.

Now, define
\[
v(x,t) = \frac{u(x,t)-l(x)}{\mathrm{C}_{\ast}},
\]
where $\mathrm{C}_{\ast}>0$ is a dimensional constant chosen so that $\|v\|_{L^{\infty}(\mathrm{B}^{+}_{14\sqrt{n}}(x_0)\times (t_0,t_2))} \le 1$, and
\[
-(|x|^2-(t-t_2)) \le v(x,t) \le |x|^2-(t-t_2) \quad \text{in} \quad \mathrm{B}^{+}_{12\sqrt{n}}(x_0)\times [0,t_2].
\]

Next, observe that $v$ is a viscosity solution to
\[
\left\{
\begin{array}{rclcl}
\tilde{F}_{\tau}(D^2v,x,t)-v_{t} &=& \tilde{f}(x,t) & \text{in} & r\Omega(x_0,t_0),\\
\beta\cdot Dv+\gamma v &=& \tilde{g}(x,t) & \text{on} & r\mathrm{S}(x_0,t_0),
\end{array}
\right.
\]
where
\[
\left\{
\begin{array}{rcl}
\tilde{F}_{\tau}(\mathrm{X},x,t) &\defeq& \frac{1}{\mathrm{C}_{\ast}}F_{\tau}(\mathrm{C}_{\ast} \mathrm{X}, x,t),\\
\tilde{f}(x,t) &\defeq& \frac{1}{\mathrm{C}_{\ast}} f(x,t),\\
\tilde{g}(x,t) &\defeq& \frac{1}{\mathrm{C}_{\ast}}(g(x,t)-\beta(x,t)\cdot Dl(x)-\gamma(x,t) l(x)).
\end{array}
\right.
\]

Now, we denote $\Omega' =\mathrm{B}^{+}_{14\sqrt{n}}\times (1,15]$ and $\mathrm{S}^{\prime} = \mathrm{T}^{+}_{14\sqrt{n}}\times (1,15]$. By hypothesis (H5), there exists a function $\mathrm{h} \in C^{2,\alpha}(\overline{r\Omega'(x_0,t_0)})$, $\varepsilon_0$-close to $v$, arising from the caloric approximation Lemma~\ref{Approx}, such that $\mathrm{h}$ solves in the viscosity sense:
\[
\left\{
\begin{array}{rclcl}
\tilde{F}^{\sharp}(D^2 \mathrm{h},0,0)-\mathrm{h}_{t} &=& 0 & \text{in} & r\Omega'(x_0,t_0),\\
\beta\cdot D\mathrm{h}+\gamma \mathrm{h} &=& \tilde{g}(x,t) & \text{on} & r\mathrm{S}^{\prime}(x_0,t_0),\\
\mathrm{h} &=& v & \text{on} & \partial_{p}(r\Omega'(x_0,t_0)) \setminus r\mathrm{S}^{\prime}(x_0,t_0),
\end{array}
\right.
\]
with
\[
\|v-\mathrm{h}\|_{L^{\infty}(r\Omega'(x_0,t_0))} \le \varepsilon,
\]
for some $\varepsilon_0$ to be chosen later. Note that $\varepsilon < 1$ is determined by the choice of $\varepsilon_0$.

Since $\beta \cdot Dl \in C^{1,\alpha}(rS'(x_0,t_0))$ (as $\beta \in C^{1,\alpha}(r\mathrm{S}^{\prime}(x_0,t_0))$ and $Dl$ is constant), the A.B.P.T. Maximum Principle (Lemma~\ref{ABP-fullversion}) yields:
\begin{eqnarray*}
\|\mathrm{h}\|_{L^{\infty}(r\Omega'(x_0,t_0))} 
&\le& \|v\|_{L^{\infty}(\partial_{p}(r\Omega'(x_0,t_0)) \setminus r\mathrm{S}^{\prime}(x_0,t_0))} \\
& +&\frac{\mathrm{C}}{\mathrm{C}_{\ast}}\left[\|g\|_{L^{\infty}(S)} + |Dl| \|\beta\|_{L^{\infty}(S)} + \|\gamma l\|_{L^{\infty}(S)}\right]\\
&\le& \mathrm{C}(n, \|l\|_{L^{\infty}(S)}, \|\gamma\|_{C^{1,\alpha}(S)}, \|g\|_{C^{1,\alpha}(S)}) \\
&=:& \widetilde{\mathrm{C}}.
\end{eqnarray*}

Thus, hypothesis (H5) ensures that in the set $\Omega''(x_0,t_0) = \mathrm{B}^{+}_{13\sqrt{n}}\times (2,15]$ we have
\[
\|\mathrm{h}\|_{C^{2,\alpha}(r\Omega''(x_0,t_0))} \le \mathrm{C}(\mathrm{c}_1, \widetilde{\mathrm{C}}),
\]
and consequently,
\[
\mathrm{A}_{\mathrm{N}}\left(\mathrm{h}, r\Omega'(x_0,t_0)\right) \cap \left((\mathrm{Q}^{n-1}_{r} \times (0,r)\times(0,r^2)) + (x_1,t_1)\right) = \emptyset
\]
for any $(x_1,t_1)\in (\mathrm{B}_{9r\sqrt{n}} \cap \overline{\mathbb{R}^{n}_{+}}) \times [t_0+2r^2, \tilde{t}_1]$ and some $\mathrm{N} = \mathrm{N}(\mathrm{c}_1, \widetilde{C}) \gg 1$.

Next, consider the set $\Omega''' = \mathrm{B}^{+}_{12\sqrt{n}}\times(2,15]$, and define $\tilde{\mathrm{h}}$ as a continuous extension of $\mathrm{h}_{|_{r\Omega''(x_0,t_0)}}$ to the sub-cylinder $\Omega_{t_2}$, where for a set $V \subset \mathbb{R}^n \times \mathbb{R}$ and $s \in \mathbb{R}$, we write $V_s := \{(x,t) \in V : t \le s\}$. Let $\tilde{\mathrm{h}} = v$ in $\Omega_{t_2} \setminus (r\Omega'(x_0,t_0))_{t_2}$, so that
\[
\|v - \tilde{\mathrm{h}}\|_{L^{\infty}(\Omega_{t_2})} = \|v - \mathrm{h}\|_{L^{\infty}((r\Omega''(x_0,t_0))_{t_2})} \le \overline{\mathrm{C}}(\mathrm{c}_1, \widetilde{\mathrm{C}}).
\]
Moreover, in $\Omega_{t_2} \setminus (r\Omega''(x_0,t_0))_{t_2}$, it holds that
\[
-(\overline{\mathrm{C}}(\mathrm{c}_1, \widetilde{\mathrm{C}}) + (|x|^2 - (t - t_2))) \le \tilde{\mathrm{h}}(x,t) \le \overline{\mathrm{C}}(\mathrm{c}_1, \widetilde{\mathrm{C}}) + (|x|^2 - (t - t_2)).
\]
Therefore, there exists a constant $\mathrm{M}_0 = \mathrm{M}_0(\mathfrak{c}_1, \widetilde{\mathrm{C}}) \ge \mathrm{N} > 1$ such that
\[
\mathrm{A}_{\mathrm{M}_0}(\tilde{\mathrm{h}}, \Omega) \cap \left((K^{n-1}_r \times (0,r)\times (0,r^2)) + (x_1,t_1)\right) = \emptyset.
\]
Consequently,
\begin{equation} \label{Sub}
\left((K^{n-1}_r \times (0,r)\times (0,r^2)) + (x_1,t_1)\right) \subset \mathrm{G}_{\mathrm{M}_0}(\tilde{h}, \Omega).
\end{equation}

Define now
\[
w(x,t) \defeq \frac{1}{2\mathrm{C} \varepsilon}(v - \tilde{\mathrm{h}})(x,t).
\]
Then, $w$ satisfies the assumptions of Proposition~\ref{Prop2.12}, and for $s>1$, we obtain
\[
|\mathrm{A}_s(w, \Omega^{\mathscr{t}}) \cap ((K^{n-1}_r \times (0,r)\times (0,r^2)) + (x_1,t_1))| \le \mathrm{C} s^{-\kappa} |K^{n-1}_r \times (0,r) \times (0,r^2)|.
\]

Using the inclusion $\mathrm{A}_{2\mathrm{M}_0}(u) \subset \mathrm{A}_{\mathrm{M}_0}(w) \cup \mathrm{A}_{\mathrm{M}_0}(\tilde{\mathrm{h}})$ and \eqref{Sub}, we conclude that
\[
\frac{|\mathrm{G}_{2\mathrm{M}_0}(v - \tilde{\mathrm{h}}, \Omega^{\ast}) \cap ((K^{n-1}_r \times (0,r)\times (0,r^2)) + (x_1,t_1))|}{|K^{n-1}_r \times (0,r) \times (0,r^2)|} \ge 1 - \mathrm{C} \varepsilon^{\kappa}.
\]
Finally, we conclude that
\[
\frac{|\mathrm{G}_{2\mathrm{M}_0}(v, \Omega^{\ast}) \cap ((K^{n-1}_r \times (0,r)\times (0,r^2)) + (x_1,t_1))|}{|K^{n-1}_r \times (0,r) \times (0,r^2)|} \ge 1 - \mathrm{C} \varepsilon^{\kappa}.
\]

The proof is completed by choosing $\varepsilon \ll 1$ suitably and defining $\mathrm{M} \defeq  2\mathrm{M}_0$.
\end{proof}

With Proposition \ref{Prop.4.6} and the Stacked Covering Lemma \cite[Lemma 2.4.27]{Imbert;Silvestre}, we can proceed to the discrete process of the decay of the measure of the sets $\mathrm{A}_{t}$, whose proof follows the same lines as \cite[Lemma 4.12]{BJ}, and for this reason, we omit it here.

\begin{lemma} \label{lemma3.6}
	Consider $\epsilon_0 \in (0, 1)$, $\Omega^{\ast} = \mathrm{B}^{+}_{14\sqrt{n}} \times (0,15]$, $\mathrm{S} = \mathrm{T}_{14\sqrt{n}} \times (0,15]$, and $0 < r \leq 1$. Furthermore, consider $(x_0, t_0) \in \mathrm{S}$ such that $r\Omega(x_0, t_0) \subset \Omega$. Let $u$ be a normalized viscosity solution to
	$$
	\left\{
	\begin{array}{rclcl}
		F_{\tau}(D^2u, x, t) - u_t & = & f(x,t) & \mbox{in} & \Omega^{\ast},\\
		\beta \cdot Du + \gamma u & = & g(x,t) & \mbox{on} & \mathrm{S}.
	\end{array}
	\right.
	$$
	Assume that conditions \(\mathrm{(H1)-(H5)}\) hold and extend $f$ by zero outside $r\Omega(x_0, t_0)$. Suppose that
	$$
	\max \left\{\tau, \|f\|_{L^{n+1}(\Omega^{\ast})} \right\} \leq \epsilon
	$$
	for some $\epsilon > 0$ depending only on $n, \epsilon_0, \lambda, \Lambda, \mu_0$. Then, for $k \in \mathbb{N}$, we define
	\begin{eqnarray*}
		\mathrm{A} & \defeq & \mathrm{A}_{\mathrm{M}^{k+1}}(u, \Omega^{\ast}) \cap \left(K^{n-1}_r \times (0,r) \times (t_0 + 2r^2, t_0 + 3r^2)\right),\\
		\mathrm{B} & \defeq & \left( \mathrm{A}_{\mathrm{M}^k}(u, \Omega^{\ast}) \cap \left(K^{n-1}_r \times (0,r) \times (t_0 + 2r^2, t_0 + 3r^2)\right) \right) \cup \\
		& & \left\{ (x,t) \in K^{n-1}_r \times (0,r) \times (t_0 + 2r^2, t_0 + 3r^2); \mathcal{M}(f^{n+1}) \ge (\mathrm{C}_0 \mathrm{M}^k)^{n+1} \right\},
	\end{eqnarray*}
	where $\mathrm{M} = \mathrm{M}(n, \lambda, \Lambda, \mu_0, \|\beta\|_{C^{1,\alpha}(r\mathrm{S}(x_0,t_0))}, \mathrm{C}_0) > 1$. Then,
	$$
	|\mathrm{A}| \leq 2 \epsilon_0 |\mathrm{B}|.
	$$
\end{lemma}

A consequence of this fact is that we can observe the decay of the measure associated with the weights. More precisely, the following result holds:

\begin{corollary}\label{cor4.6}
	Let $\omega \in \mathfrak{A}_p$ be a weight  for some $1 < p < \infty$. Under the same conditions as Lemma \ref{lemma3.6}, fix $\epsilon_0 \in (0, 1)$. For each integer $k \geq 0$, define
	\begin{eqnarray*}
		\mathrm{A}^k & \defeq & \mathrm{A}_{\mathrm{M}^{k+1}}(u, \Omega) \cap \left(K^{n-1}_r \times (0,r) \times (t_0 + 2r^2, t_0 + 3r^2)\right),\\
		\mathrm{B}^k & \defeq & \left\{ (x,t) \in K^{n-1}_r \times (0,r) \times (t_0 + 2r^2, t_0 + 3r^2); \mathcal{M}(f^{n+1}) \ge (\mathrm{C}_0 \mathrm{M}^k)^{n+1} \right\},
	\end{eqnarray*}
	where the constants $\mathrm{C}_0$ and $\mathrm{M}$ are the same as in Lemma \ref{lemma3.6}. Then,
	\begin{eqnarray*}
		\omega(\mathrm{A}^k) \leq \epsilon_0^k \omega(\mathrm{A}^0) + \sum_{j=1}^{k-1} \epsilon_0^{k-j} \omega(\mathrm{B}^j), \quad \text{for all} \ k \geq 0.
	\end{eqnarray*}
\end{corollary}

\begin{proof}
	Apply Lemma \ref{lemma3.6} with the constant $\tilde{\epsilon} = \frac{1}{2} \left(\frac{\epsilon_0}{\kappa_1}\right)^{\frac{1}{\theta}}$, where the positive constants $\kappa_1$ and $\theta$ are the same as in the Strong Doubling Lemma \ref{Strongdoubling}. For each $k$, we obtain $|\mathrm{A}^{k+1}| \leq 2 \tilde{\epsilon} |\mathrm{A}^k \cup \mathrm{B}^k|$. Thus, by the Strong Doubling Lemma \ref{Strongdoubling},
	\begin{eqnarray}\label{iteracao}
		\omega(\mathrm{A}^{k+1}) & \leq & \kappa_1 \left(\frac{|\mathrm{A}^{k+1}|}{|\mathrm{A}^k \cup \mathrm{B}^k|}\right)^\theta \omega(\mathrm{A}^k \cup \mathrm{B}^k) \nonumber\\
		& \leq & \kappa_1 (2 \tilde{\epsilon})^\theta (\omega(\mathrm{A}^k) + \omega(\mathrm{B}^k)) \nonumber\\
		& = & \epsilon_0 \omega(\mathrm{A}^k) + \epsilon_0 \omega(\mathrm{B}^k).
	\end{eqnarray}
	Finally, note that inequality \eqref{iteracao} holds for all $k \geq 0$, and consequently, iterating these estimates leads to the desired result.
\end{proof}

\section{Weighted Orlicz-Sobolev Estimates}\label{Section4}

In this section, we establish Theorem \ref{T1}. To this end, we first derive boundary estimates for solutions to the following flatness problem: 
\begin{equation}\label{problemaflat}
\left\{
\begin{array}{rclcl}
F(D^2u,Du,x,t)-u_{t}&=& f(x,t) & \mbox{in} & \mathrm{Q}^+_1,\\
\beta\cdot Du+\gamma u &=& g(x,t) & \mbox{on} & \mathrm{Q}_1^{*},
\end{array}
\right.
\end{equation}

Subsequently, employing standard covering arguments, we deduce a proof of Theorem \ref{T1}. To carry out this strategy, we commence by deriving Hessian and time derivative estimates for solutions of problem \eqref{problemaflat}, under the assumption that the governing operator \( F \) is independent of the lower-order terms \( u \) and \( Du \).

\begin{proposition}\label{T-flat}
Let \( \Phi \in \Delta_{2} \cap \nabla_{2} \) be an \( \mathrm{N} \)-function, and let \( f \in L^{\Upsilon}_{\omega}(\mathrm{Q}^+_1) \cap C^{0}(\mathrm{Q}^{+}_{1}) \), where \( \omega \in \mathfrak{A}_{i(\Phi)} \) is a weight and \( \Upsilon(t) = \Phi(t^{n+1}) \). Suppose \( u \) is a normalized \( C^{0} \)-viscosity solution of
\begin{equation*}
\left\{
\begin{array}{rclcl}
F(D^2u, x,t)-u_{t}&=& f(x,t) & \mbox{in} & \mathrm{Q}^+_1,\\
\beta\cdot Du+\gamma u &=& g(x,t) & \mbox{on} & \mathrm{Q}_1^{*},
\end{array}
\right.
\end{equation*}
and assume that the structural conditions \(\mathrm{(H1)-(H5)}\) hold. Then \( u_{t}, D^{2}u \in L^{\Upsilon}_{\omega}\left(\mathrm{Q}^+_{\frac{1}{2}}\right) \), and the following estimate is satisfied:
\[
\|u_{t}\|_{L^{\Upsilon}_{\omega}\left(\mathrm{Q}^+_{\frac{1}{2}}\right)}+\|D^{2}u\|_{L^{\Upsilon}_{\omega}\left(\mathrm{Q}^+_{\frac{1}{2}}\right)} \le \mathrm{C} \cdot \left( \|u\|^{n+1}_{L^{\infty}(\mathrm{Q}^+_1)} + \|f\|_{L^{\Upsilon}_{\omega}(\mathrm{Q}^+_1)} + \Vert g \Vert_{C^{1,\alpha}(\mathrm{Q}_{1}^{*})} \right),
\]
where \( \mathrm{C} = \mathrm{C}(n, \lambda, \Lambda, i(\Phi), p_{2}, \omega, \Vert \beta \Vert_{C^{1,\alpha}(\overline{\mathrm{T}_{1}})}, \Vert \gamma \Vert_{C^{1,\alpha}(\overline{\mathrm{T}_{1}})}, \alpha, r_{0}, \theta_0, \mu_{0}) > 0 \).
\end{proposition}

\begin{proof}
We begin by observing that since $\Phi$ is an $\mathrm{N}$-function in the class $\Delta_{2} \cap \nabla_{2}$, the same holds for $\Upsilon$. It is not difficult to verify that $i(\Upsilon) = (n+1)i(\Phi)$, and consequently, $\mathfrak{A}_{i(\Phi)} \subset \mathfrak{A}_{i(\Upsilon)}$. This inclusion guarantees that the embedding result stated in Lemma \ref{mergulhoorliczlebesgue}, when applied to $L^{\Upsilon}_{\omega}$, depends only on the parameters $n$, $\omega$, and $i(\Phi)$.

With this in mind, fix $(x_0, t_0) \in \mathrm{Q}_{\frac{1}{2}} \cup \mathrm{Q}^{*}_{\frac{1}{2}}$. In the case where $(x_0, t_0) \in \mathrm{Q}^{*}_{\frac{1}{2}}$, choose 
\[
0 < r < \min\left\{ \frac{1 - |x_0|}{14\sqrt{n}}, \sqrt{\frac{-t_0}{15}} \right\},
\]
and define
\[
\nu \defeq \frac{\epsilon r^{\frac{n+2}{n+1}}}{\left(\left(\frac{\epsilon}{r}\right)^{n+1}  \|u\|_{L^{\infty}(r\Omega(x_0,t_0))}^{n+1} + (\mathrm{C}')^{n+1} \|f\|_{L^{\Upsilon}_{\omega}(r\Omega(x_0,t_0))} + \left(\frac{\epsilon}{r}\right)^{n+1} \|g\|_{C^{1,\alpha}(r\mathrm{S}(x_0,t_0))} \right)^{\frac{1}{n+1}}},
\]
where $\Omega^{\ast} = \mathrm{B}^{+}_{14\sqrt{n}} \times (0,15]$ and $\mathrm{S} = \mathrm{T}_{14\sqrt{n}} \times (0,15]$. The constants $\mathrm{C}'$ and $\epsilon > 0$ are those appearing in Lemmas \ref{mergulhoorliczlebesgue} and \ref{Prop.4.6}, with $\epsilon_0 \in (0,1)$ to be determined later.

Now, define the rescaled function:
\[
\tilde{u}(y,s) \defeq \frac{\nu}{r^{2}} u(x_0 + ry, t_0 + r^2 s).
\]
Then, $\tilde{u}$ is a normalized viscosity solution to the problem
\[
\left\{
\begin{array}{rclcl}
\tilde{F}(D^2 \tilde{u}, x,t) - \tilde{u}_t &=& \tilde{f}(x,t) & \text{in} & \Omega^{\ast}, \\
\tilde{\beta} \cdot D\tilde{u} + \tilde{\gamma} \tilde{u} &=& \tilde{g}(x,t) & \text{on} & \mathrm{S},
\end{array}
\right.
\]
where
\[
\left\{
\begin{array}{rcl}
\tilde{F}(\mathrm{X}, y, s) &\defeq& \nu F\left(\frac{1}{\nu} \mathrm{X}, x_0 + ry, t_0 + r^2 s\right), \\
\tilde{f}(y,s) &\defeq& \nu f(x_0 + ry, t_0 + r^2 s), \\
\tilde{\beta}(y,s) &\defeq& \beta(x_0 + ry, t_0 + r^2 s), \\
\tilde{\gamma}(y,s) &\defeq& r \gamma(x_0 + ry, t_0 + r^2 s), \\
\tilde{g}(y,s) &\defeq& \frac{\nu}{r} g(x_0 + ry, t_0 + r^2 s), \\
\tilde{\omega}(y,s) &\defeq& \omega(x_0 + ry, t_0 + r^2 s).
\end{array}
\right.
\]

We observe that $\tilde{F}$ satisfies the structural conditions (H1)–(H5), and $\tilde{\omega} \in \mathfrak{A}_{i(\Phi)}$, since $\omega \in \mathfrak{A}_{i(\Phi)}$. Applying Lemma \ref{mergulhoorliczlebesgue} and H\"{o}lder's inequality, we estimate
\begin{equation*}
\|\tilde{f}\|_{L^{n+1}(\Omega^{\ast})} = \frac{\nu}{r^{\frac{n+2}{n+1}}} \|f\|_{L^{n+1}(r\Omega(x_0,t_0))} \leq \frac{\nu}{r^{\frac{n+2}{n+1}}} \mathrm{C}' \|f\|_{L^{\Upsilon}_{\omega}(r\Omega(x_0,t_0))}^{\frac{1}{n+1}} \leq \epsilon,
\end{equation*}
which ensures the applicability of Corollary \ref{cor4.6}. For each $k \geq 0$, define the sets
\begin{align*}
\mathrm{A}^{k} &\defeq \mathrm{A}_{\mathrm{M}^{k+1}}(\tilde{u}, \Omega^{\ast}) \cap (K^{n-1}_1 \times (0,1) \times (2,3)), \\
\mathrm{B}^{k} &\defeq \left\{ (x,t) \in K^{n-1}_1 \times (0,1) \times (2,3) \; ; \; \mathcal{M}(\tilde{f}^{n+1}) \geq (\mathrm{C}_0 \mathrm{M}^k)^{n+1} \right\}.
\end{align*}
Then,
\begin{equation} \label{estdospesos}
\tilde{\omega}(\mathrm{A}^{k}) \leq \epsilon_0^k \tilde{\omega}(\mathrm{A}^{0}) + \sum_{i=1}^{k-1} \epsilon_0^{k-i} \tilde{\omega}(\mathrm{B}^{i}).
\end{equation}

Since $\tilde{f} \in L^{\Upsilon}_{\tilde{\omega^{\ast}}}(\Omega)$—by assumption that $|f| \in L^{\Phi}_{\omega}(\mathrm{Q}^{+}_1)$ from condition (H2)—it follows that $\tilde{f}^{n+1} \in L^{\Phi}_{\tilde{\omega}}(\Omega^{\ast})$. By Lemma \ref{maximalorlicz}, we also have $\mathcal{M}(|\tilde{f}|^{n+1}) \in L^{\Phi}_{\tilde{\omega}}(\Omega^{\ast})$, and
\begin{align*}
\rho_{\Phi,\tilde{\omega}}(\mathcal{M}(|\tilde{f}|^{n+1})) &\leq \mathrm{C} \rho_{\Phi,\tilde{\omega}}(|\tilde{f}|^{n+1}) \\
&= \frac{\mathrm{C}}{r^{n+2}} \int_{r\Omega(x_0,t_0)} \Phi(\nu^{n+1} |f(y,s)|^{n+1}) \omega(y,s) \, dy\,ds \\
&\stackrel{\eqref{modularestimate}}{\leq} \frac{\mathrm{C}}{r^{n+2}} \left( \|(\nu f)^{n+1}\|^{p_2}_{L^{\Phi}_{\omega}(r\Omega(x_0,t_0))} + 1 \right) \\
&= \frac{\mathrm{C}}{r^{n+2}} \left( \nu^{(n+1)p_2} \|f\|_{L^{\Upsilon}_{\omega}(r\Omega(x_0,t_0))}^{p_2} + 1 \right) \\
&\leq \frac{\mathrm{C}}{r^{n+2}} \left( (\epsilon r^{\frac{n+2}{n+1}})^{(n+1)p_2} + 1 \right) \\
&\leq \mathrm{C}',
\end{align*}
and therefore,
\begin{equation} \label{estimativadamodulardeMF}
\|\mathcal{M}(|\tilde{f}|^{n+1})\|_{L^{\Phi}_{\tilde{\omega}}(\Omega^{\ast})} \leq \mathrm{C}'.
\end{equation}

On the other hand, by $\Phi\in \Delta_{2}$ there exists $\mathrm{k}_{1}>1$ such that $\Phi(2s)\leq \mathrm{k}_{1}\Phi(s)$ for all $s\geq0$. Now, as $\mathrm{M}>1$, there exists $m_{0}\in\mathbb{N}$ such that $\mathrm{M}^{n+1}\leq 2^{m_{0}}$, consequently, $$\Phi(\mathrm{M}^{n+1}s)\leq\Phi(2^{m_{0}}s)\leq \mathrm{K}_{0}\Phi(s), \forall s>0,$$ 
where $\mathrm{K}_{0}=\mathrm{k}_{1}^{m_{0}}$ and we use that $\Phi$ is an increasing function. By these observations, we can conclude that $\Phi(\mathrm{M}^{k(n+1)})\leq \mathrm{K}_{0}^{k}\Phi(1)$ and $\Phi(\mathrm{M}^{k(n+1)})\leq \mathrm{K}_{0}^{k-i}\Phi(\mathrm{M}^{i(n+1)})$ for all $i=1,\ldots,k-1$. Therefore, by \eqref{estdospesos} and \eqref{estimativadamodulardeMF} estimates we obtain
\begin{eqnarray} \label{distribuicao}
\sum_{k=1}^{\infty} \Upsilon(\mathrm{M}^{k})\tilde{\omega}(\mathrm{A}^{k})&=&\sum_{k=1}^{\infty} \Phi(\mathrm{M}^{k(n+1)})\tilde{\omega}(\mathrm{A}^{k})\nonumber\\
&\leq& \sum_{k=1}^{\infty} \Phi(\mathrm{M}^{k(n+1)})\epsilon_{0}^{k}\tilde{\omega}(\mathrm{A}^{0})+ \sum_{k=1}^{\infty} \Phi(\mathrm{M}^{k(n+1)})\sum_{i=1}^{k-1}\epsilon_{0}^{k-i}\tilde{\omega}(\mathrm{B}^{i})\nonumber\\
&\leq& \sum_{k=1}^{\infty}(\mathrm{K}_{0}\epsilon_{0})^{k}\Phi(1)\tilde{\omega}(\mathrm{A}^{0})+\sum_{k=1}^{\infty}\sum_{i=1}^{k-1}(\mathrm{K}_{0}\epsilon_{0})^{k-i}\Phi(\mathrm{M}^{i(n+1)})\tilde{\omega}(\mathrm{B}^{i})\nonumber\\
&\leq&\Phi(1)\tilde{\omega}(\mathrm{Q}^{+}_{1})\sum_{k=1}^{\infty}(\mathrm{K}_{0}\epsilon_{0})^{k}+\sum_{k=1}^{\infty}(\mathrm{K}_{0}\epsilon_{0})^{k})\sum_{j=1}^{\infty}\Phi(\mathrm{M}^{j(n+1)})\tilde{\omega}(\mathrm{B}^{j})\nonumber\\
&=&\sum_{k=1}^{\infty}(\mathrm{K}_{0}\epsilon_{0})^{k}\left(\tilde{\omega}(\mathrm{Q}^{+}_{1})\Phi(1)+\sum_{j=1}^{\infty}\Phi(\mathrm{M}^{j(n+1)})\tilde{\omega}(\mathrm{B}^{j})\right)\nonumber\\
&\leq&\tilde{\mathrm{C}}(\mathrm{C}',\tilde{\omega}(\mathrm{Q}^{+}_{1}),\Phi(1))\sum_{k=1}^{\infty}(\mathrm{K}_{0}\epsilon_{0})^{k}<\infty
\end{eqnarray}
for $\epsilon_{0}\ll1$ such that $\epsilon_{0}\mathrm{K}_{0}<1$. 

Taking into account the choice of $\epsilon_0$ above and recalling the inclusion 
$$\left\{(y,s)\in \mathrm{Q}^{+}_{\frac{1}{2}}(0,3)|\Theta(y,s)>c\right\}\subset \mathrm{A}_{c}(\tilde{u},\Omega^{\ast})
$$
which implies that by estimate \eqref{distribuicao} and Proposition \ref{caracterizacaodosespacosdeorliczcompeso} that $\Theta\left(\tilde{u},\mathrm{Q}^{+}_{\frac{1}{2}}(0,3)\right)\in L^{\Upsilon}_{\tilde{\omega}}\left(\mathrm{Q}^{+}_{\frac{1}{2}}(0,3)\right)$ and by Lemma \ref{caracterizationofhessian} $$\|\tilde{u}_{t}\|_{L^{\Upsilon}_{\tilde{\omega}}\left(\mathrm{Q}^{+}_{\frac{1}{2}}(0,3)\right)}+\|D^{2} \tilde{u}\|_{L^{\Upsilon}_{\tilde{\omega}}\left(\mathrm{Q}^{+}_{\frac{1}{2}}(0,3)\right)} \le \mathrm{C}$$
equivalently,
\begin{equation*} \label{(19)}
\|u_{t}\|_{L^{\Upsilon}_{\tilde{\omega}}\left(\mathrm{Q}^{+}_{\frac{r}{2}}(x_{0},t_{0}-3r^{2})\right)}+\|D^{2} u\|_{L^{\Upsilon}_{\omega}\left(\mathrm{Q}^{+}_{\frac{r}{2}}(x_{0},t_{0}-3r^{2})\right)} \le \mathrm{C}\cdot\left(\|u\|_{L^{\infty}(\mathrm{Q}^{+}_{1})}^{n+1}+ \|f\|_{L^{\Upsilon}_{\omega}(\mathrm{Q}^{+}_{1})}+\Vert g\Vert_{C^{1,\alpha}(\mathrm{Q}^{*}_{1})}\right),
\end{equation*}
where $\mathrm{C}=\mathrm{C}(n,\lambda, \Lambda, i(\Phi), p_{2},r,\mu_{0},r_{0},\theta_0,\|\beta\|_{C^{1,\alpha}(\mathrm{Q}^{*}_{1})},\|\gamma\|_{C^{1,\alpha}(\mathrm{Q}^{*}_{1})})$ is positive constant.

On the other hand, if $(x_0,t_0) \in \mathrm{Q}^{+}_{\frac{1}{2}}$, then by hypothesis (H4), we can apply the interior estimate result (cf. \cite[Proposition 4.5]{CP}; see also \cite{WangI}) to analogously obtain a decay of the measure of the sets $\mathrm{A}_{t}$ concerning the weights and proceed in an entirely analogous manner to derive interior estimates. Thus, by combining the interior and boundary estimates, and proceeding analogously to \cite[Theorem 4.1]{BHLp}, we obtain the desired results by a standard covering argument.
\end{proof}
With estimates obtained analogously to those in \cite{BessaOrlicz} and using standard density arguments, we derive weighted Orlicz-Sobolev estimates for $L^{p}$-viscosity solutions to problem \eqref{problemaflat}. This result is stated in the following Proposition. Due to the similarity with previous works, we omit its proof (cf. \cite{BessaOrlicz}, \cite{Bessa}, and \cite{BHLp} for further details).

\begin{proposition}\label{Prop4.2}
Let $u$ be a bounded $L^{p}$-viscosity solution of \eqref{problemaflat} for $p=p_{0}(n+1)$. Suppose the structural conditions \(\mathrm{(H1)-(H5)}\) and \(\mathrm{({\bf E})}\) hold. Then, there exist positive constants $c_0=c_0(n,\lambda,\Lambda,p_{0},p_{2})$ and $r_0 = r_0(n, \lambda, \Lambda, p_{0},p_{2})$ such that, if
$$
\left(\intav{\mathrm{Q}_r(x_0,t_{0}) \cap \mathrm{Q}^+_1} \psi_{F^{\sharp}}((x,t),( x_0,t_{0}))^{p} dxdt\right)^{\frac{1}{p}} \le c_0
$$
for every $(x_0,t_{0}) \in \mathrm{Q}^+_1$ and $r \in (0, r_0)$, then $u \in W^{2,\Upsilon}_{\omega}\left(\mathrm{Q}^+_{\frac{1}{2}}\right)$ and satisfies
$$
\|u\|_{W^{2,\Upsilon}_{\omega}\left(\mathrm{Q}^+_{\frac{1}{2}}\right)} \le \mathrm{C} \cdot\left( \|u\|^{n+1}_{L^{\infty}(\mathrm{Q}^{+}_{1})} +\|f\|_{L^{\Upsilon}_{\omega}(\mathrm{Q}^{+}_{1})}+\Vert g\Vert_{C^{1,\alpha}(\mathrm{Q}_{1}^{*})}\right),
$$
where $\mathrm{C}=\mathrm{C}(n,\lambda,\Lambda,\xi,\sigma,\mu_{0},p_{0},p_{2},i(\Phi),\omega,\theta_0,\|\beta\|_{C^{1,\alpha}(\mathrm{Q}_{1}^{*})},\|\gamma\|_{C^{1,\alpha}(\mathrm{Q}_{1}^{*})},r_0)>0$.
\end{proposition}
\bigskip

We are now in a position to prove Theorem \ref{T1}.

\begin{proof}[{\bf Proof of Theorem \ref{T1}}]
The strategy of the proof relies on a classical covering argument (cf. \cite{BessaOrlicz, Bessa, BHLp, Winter}). Let $(x_{0},t_{0})\in\partial\Omega\times (0,\mathrm{T})$. Since $\partial \Omega \in C^{2,\alpha}$, there exists a neighborhood $\mathcal{V}(x_0, t_0)$ of $(x_0, t_0)$ and a $C^{2,\alpha}$-diffeomorphism $\Psi: \mathcal{V}(x_0, t_0) \to \mathrm{Q}_1(0)$ such that
\[
\Psi(x_0, t_0) = 0 \quad \text{and} \quad \Psi(\Omega \cap \mathcal{V}(x_0, t_0)) = \mathrm{Q}^{+}_1.
\]
We then define $\Psi_{0}(y,s) = (\Psi(y), s - t_0)$ for $(y,s) \in \mathcal{V}_{0}(x_0, t_0)\times (t_{0}-1,t_{0})$, and set $\tilde{u} \defeq u \circ \Psi_{0}^{-1} \in C^0(\mathrm{Q}^+_1\cup \mathrm{Q}^{*}_{1})$. Observe that $\tilde{u}$ is an $L^p$-viscosity solution of the problem
\[
\left\{
\begin{array}{rclcl}
\tilde{F}(D^2 \tilde{u}, D \tilde{u}, \tilde{u}, x,t)-\tilde{u}_{t} &=& \tilde{f}(x,t) & \text{in} & \mathrm{Q}^+_1,\\
\tilde{\beta}\cdot D\tilde{u}+\tilde{\gamma}\tilde{u}& = & \tilde{g}(x,t) &\text{on} & \mathrm{Q}^{*}_1,
\end{array}
\right.
\]
where, for $(y,s)=\Psi_{0}^{-1}(x,t)$,
\[
\left\{
\begin{array}{rcl}
\tilde{F}(\mathrm{X}, \varsigma, \eta, x,t) &=& F\left(D \Psi_{0}^t(y,s) \cdot \mathrm{X} \cdot D \Psi_{0}(y,s) + \varsigma D^2\Psi_{0}, \varsigma D\Psi_{0}(y,s), \eta, y,s\right),\\
\tilde{f}(x,t) & \defeq & f(y,s),\\
\tilde{\beta}(x,t) & \defeq & \beta(y,s) \cdot (D \Psi_{0} \circ \Psi_{0}^{-1})^{t}(x,t),\\
\tilde{\gamma}(x,t) & \defeq & \gamma(y,s)(D \Psi_{0} \circ \Psi_{0}^{-1})^{t}(x,t),\\
\tilde{g}(x,t) & \defeq & g(y,s),\\
\tilde{\omega}(x,t) &=& \omega(y,s).
\end{array}
\right.
\]

Note that $\tilde{F}$ is a uniformly parabolic operator with ellipticity constants $\lambda \mathrm{C}(\Psi_{0})$ and $\Lambda \mathrm{C}(\Psi_{0})$, and that $\tilde{\omega} \in \mathcal{A}_{i(\Phi)}$, by the change-of-variable formula for the Lebesgue integral.

Thus,
\[
\tilde{F}^{\sharp}(\mathrm{X}, \varsigma, \eta, x,t) = F^{\sharp}\left(D\Psi_{0}^t(y,s) \cdot \mathrm{X}\cdot D\Psi_{0}(y,s) +  \varsigma D^2 \Psi_{0}(y,s),\varsigma D\Psi_{0}(y,s), \eta, y,s\right).
\]
Consequently, we deduce that
\[
\psi_{\tilde{F}^{\sharp}}((x,t),(x_0,t_{0})) \le \mathrm{C}(\Psi) \psi_{F^{\sharp}}((x,t),(x_0,t_{0})),
\]
ensuring that $\tilde{F}$ satisfies the hypotheses of Proposition \ref{Prop4.2}. Hence, $\tilde{u}\in W^{2,\Upsilon}_{\omega}(\mathrm{Q}^{+}_{\frac{1}{2}})$ and
\begin{eqnarray*}\label{estimativa3.6}
\|\tilde{u}\|_{W^{2,\Upsilon}_{\omega}(\mathrm{Q}^{+}_{\frac{1}{2}})}\leq \mathrm{C}(\|u\|^{n+1}_{L^{\infty}(\Omega_{\mathrm{T}})}+\|f\|_{L^{\Upsilon}_{\omega}(\Omega_{\mathrm{T}})}+\|g\|_{C^{1,\alpha}(\partial\Omega\times (0,\mathrm{T}))}).
\end{eqnarray*}
By combining the interior estimates from \cite[Theorem 6.2]{Lee} and \cite[Theorem 1.1]{CP} for every $(x_{0},t_{0})\in\Omega_{\mathrm{T}}$ and applying a standard covering argument, the proof of the theorem is completed.
\end{proof}

\begin{remark}
Regarding Theorem \ref{T1}:
\begin{itemize}
\item [\checkmark] When $\omega\equiv 1$ and $\Phi(s)=s^{p}$ with $n+1<p<\infty$, Theorem \ref{T1} can be seen as a generalization of \cite[Theorem 5.5]{BHLp}. This generalization includes both the first-order operator $\mathfrak{B}(q,r,x,t)=\beta(x,t)\cdot q+\gamma(x,t)r$ governing the oblique boundary condition, and the relaxation of the structural assumptions on $F$, which no longer requires convexity or concavity, but rather good estimates on its asymptotic profile $F^{\sharp}$.
\item[\checkmark] If the regularity of the boundary data exceeds the requirements of (H2), then more refined estimates for solutions to problem \eqref{1.1} can be obtained. Specifically, assuming $\beta,\gamma,g\in C^{2}(\partial \Omega\times (0,\mathrm{T}))$ instead of the $C^{1,\alpha}$ regularity in (H2), it follows from Theorem \ref{T1} and compactness arguments that $u\in W^{2,\Upsilon}(\Omega_{\mathrm{T}})$ with the estimate
\begin{eqnarray*}
\|u\|_{W^{2,\Upsilon}_{\omega}(\Omega_{\mathrm{T}})}\leq \mathrm{C}(\|f\|_{L^{\Upsilon}_{\omega}(\Omega_{\mathrm{T}})}+\|g\|_{C^{1,\alpha}(\partial \Omega\times (0,\mathrm{T}))}).
\end{eqnarray*}
For further discussion, see \cite[Theorem 3.5]{BessaOrlicz}.
\end{itemize}
\end{remark}

\section{Some Applications}\label{Section5}

\hspace{0.4cm} In this section, we present several consequences of the weighted Orlicz estimates established in Theorem \ref{T1}.

\subsection{Density results in a suitable class}

In this framework, we now demonstrate that even in the absence of assumptions (H4) and (H5), one can locally approximate viscosity solutions of the problem by functions in weighted Orlicz-Sobolev spaces. These weights belong to the fundamental class of solutions $\mathcal{S}$. More precisely, we establish the following result:

\begin{theorem}[{\bf $W^{2,\Upsilon}_{\omega}$-density}]\label{Thm5.1-Density}
Let $u$ be a viscosity solution of
\[
\left\{
\begin{array}{rclcl}
F(D^{2}u,x,t)-u_{t} & = & f(x,t) & \text{in} & \mathrm{Q}^{+}_{1}, \\
\beta \cdot Du + \gamma u & = & g(x,t) & \text{on} & \mathrm{Q}^{*}_{1},
\end{array}
\right.
\]
and assume the structural conditions \(\mathrm{(H1)-(H3)}\). Then, for any $\delta > 0$, there exists a sequence $(u_{j})_{j \in \mathbb{N}} \subset W^{2,\Upsilon}_{\text{loc}}(\mathrm{Q}^{+}_{1}) \cap \mathcal{S}_{p}(\lambda - \delta, \Lambda + \delta, f)$ that converges locally uniformly to $u$.
\end{theorem}

\begin{proof}
The argument follows ideas inspired by \cite[Theorem 6.1]{Bessa} and \cite[Theorem 8.1]{PT}, suitably adapted to the parabolic framework. We present the details here for the reader’s convenience and completeness.

We begin by constructing the desired sequence of operators \( F_{j} : \mathrm{Sym}(n) \times \mathrm{Q}^{+}_{1} \longrightarrow \mathbb{R} \) as follows: for a given \( \delta > 0 \), consider the Pucci maximal operator defined by \( L_{\delta}(\mathrm{X}) \defeq \mathcal{M}^{+}_{(\lambda-\delta),(\Lambda+\delta)}(\mathrm{X}) \). We then define
\begin{eqnarray*}
F_{j} : & \mathrm{Sym}(n) \times \mathrm{Q}^{+}_{1} & \longrightarrow \mathbb{R} \\
& (\mathrm{X}, x, t) & \longmapsto \max\{ F(\mathrm{X}, x, t), L_{\delta}(\mathrm{X}) - \mathrm{d}_{j} \},
\end{eqnarray*}
where \( (\mathrm{d}_{j})_{j \in \mathbb{N}} \) is a divergent sequence given by \( \mathrm{d}_{j} = j(2\Lambda - \lambda + \delta) \).

It is straightforward to verify that each \( F_{j} \) is continuous and \((\lambda - \delta, \Lambda + \delta)\)-parabolic. Moreover, from the definition of \( (\mathrm{d}_{j}) \), the following properties hold:
\begin{itemize}
    \item[\checkmark] \( F = F_{j} \) in \( \mathrm{B}_{j} \times \mathrm{Q}^{+}_{1} \subset \mathrm{Sym}(n) \times \mathrm{Q}^{+}_{1} \).
    \item[\checkmark] \( F_{j} = L_{\delta} - \mathrm{d}_{j} \) outside a ball of radius approximately \( \mathrm{d}_{j} \).
\end{itemize}

Thus, the convex envelope \( F_{j}^{\sharp} \) coincides with \( L_{\delta} \), which is a convex operator. Furthermore, \( F_{j}^{\sharp} \) satisfies structural conditions (H4) and (H5) via \cite[Theorem 1.1]{Kry83} (see also \cite[Theorem 1.1]{WangII} for related results) and \cite[Theorem 5.8]{Geo.Mil}, respectively. Consequently, the conclusions of Proposition \ref{Prop4.2} apply.

Fixing \( j \in \mathbb{N} \), any viscosity solution \( v \) to
\[
\left\{
\begin{array}{rcll}
F_{j}(D^{2}v, x, t) - v_{t} & = & f(x,t) & \text{in } \mathrm{Q}^{+}_{1}, \\
\beta \cdot Dv + \gamma v & = & g(x,t) & \text{on } \mathrm{Q}^{*}_{1},
\end{array}
\right.
\]
admits weighted Orlicz-Sobolev regularity. Specifically, for each \( j \in \mathbb{N} \), there exists a constant \( \kappa_{j} > 0 \) such that
\[
\|v\|_{W^{2,\Upsilon}_{\omega}\left(\mathrm{Q}^{+}_{\frac{1}{2}}\right)} \leq \kappa_{j} \cdot \left( \|v\|_{L^{\infty}(\mathrm{Q}^{+}_{1})}^{n+1} + \|f\|_{L^{\Upsilon}_{\omega}(\mathrm{Q}^{+}_{1})} + \|g\|_{C^{1,\alpha}(\mathrm{Q}^{*}_{1})} \right).
\]

We now define the sequence \( (u_{j})_{j \in \mathbb{N}} \) as viscosity solutions to
\[
\left\{
\begin{array}{rcll}
F_{j}(D^{2}u_{j}, x, t) - (u_{j})_{t} & = & f(x,t) & \text{in } \mathrm{Q}^{+}_{1}, \\
\beta \cdot Du_{j} + \gamma u_{j} & = & g(x,t) & \text{on } \mathrm{Q}^{*}_{1}, \\
u_{j} & = & u & \text{on } \partial_{p} \mathrm{Q}^{+}_{1} \setminus \mathrm{Q}^{*}_{1},
\end{array}
\right.
\]
whose existence is ensured by Theorem \ref{Existencia}. Furthermore, each \( u_{j} \in W^{2,\Upsilon}_{\omega,\mathrm{loc}}(\mathrm{Q}^{+}_{1}) \cap \mathcal{S}_{p}(\lambda - \delta, \Lambda + \delta) \).

By compactness and stability arguments, the sequence \( (u_{j})_{j \in \mathbb{N}} \) converges, up to a subsequence, locally uniformly to a function \( u_{0} \) in the \( C^{0, \alpha} \)-topology. Moreover, \( u_{0} \) is a viscosity solution of
\[
\left\{
\begin{array}{rcll}
F(D^{2}u_{0}, x, t) - (u_{0})_{t} & = & f(x,t) & \text{in } \mathrm{Q}^{+}_{1}, \\
\beta \cdot Du_{0} + \gamma u_{0} & = & g(x,t) & \text{on } \mathrm{Q}^{*}_{1}, \\
u_{0} & = & u & \text{on } \partial_{p} \mathrm{Q}^{+}_{1} \setminus \mathrm{Q}^{*}_{1}.
\end{array}
\right.
\]

Finally, let \( w = u_{0} - u \). By Theorem \ref{comparation}, we have that \( w \) satisfies, in the viscosity sense,
\[
\left\{
\begin{array}{rcll}
w \in \mathcal{S}_{p}(\lambda/n, \Lambda,0)& \text{in } & \mathrm{Q}^{+}_{1}, \\
\beta \cdot Dw + \gamma w = 0 & \text{on } & \mathrm{Q}^{*}_{1}, \\
w = 0 & \text{on } & \partial_{p} \mathrm{Q}^{+}_{1} \setminus \mathrm{Q}^{*}_{1}.
\end{array}
\right.
\]

Applying the A.B.P.T. estimate (Lemma \ref{ABP-fullversion}), we conclude that \( w = 0 \) in \( \overline{\mathrm{Q}^{+}_{1}} \setminus \mathrm{Q}^{*}_{1} \). Therefore, by continuity, we have \( w \equiv 0 \), i.e., \( u = u_{0} \), completing the proof.
\end{proof}

\subsection{Calder\'{o}n-Zygmund Type Estimates for Solutions of the Obstacle Problem}

In the modern mathematical literature, Calder\'{o}n-Zygmund type estimates for obstacle problems with oblique tangential derivatives of the form
\begin{equation}\label{Obst1}
\left\{
\begin{array}{rclcl}
F(D^2 u,Du,x,t)-\frac{\partial u}{\partial t} &\le& f(x,t)& \mbox{in} &   \Omega_{\mathrm{T}}, \\
(F(D^2 u, Du,x,t)-\frac{\partial u}{\partial t}- f)(u-\phi) &=& 0 &\mbox{in}& \Omega_{\mathrm{T}},\\
		u(x, t) &\ge& \phi(x, t) &\mbox{in}& \Omega_{\mathrm{T}},\\
		\beta \cdot Du+\gamma u&=&g(x,t) &\mbox{on}& \mathrm{S}_{T},\\
		u(x,0)&=&0 &\mbox{in}& \overline{\Omega},
	\end{array}
	\right.
\end{equation}
with appropriate data $f$, $\beta$, $\gamma$, and $g$, and obstacle $\phi$, have garnered increasing attention in recent decades due to their connection with extensions of the classical theory for the heat operator and, more generally, for operators in divergence form such as:
\[
\mathfrak{L}\,u = \frac{\partial u}{\partial t}-\div(\mathcal{A}(x, t, u, \nabla u)\nabla u) = f(x, t) \quad \text{in} \quad \Omega_{\mathrm{T}}.
\]

Such regularity estimates are typically derived via a penalization method associated with the corresponding obstacle-free problem, combined with a priori estimates for that problem. In particular, for the obstacle problem \eqref{Obst1}, we consider its associated non-obstacle counterpart \eqref{1.1}.

In this framework, the estimates obtained in Theorem \ref{T1} ensure the existence of a unique solution to \eqref{Obst1}, possessing regularity properties in the setting of weighted Orlicz spaces.

It is worth highlighting that related results for obstacle problems with oblique boundary conditions in the elliptic context have been developed. For instance, Byun et al. \cite{BJ1} established $W^{2,p}$ estimates for convex elliptic models similar to \eqref{Obst1} in the case where $\gamma = g = 0$. Extending this result, Bessa et al. \cite{Bessa} obtained $W^{2,p}$ estimates for the elliptic version of \eqref{Obst1} under relaxed convexity assumptions. Furthermore, weighted Orlicz–Sobolev regularity under such relaxed conditions was achieved by Bessa in \cite{BessaOrlicz}. In the same vein, the work of Bessa and Ricarte \cite{BessaRicarte} provides weighted Lorentz regularity estimates for the obstacle problem.

For our purposes, we require the following further structural assumptions:

\begin{enumerate}
\item[(\bf Obst1)] There exists a modulus of continuity $\iota: [0,+\infty) \to [0,+\infty)$ with $\iota(0)=0$, such that
	\[
	F(\mathrm{X}_1, \vec{q}, r,x_1,t) - F(\mathrm{X}_2, \vec{q},r,x_2,t) \le \iota\left(|x_1-x_2|\right)\left[(|\vec{q}| +1) + \alpha_0 |x_1-x_2|^2\right]
	\]
	holds for any $x_1,x_2 \in \Omega$, $t\in[0,T]$, $\vec{q} \in \mathbb{R}^n$, $r \in \mathbb{R}$, $\alpha_0 >0$, and $\mathrm{X}_1,\mathrm{X}_2 \in \textrm{Sym}(n)$ satisfying
	\[
	- 3 \alpha_0
	\begin{pmatrix}
		\mathrm{Id}_n& 0 \\
		0& \mathrm{Id}_n
	\end{pmatrix}
	\leq
	\begin{pmatrix}
		\mathrm{X}_2&0\\
		0&-\mathrm{X}_1
	\end{pmatrix}
	\leq
	3 \alpha_0
	\begin{pmatrix}
		\mathrm{Id}_n & -\mathrm{Id}_n \\
		-\mathrm{Id}_n& \mathrm{Id}_n
	\end{pmatrix},
	\]
	where $\mathrm{Id}_n$ denotes the identity matrix in $\mathbb{R}^n$.

	\item[(\bf Obst2)] The operator $F$ is proper in the sense that
	\[
	d\cdot (r_{2}-r_{1}) \leq F(\mathrm{X},q,r_{1},x,t)-F(\mathrm{X},q,r_{2},x,t),
	\]
	for any $\mathrm{X} \in \text{Sym}(n)$, $r_1,r_2 \in \mathbb{R}$ with $r_{1}\leq r_{2}$, $x \in \Omega$, $q \in \mathbb{R}^n$, and $t\in[0,T]$, for some $d >0$.
\end{enumerate}

These structural conditions are imposed to ensure the validity of the Comparison Principle for oblique derivative problems such as \eqref{1.1} (cf. \cite[Theorem 2.10]{CCKS}, \cite[Theorem 2.1]{IshiiSato}, and \cite[Theorem 7.17]{Leiberman}), thereby allowing the application of Perron's method for viscosity solutions (see \cite[Sections 7.4 and 7.6]{Leiberman} and \cite[Theorem 3.1]{IshiiSato}).

We now state the principal result of this section.

\begin{theorem}[{\bf Obstacle Problems and Weighted Orlicz Spaces}]\label{T3}
Assume the structural conditions \(\mathrm{(H1)-(H5)}\) and \textnormal{({\bf Obst1})–({\bf Obst2})}. Let $u$ be an $L^{p}$-viscosity solution of \eqref{Obst1}, where $p = p_{0}(n+1)$. Furthermore, suppose that $\partial \Omega \in C^{3}$, $\beta,\gamma \in C^{2}(\mathrm{S}_\mathrm{T})$, and $\phi \in W^{2,\Upsilon}_{\omega}(\Omega_{\mathrm{T}})$, where $\Upsilon(t)=\Phi(t^{n+1})$ (cf. (H2)). Assume also that $\phi$ satisfies $\beta\cdot D\phi+\gamma\phi\geq g$ almost everywhere on $\mathrm{S}_\mathrm{T}$. Then, $u \in W^{2,\Upsilon}_{\omega}(\Omega_{\mathrm{T}})$, and the following estimate holds:
\begin{equation*}
\|u\|_{W^{2, \Upsilon}_{\omega}(\Omega_{\mathrm{T}})} \le \mathrm{C}\cdot \left( \|f\|_{L^{\Upsilon}_{\omega}(\Omega_{\mathrm{T}})}+ \|\phi\|_{W^{2,\Upsilon}_{\omega}(\Omega_{\mathrm{T}})}+\|g\|_{C^{1,\alpha}(\mathrm{S}_\mathrm{T})}\right),
\end{equation*}
where $\mathrm{C}>0$ is a universal constant.
\end{theorem}

\begin{proof}
	
Fixed $\varepsilon \in (0,1)$, we consider a non-decreasing function $\Phi_{\varepsilon} \in C^{\infty}(\mathbb{R})$ such that
$$
\Phi_{\varepsilon}(s) \equiv 0 \quad \text{if} \quad s \leq 0; \quad \Phi_{\varepsilon}(s) \equiv 1 \quad \text{if} \quad s \geq \varepsilon,
$$
$$
\text{and} \quad 0 \leq \Phi_{\varepsilon}(s) \leq 1 \quad \text{for any} \quad s \in \mathbb{R}.
$$
In the sequel, we consider the following penalized problem
\begin{equation}\label{3.1}
\left\{
\begin{array}{rclcl}
F(D^2 u_{\varepsilon}, D u_{\varepsilon}, x, t) - \frac{\partial u_{\varepsilon}}{\partial t} &=& \mathrm{h}^+(x,t) \Phi_{\varepsilon}(u_{\varepsilon} - \phi) + f(x,t) - \mathrm{h}^+(x,t) & \text{in} & \Omega_{\mathrm{T}}, \\
\beta \cdot D u_{\varepsilon} + \gamma u_{\varepsilon} &=& g(x,t) & \text{on} & \mathrm{S}_{\mathrm{T}}, \\
u_{\varepsilon}(x, 0) &=& 0 & \text{in} & \overline{\Omega}.
\end{array}
\right.
\end{equation}
where 
$$
\mathrm{h}(x,t) \defeq f(x,t) - \left(F(D^2 \phi, D \phi, x, t) - \frac{\partial \phi}{\partial t}(x, t)\right).
$$
Now, observe that
\begin{eqnarray*}
\nonumber |\mathrm{h}(x,t)| &\leq& |f(x,t)| + |F(D^2 \phi, D \phi, x, t)| + \left| \frac{\partial \phi}{\partial t}(x,t) \right| \\
\nonumber &\leq& |f(x,t)| + \mathrm{C}(\lambda, \Lambda, \sigma, \xi) \cdot (|D \phi(x,t)| + |D^2 \phi(x,t)|) + \left| \frac{\partial \phi}{\partial t}(x,t) \right| \\
&\Longrightarrow& \| h \|_{L^{\Upsilon}_{\omega}(\Omega_{\mathrm{T}})} \leq \mathrm{C}(\lambda, \Lambda, \sigma, \xi) \left( \| f \|_{L^{\Upsilon}_{\omega}(\Omega_{\mathrm{T}})} + \|\phi\|_{W^{2,\Upsilon}_{\omega}(\Omega_{\mathrm{T}})} \right).
\end{eqnarray*}

We claim that the problem \eqref{3.1} admits a viscosity solution $u_{\varepsilon}$. Indeed, given $v_0 \in L^{\Upsilon}_{\omega}(\Omega_{\mathrm{T}})$, we study the auxiliary problem
\begin{equation}\label{3.3}
\left\{
\begin{array}{rclcl}
F(D^2 u_{\varepsilon}, D u_{\varepsilon}, x, t) - \frac{\partial u_{\varepsilon}}{\partial t} &=& \mathrm{h}^+(x) \Phi_{\varepsilon}(v_0 - \phi) + f(x,t) - \mathrm{h}^+(x,t) & \text{in} & \Omega_{\mathrm{T}}, \\
\beta \cdot D u_{\varepsilon} + \gamma u_{\varepsilon} &=& g(x,t) & \text{on} & \mathrm{S}_{\mathrm{T}}, \\
v(x, 0) &=& 0 & \text{in} & \overline{\Omega}.
\end{array}
\right.
\end{equation}

By Perron's method, we know that under the assumed hypotheses, there exists a unique solution to the problem \eqref{3.3}. Now, we assert that $f_{v_0} = \mathrm{h}^+ \Phi_{\varepsilon}(v_0 - \phi) + f - \mathrm{h}^+$ belongs to $L^{\Upsilon}_{\omega}(\Omega_{\mathrm{T}})$. In fact, we consider the case where $\|\mathrm{h}^+\|_{L^{\Upsilon}_{\omega}(\Omega_{\mathrm{T}})} > 0$. In this case, by the triangle inequality, we have almost everywhere $(x,t) \in \Omega_{\mathrm{T}}$,
\begin{eqnarray*}
|f_{v_0}(x,t)| &\leq& 2 |\mathrm{h}^+(x)| + |f(x,t)|.
\end{eqnarray*}
Consequently, since $\| \mathrm{h}^+ \|_{L^{\Upsilon}_{\omega}(\Omega_{\mathrm{T}})} > 0$, it follows that
\begin{eqnarray*}
\left\| \frac{f_{v_0}}{\| \mathrm{h}^+ \|_{L^{\Upsilon}_{\omega}(\Omega_{\mathrm{T}})} + \| f \|_{L^{\Upsilon}_{\omega}(\Omega_{\mathrm{T}})}} \right\|_{L^{\Upsilon}_{\omega}(\Omega_{\mathrm{T}})} &\leq& 2 \left\| \frac{\mathrm{h}^+}{\| \mathrm{h}^+ \|_{L^{\Upsilon}_{\omega}(\Omega_{\mathrm{T}})}} \right\|_{L^{\Upsilon}_{\omega}(\Omega_{\mathrm{T}})} \\
& +&  \left\| \frac{f}{\| \mathrm{h}^+ \|_{L^{\Upsilon}_{\omega}(\Omega_{\mathrm{T}})} + \| f \|_{L^{\Upsilon}_{\omega}(\Omega_{\mathrm{T}})}} \right\|_{L^{\Upsilon}_{\omega}(\Omega_{\mathrm{T}})} \\
&\leq& 3.
\end{eqnarray*}
Thus,
\begin{eqnarray*}
\| f_{v_0} \|_{L^{\Upsilon}_{\omega}(\Omega_{\mathrm{T}})} &\leq& 3 \left( \| f \|_{L^{\Upsilon}_{\omega}(\Omega_{\mathrm{T}})} + \| \mathrm{h}^+ \|_{L^{\Upsilon}_{\omega}(\Omega_{\mathrm{T}})} \right) \\
&\leq& \mathrm{C}(n, \lambda, \Lambda, \sigma, \xi) \left( \| f \|_{L^{\Upsilon}_{\omega}(\Omega_{\mathrm{T}})} + \| \phi \|_{W^{2,\Upsilon}_{\omega}(\Omega_{\mathrm{T}})} \right).
\end{eqnarray*}
On the other hand, if $\| \mathrm{h}^+ \|_{L^{\Upsilon}_{\omega}(\Omega_{\mathrm{T}})} = 0$, then $\mathrm{h}^+ = 0$ almost everywhere in $\Omega_{\mathrm{T}}$, and thus $f_{v_0} = f$, which is independent of $v_0$. Therefore, from these two cases, we conclude that
\begin{equation}\label{4.3'}
\| f_{v_0} \|_{L^{\Upsilon}_{\omega}(\Omega_{\mathrm{T}})} \leq \mathrm{C}(n, \lambda, \Lambda, \sigma, \xi) \left( \| f \|_{L^{\Upsilon}_{\omega}(\Omega_{\mathrm{T}})} + \| \phi \|_{W^{2,\Upsilon}_{\omega}(\Omega_{\mathrm{T}})} \right),
\end{equation}	
where $\mathrm{C} > 0$ is independent of $v_0$. Now, since the operator $F$ and the data satisfy the hypotheses (H1)-(H5), we can apply Theorem \ref{T1} and conclude that there exists a unique solution $u_{\varepsilon} \in W^{2,\Upsilon}_{\omega}(\Omega_{\mathrm{T}})$. Moreover, $u_{\varepsilon}$ satisfies the following estimate
\begin{eqnarray*}
\| u_{\varepsilon} \|_{W^{2,\Upsilon}_{\omega}(\Omega_{\mathrm{T}})} &\leq& \mathrm{C} \cdot \left( \| u_{\varepsilon} \|_{L^{\infty}(\Omega_{\mathrm{T}})}^{n+1} + \| f_{v_0} \|_{L^{\Upsilon}_{\omega}(\Omega_{\mathrm{T}})} + \|g\|_{C^{1,\alpha}(\mathrm{S}_\mathrm{T})} \right).
\end{eqnarray*}
Finally, using the A.B.P.T. estimate (Theorem \ref{ABP-fullversion}) and \eqref{4.3'}, we obtain
\begin{eqnarray}\label{3.5}
\| u_{\varepsilon} \|_{W^{2,\Upsilon}_{\omega}(\Omega_{\mathrm{T}})} &\leq& \mathrm{C}_0 \cdot \left( \| f \|_{L^{\Upsilon}_{\omega}(\Omega_{\mathrm{T}})} + \| \phi \|_{W^{2,\Upsilon}_{\omega}(\Omega_{\mathrm{T}})} + \|g\|_{C^{1,\alpha}(\mathrm{S}_\mathrm{T})} \right),
\end{eqnarray}
where the constant $\mathrm{C}_0 > 0$ does not depend on $v_0$ or $\varepsilon$.
	
At this point, by defining the operator $\mathcal{T}: L^{\Upsilon}_{\omega}(\Omega_{\mathrm{T}}) \rightarrow W^{2,\Upsilon}_{\omega}(\Omega_{\mathrm{T}}) \subset L^{\Upsilon}_{\omega}(\Omega_{\mathrm{T}})$ given by $\mathcal{T}(v_0) = u_{\varepsilon}$, we conclude that $\mathcal{T}$ maps the $\mathrm{C}_0$-ball in $L^{\Upsilon}_{\omega}(\Omega_{\mathrm{T}})$ into itself (since the estimate \eqref{3.5} holds for all $u_{\varepsilon}$). Hence, $\mathcal{T}$ is a continuous and compact operator. Therefore, by Schauder's Fixed-Point Theorem, there exists $u_{\varepsilon}$ such that $\mathcal{T}(u_{\varepsilon}) = u_{\varepsilon}$, which is a viscosity solution to \eqref{3.1}.

By the definition of the operator $\mathcal{T}$, it follows that the sequence $\{u_{\varepsilon}\}_{\varepsilon \in (0,1)}$ is bounded in the weighted Orlicz-Sobolev space $W^{2,\Upsilon}_{\omega}(\Omega_{\mathrm{T}})$. Thus, by standard compactness arguments and Lemma \ref{mergulhoorliczlebesgue}, we can find a subsequence $(u_{\varepsilon_{j}})_{j \in \mathbb{N}}$ with $\varepsilon_{j} \to 0$ as $j \to \infty$ and a function $u \in W^{2,\Upsilon}_{\omega}(\Omega_{\mathrm{T}})$ such that $u_{\varepsilon_{j}} \rightharpoonup u$ in $W^{2,\Upsilon}_{\omega}(\Omega_{\mathrm{T}})$, $u_{\varepsilon_{j}} \to u$ almost everywhere, and $Du_{\varepsilon_{j}} \in C^{0,\overline{\alpha}}(\overline{\Omega_{\mathrm{T}}})$ for some $\overline{\alpha} = \overline{\alpha}(n, i(\Phi), \omega, \lambda, \Lambda) \in (0,1)$, and $Du_{\varepsilon_{j}} \to Du$ in the $C^{0,\overline{\alpha}}$-topology (cf. \cite[Theorem 1.6]{BessaOrlicz}).

We claim that $u$ is a viscosity solution of \eqref{Obst1}. In fact, since the function $u_{\varepsilon_{j}}$ is a viscosity solution to \eqref{3.1}, we have
\begin{eqnarray}\label{3.7}
\nonumber F(D^{2}u_{\varepsilon_{j}}, Du_{\varepsilon_{j}}, x, t) - \frac{\partial u_{\varepsilon_{j}}}{\partial t} & = & \mathrm{h}^+ \Phi_{\varepsilon_j}(u_{\varepsilon_j} - \phi) + f - \mathrm{h}^+ \\
& \leq & f \quad \text{in} \quad \Omega_{\mathrm{T}}, \quad \forall j \in \mathbb{N},
\end{eqnarray}
since $\Phi_{\varepsilon_{j}}(s) \in [0,1]$ for all $s \in \mathbb{R}$ and $j \in \mathbb{N}$. Thus, by Stability Lemma \ref{Est}, it follows from \eqref{3.7} in the viscosity sense that
\begin{eqnarray*}
F(D^{2}u, Du, x, t) - \frac{\partial u}{\partial t} \leq f(x,t) \quad \text{in} \quad \Omega_{\mathrm{T}}.
\end{eqnarray*}

On the other hand, by the condition $\beta(x,t) \cdot Du_{\varepsilon_{j}}(x,t)+\gamma(x,t)u_{\varepsilon_{j}}(x,t) = g(x,t)$ and since $(Du_{\varepsilon_{j}})_{j \in \mathbb{N}}$ are uniformly bounded and equi-continuous on $\partial_{p}\Omega_{\mathrm{T}}$, we obtain, in the viscosity sense,
\begin{eqnarray}
\beta(x,t) \cdot Du(x,t) + \gamma u = g(x,t) \quad \text{on} \quad \partial_{p}\Omega_{\mathrm{T}}.
\end{eqnarray}

Now, we will show that $u \geq \phi$ on $\Omega_{\mathrm{T}}$. To do so, fix $j \in \mathbb{N}$ and define the set $\mathcal{V}_{j} = \{(x,t) \in \Omega_{\mathrm{T}}; u_{\varepsilon_{j}}(x,t) < \phi(x,t)\}$. If $\mathcal{V}_{j} = \emptyset$, there is nothing to prove. However, if $\mathcal{V}_{j} \neq \emptyset$, then
\begin{eqnarray*}
F(D^{2}u_{\varepsilon_{j}}, Du_{\varepsilon_{j}}, x, t) - \frac{\partial u_{\varepsilon_{j}}}{\partial t} & = & \mathrm{h}^+ \Phi_{\varepsilon_j}(u_{\varepsilon_{j}} - \phi) + f - \mathrm{h}^+ \\
& = & f - \mathrm{h}^{+} \\
& \leq & f - \mathrm{h} \\
& = & F(D^{2}\phi, D\phi, x, t) - \frac{\partial \phi}{\partial t} \quad \text{in} \quad \mathcal{V}_{j}.
\end{eqnarray*}
Moreover, $u_{\varepsilon_{j}}(x,0) \leq \phi(x,0)$ in $\overline{\mathcal{V}_{j}}$, and consequently, by the Comparison Principle \cite[Theorem 2.1]{IshiiSato}, it follows that $u_{\varepsilon_{j}} \geq \phi$ in $\mathcal{V}_{j}$, which leads to a contradiction. Thus, $\mathcal{V}_{j} = \emptyset$. This proves the claim.

Finally, we need to show that
$$
F(D^2 u, Du, x, t) - \frac{\partial u}{\partial t} = f(x,t) \quad \text{in} \quad \{ u > \phi \}
$$
in the viscosity sense. In fact, we observe that for each $j \in \mathbb{N}$, it follows that 
\[
\mathrm{h}^+ \Phi_{\varepsilon_j}(u_{\varepsilon_j} - \phi) + f - \mathrm{h}^+ \to f\,\, \mbox{a.e. on} \left\{ u > \phi + \frac{1}{k} \right\}
\] Therefore, by Stability results (Lemma \ref{Est}), we conclude (in the viscosity sense) that
$$
F(D^2 u, Du, x, t) - \frac{\partial u}{\partial t} = f(x,t) \quad \text{in} \quad \{u > \phi\} = \bigcup_{k=1}^{\infty} \left\{ u > \phi + \frac{1}{k} \right\} \quad \text{as} \quad j \to +\infty,
$$
therefore $u$ is a viscosity solution of \eqref{Obst1}.

To conclude the proof of this theorem, by weak convergence $u_{\varepsilon_{j}} \rightharpoonup u$ in $W^{2,\Upsilon}_{\omega}(\Omega_{\mathrm{T}})$ and the estimate \eqref{3.5} holding for all $u_{\varepsilon_{j}}$, we have
\begin{eqnarray*}
\|u\|_{W^{2,\Upsilon}_{\omega}(\Omega_{\mathrm{T}})} \leq \liminf_{j \to \infty} \|u_{\varepsilon_{j}}\|_{W^{2,\Upsilon}_{\omega}(\Omega_{\mathrm{T}})} \leq \mathrm{C}_{0} \left( \|f\|_{L^{\Upsilon}_{\omega}(\Omega_{\mathrm{T}})} + \|\phi\|_{W^{2,\Upsilon}_{\omega}(\Omega_{\mathrm{T}})} + \|g\|_{C^{1,\alpha}(\mathrm{S}_\mathrm{T})} \right).
\end{eqnarray*}
This concludes the proof of the theorem.
\end{proof}

As a consequence, we prove the following result:
\begin{corollary}\label{Uniqueness-result}
	Under the same assumptions as Theorem \ref{T3}, the problem \eqref{Obst1} has a unique solution.
\end{corollary}

\begin{proof}
	Let $u$ and $v$ be two viscosity solutions of \eqref{Obst1}. Assume, by contradiction, that $u \neq v$. Without loss of generality, we may assume that
	$$
	\mathcal{O}_{\sharp} = \{v > u\} \neq \emptyset.
	$$
	Since $v > u \ge \phi$ in $\mathcal{O}_{\sharp}$, we obtain in the viscosity sense
	$$
	F(D^2 v, Dv, x, t) - \frac{\partial v}{\partial t} = f(x,t) \quad \textrm{in} \quad \mathcal{O}_{\sharp}.
	$$
	Consequently, we conclude that
	$$
	\left\{
	\begin{array}{rclclcl}
	F(D^2u,Du,x,t)- \frac{\partial u}{\partial t} &\leq& f(x,t) & \leq & F(D^2 v, Dv, x,t) - \frac{\partial v}{\partial t} & \textrm{in} & \mathcal{O}_{\sharp}, \\
	& & u(x,t) &=& v(x,t) & \textrm{on} & \partial_{p} \mathcal{O}_{\sharp} \setminus \partial_{p} \Omega, \\
	\beta \cdot Du + \gamma u &=& g(x,t) &=& \mathcal{B}(x,t,Dv) & \textrm{on} & \partial_{p} \mathcal{O}_{\sharp} \cap \partial_{p} \Omega, \\
	u(x,t) &=& 0 &=& v(x, 0) & \textrm{in} & \mathcal{O}_{\sharp} \cap \overline{\Omega}.
	\end{array}
	\right.
	$$
	
	Therefore, by the Comparison Principle \cite[Theorem 2.1]{IshiiSato}, it follows that $u \geq v$ in $\mathcal{O}_{\sharp}$ if $\partial_{p} \mathcal{O}_{\sharp} \cap \partial_{p} \Omega \neq \emptyset$. Otherwise, the same conclusion holds by \cite[Section 3]{Imbert;Silvestre} or \cite[Theorem 1]{HK}. This contradicts the definition of the set $\mathcal{O}_{\sharp}$, thereby proving the uniqueness of the solution.
\end{proof}

\subsection{Weighted Orlicz-BMO Estimates} 

In this section, we will address another application of weighted Orlicz estimates arising in the context of problem \eqref{1.1}.  Specifically, when the source term $f$ possesses \textit{weighted Orlicz bounded mean oscillation}. In regularity theory, it is well known that the boundedness of $f$ does not necessarily imply the boundedness of the Hessian $D^2 u$ for solutions of \eqref{1.1} - even in the linear scenario (cf. \cite{BessaOrlicz}, \cite{Bessa}, \cite{CH} and \cite{daSR19}). This fact highlights the subtlety involved in estimating $D^2 u$ and the time derivative $u_t$ when the source term lacks sufficient regularity. In this direction, we consider the following boundary value problem:
\begin{eqnarray}\label{problemalocal}
\left\{
\begin{array}{rclcl}
F(D^2u,x,t)-u_{t} &=& f(x,t) & \text{in} & \mathrm{Q}^+_1, \\
\beta\cdot Du + \gamma u &=& g(x,t) & \text{on} & \mathrm{Q}_1^{*},
\end{array}
\right.
\end{eqnarray}
and aim to demonstrate that the $D^2 u$ and $u_t$ exhibit bounded mean oscillation concerning the weighted Orlicz space $L^{\Upsilon}_{\omega}(\mathrm{Q}^{+}_{1})$. 

\begin{definition}\label{deforlicbmospace}
We recall that a function $f \in L^1_{\text{loc}}(\Omega_{\mathrm{T}})$ is said to belong to the space $L^{\Phi}_{\omega}\text{-}BMO(\Omega_{\mathrm{T}})$, for an $\mathrm{N}$-function $\Phi \in \Delta_2 \cap \nabla_2$ and a weight $\omega$, if
\begin{eqnarray*}
\|f\|_{L^{\Phi}_{\omega}\text{-}BMO(\Omega_{\mathrm{T}})} := \sup_{\mathrm{Q} \subset \Omega_{\mathrm{T}}} \frac{\|(f - f_{\mathrm{Q}})\chi_{\mathrm{Q}}\|_{L^{\Phi}_{\omega}(\Omega_{\mathrm{T}})}}{\|\chi_{\mathrm{Q}}\|_{L^{\Phi}_{\omega}(\Omega_{\mathrm{T}})}} < +\infty,
\end{eqnarray*}
where the supremum is taken over all parabolic cubes $\mathrm{Q} \subset \Omega_{\mathrm{T}}$, and for each such cube, we define
\begin{eqnarray*}
f_{\mathrm{Q}} := \intav{\mathrm{Q}} f(x,t) \, dxdt.
\end{eqnarray*}
\end{definition}

\begin{example}
If $\Phi(s) = s^p$ for $p > 1$ and $\omega \equiv 1$, then
\begin{eqnarray*}
\|f\|_{L^{\Phi}_{\omega}\text{-}BMO(\Omega_{\mathrm{T}})} = \sup_{\mathrm{Q} \subset \Omega_{\mathrm{T}}} \left( \intav{\mathrm{Q}} |f - f_{\mathrm{Q}}|^p \, dxdt \right)^{\frac{1}{p}} = \|f\|_{p\text{-}BMO(\Omega_{\mathrm{T}})},
\end{eqnarray*}
which coincides with the classical definition of the $p$-BMO space.
\end{example}

\begin{remark}\label{observacaodeequivalencia}
It follows from \cite[Theorem 2.3]{Ho} that, under the assumptions $\Phi \in \Delta_2 \cap \nabla_2$ and $\omega \in \mathcal{A}_{i(\Phi)}$, there exist universal constants $0 < \mathfrak{a} \leq \mathfrak{b}$ such that
\begin{eqnarray*}
\mathfrak{a} \|f\|_{BMO(\mathrm{Q}^{+}_{1})} \leq \|f\|_{L^{\Phi}_{\omega}\text{-}BMO(\mathrm{Q}^{+}_{1})} \leq \mathfrak{b} \|f\|_{BMO(\mathrm{Q}^{+}_{1})}, \quad \forall f \in L^1_{\text{loc}}(\mathrm{Q}^{+}_{1}),
\end{eqnarray*}
where $BMO(\mathrm{Q}^{+}_{1})$ denotes the classical space of functions with bounded mean oscillation.
\end{remark}

The application we intend to present is summarized in the following theorem:

\begin{theorem}[{\bf $L^{\Upsilon}_{\omega}$-BMO Regularity of the Hessian}]\label{BMO}
Let $u$ be an $L^p$-viscosity solution to problem \eqref{problemalocal}, where $f \in L^{\Upsilon}_{\omega}\text{-}BMO(\mathrm{Q}^+_1) \cap L^{\Upsilon}_{\omega}(\mathrm{Q}^{+}_{1})$, with $p = p_0 (n+1)$ and $\Upsilon(s) = \Phi(s^{n+1})$, where $\Phi$ is an $\mathrm{N}$-function and $\omega \in \mathfrak{A}_{i(\Phi)}$ as assumed in hypothesis \(\mathrm{(H2)}\). Suppose further that conditions \(\mathrm{(H1)-(H3)}\) and \(\mathrm{(H5)}\) hold. Then,
\[
u_{t}, D^2 u \in L^{\Upsilon}_{\omega}\text{-}BMO\left(\mathrm{Q}^{+}_{\frac{1}{2}}\right),
\]
and the following estimate is satisfied:
\begin{equation*}
\|u_{t}\|_{L^{\Upsilon}_{\omega}\text{-}BMO\left(\mathrm{Q}^{+}_{\frac{1}{2}}\right)} + \|D^2 u\|_{L^{\Upsilon}_{\omega}\text{-}BMO\left(\mathrm{Q}^{+}_{\frac{1}{2}}\right)} 
\leq \mathrm{C} \left( \|u\|^{n+1}_{L^{\infty}(\mathrm{Q}^+_1)} + \|f\|_{L^{\Upsilon}_{\omega}\text{-}BMO(\mathrm{Q}^+_1)} + \|g\|_{C^{1, \alpha}(\mathrm{Q}^{*}_{1})} \right),
\end{equation*}
where the constant $\mathrm{C} > 0$ depends only on $n$, $T$, $\lambda$, $\Lambda$, $\mu_0$, $p_0$, $\omega$, $i(\Phi)$, $\mathrm{c}_2$, and the norms $\|\beta\|_{C^{1, \alpha}(\mathrm{Q}^{*}_1)}$ and $\|\gamma\|_{C^{1, \alpha}(\mathrm{Q}^{*}_1)}$.
\end{theorem}

To establish this result, we follow a similar strategy to that presented in \cite{BessaOrlicz}, \cite{Bessa}, \cite{CH}, and \cite{daSR19}. Through an approximation lemma involving frozen coefficients, we approximate the limiting profile in a manner analogous to Lemma \ref{Approx}. In this way, for sufficiently small values of $\mu$ and $f$, it is possible to construct quadratic polynomials that approximate the limiting profile with an error of order $r^2$ in parabolic cylinders. The following approximation lemma will be instrumental (see Lemma \ref{Approx} for further details).

\begin{lemma}[\bf Approximation Lemma II]\label{aprox2}
Assume conditions \(\mathrm{(H1)-(H3)}\) hold. Given $\delta > 0$, there exists $\epsilon_0 = \epsilon_0(\delta, n, \lambda, \Lambda, \mu_{0}) < 1$ such that if
\[
\max\left\{\tau, \, \|f\|_{L^{\Upsilon}_{\omega}-BMO(\mathrm{Q}^+_1)}\right\} \le \epsilon_0,
\]
then any two (normalized) $L^{p}$-viscosity solutions $u$ and $v$ of
\[
\left\{
\begin{array}{rclcl}
F_{\tau}(D^2 u , x,t) - u_{t} &=& f(x,t) & \text{in} & \mathrm{Q}^+_1,\\
\beta \cdot Du + \gamma u &=& g(x,t) & \text{on} & \mathrm{Q}_1^{*}
\end{array}
\right.
\]
and
\[
\left\{
\begin{array}{rclcl}
F^{\sharp}(D^2 \mathfrak{h}, x_0, t_0) - \mathfrak{h}_{t} &=& 0 & \text{in} & \mathrm{Q}^+_1,\\
\beta \cdot D\mathfrak{h} + \gamma \mathfrak{h} &=& g(x,t) & \text{on} & \mathrm{Q}_1^{*}
\end{array}
\right.
\]
satisfy the following estimate:
\[
\|u - \mathfrak{h}\|_{L^{\infty}(\mathrm{Q}^+_{\frac{7}{8}})} \le \delta.
\]
\end{lemma}

With the aid of this Approximation Lemma II, we now establish a quantitative result regarding the closeness of viscosity solutions to quadratic polynomials.

\begin{lemma}[\bf Quadratic Approximation] \label{quadraticaprox}
Under the hypotheses of Theorem \ref{BMO}, there exist universal constants $\mathrm{C}^{\ast}>0$, $\tau_0 > 0$, and $r \in \left(0, \frac{1}{2}\right]$ such that if $u$ is a (normalized) viscosity solution of
\[
\left\{
\begin{array}{rclcl}
F_{\tau}(D^2 u, x,t) - u_{t} &=& f(x,t) & \text{in} & \mathrm{Q}^+_1,\\
\beta \cdot Du + \gamma u &=& g(x,t) & \text{on} & \mathrm{Q}^{*}_1
\end{array}
\right.
\]
with
\[
\max\left\{\tau, \, \|f\|_{L^{\Upsilon}_{\omega}-BMO(\mathrm{Q}^+_1)}\right\} \le \tau_0,
\]
then there exists a quadratic polynomial $\mathrm{P}: \mathrm{Q}^+_1 \to \mathbb{R}$, with $\|\mathrm{P}\|_{\infty} \le \mathrm{C}^{\ast}$, such that
\[
\sup_{\mathrm{Q}^+_r} |u(x,t) - \mathrm{P}(x,t)| \le r^2.
\]
\end{lemma}

\begin{proof}
The proof follows the same general methodology as the elliptic case treated in  \cite[Lemmas 4.1 and 4.2]{Bessa}, with the following remark: since $f \in L^{\Upsilon}_{\omega}(\mathrm{Q}^+_1)$, it follows from Lemma \ref{mergulhoorliczlebesgue} that
\[
\|f\|_{L^p(\mathrm{Q}^+_1)} \le \mathrm{C} \|f\|_{L^{\Upsilon}_{\omega}(\mathrm{Q}^+_1)}.
\]
Thus, proceeding analogously to \cite[Corollary 4.3]{Bessa}, we obtain the desired quadratic polynomial with the claimed bounds.
\end{proof}

We are now in a position to present the proof of Theorem \ref{BMO}.

\begin{proof}[\bf Proof of Theorem \ref{BMO}]
We begin by selecting $\kappa \in (0,1)$, to be determined later. Define $w(x,t) := \kappa u(x,t)$, so that $w$ is a normalized viscosity solution to
\[
\left\{
\begin{array}{rclcl}
F_{\tau}(D^2 w, x,t)-w_{t} &=& \tilde{f}(x,t) & \mbox{in} & \mathrm{Q}^+_1,\\
\beta \cdot Dw+\gamma w &=& \tilde{g}(x,t) & \mbox{on} & \mathrm{Q}^{*}_1,
\end{array}
\right.
\]
where $\tau := \kappa$, $\tilde{f} := \kappa f$, and $\tilde{g} := \kappa g$. We now choose $\kappa$ such that 
\[
\max\left\{\tau, \mathfrak{a}^{-1}\mathfrak{b}\|\tilde{f}\|_{L^{\Upsilon}_{\omega}-BMO(\mathrm{Q}^{+}_{1})}\right\} \leq \tau_0,
\]
where $\tau_0$ is given by Lemma \ref{quadraticaprox} and $\mathfrak{a}, \mathfrak{b}$ are the constants introduced in Remark \ref{observacaodeequivalencia}. Under this assumption, we prove the result for $w$, and the corresponding conclusion for $u$ follows immediately.

Our goal is to construct a sequence of quadratic polynomials $(\mathrm{P}_{k})_{k\geq 0}$ of the form
\[
\mathrm{P}_{k}(x,t) = a_{k} + b_{k} t + c_{k} \cdot x + \frac{1}{2}x^{t}M_{k}x,
\]
satisfying the following properties:
\begin{itemize}
    \item[\checkmark] $F^{\sharp}(M_{k},x,t) = \tilde{f}_{\mathrm{Q}^{+}_{1}} + b_{k}$;
    \item[\checkmark] $\displaystyle \sup_{\mathrm{Q}^{+}_{r^{k}}}|w - P_{k}| \leq r^{2k}$;
    \item[\checkmark] $\displaystyle r^{2(k-1)}|a_{k} - a_{k-1}| + r^{k-1}|c_{k} - c_{k-1}| + |M_{k} - M_{k-1}| \leq \mathrm{C}^{\ast}r^{2(k-1)}$,
\end{itemize}
for $r \in (0,\frac{1}{2}]$ as given in Lemma \ref{quadraticaprox}. The proof proceeds by induction on $k$.

Let $\mathrm{P}_{-1} = \mathrm{P}_0 := \frac{1}{2}x^{t}M'x$, where $M' \in \text{Sym}(n)$ satisfies $F^{\sharp}(M',x,t) = \tilde{f}_{\mathrm{Q}^{+}_{1}}$. This trivially verifies the base case $k=0$.

Assume now that the polynomials $\mathrm{P}_{0}, \mathrm{P}_{1}, \ldots, \mathrm{P}_{k}$ have been constructed to satisfy the conditions above. Define the auxiliary function $w_k : \overline{\mathrm{Q}_{1}^{+}} \to \mathbb{R}$ by
\[
w_k(x,t) := \frac{(w - \mathrm{P}_k)(r^{k}x, r^{2k}t)}{r^{2k}}.
\]
Then, $w_k$ is a normalized viscosity solution to
\[
\left\{
\begin{array}{rclcl}
(F_{k})_{\tau}(D^2 w_{k}, x,t) - (w_k)_t &=& f_{k}(x,t) & \mbox{in} & \mathrm{Q}^+_1,\\
\beta_{k} \cdot Dw_{k} + \gamma_{k} w_{k} &=& g_{k}(x,t) & \mbox{on} & \mathrm{Q}^{*}_1,
\end{array}
\right.
\]
where
\[
F_k(M,x,t) := F(M + M_k, r^k x, r^{2k} t) + b_k,
\]
and the rescaled data are defined as
\[
f_k(x,t) := \tilde{f}(r^k x, r^{2k} t), \quad
\beta_k(x,t) := \beta(r^k x, r^{2k} t), \quad
\gamma_k(x,t) := r^k \gamma(r^k x, r^{2k} t),
\]
\[
g_k(x,t) := r^{-k} \left(\tilde{g} - \beta \cdot D\mathrm{P}_k - \gamma \mathrm{P}_k\right)(r^k x, r^{2k} t).
\]

From the definition of $\kappa$ and $f_k$, it follows that
\begin{equation}
\label{ineq:fk_bound}
\|f_k\|_{L^{\Upsilon}_{\omega}-BMO(\mathrm{Q}^{+}_{1})}
\leq \mathfrak{b} \|\tilde{f}\|_{BMO(\mathrm{Q}^{+}_{1})}
\leq \mathfrak{a}^{-1} \mathfrak{b} \|\tilde{f}\|_{BMO(\mathrm{Q}^{+}_{1})}
\leq \tau_0.
\end{equation}

Additionally, observe that
\[
F_k^{\sharp}(M,x,t) = F^{\sharp}(M + M_k, r^k x, r^{2k} t) + b_k,
\]
and hence $F_k^{\sharp}$ satisfies condition (H5), since $F^{\sharp}$ does and $F^{\sharp}(M_k,x,t) = \tilde{f}_{\mathrm{Q}^{+}_1} + b_k$.

Thus, we may apply Lemma \ref{quadraticaprox} to obtain a quadratic polynomial $\tilde{\mathrm{P}}$ of the form
\[
\tilde{\mathrm{P}}(x,t) = \tilde{a} + \tilde{b}t + \tilde{c} \cdot x + \frac{1}{2}x^{t} \tilde{M} x,
\]
such that
\begin{equation}\label{estbmo1}
\sup_{\mathrm{Q}^+_r}|w_k - \tilde{\mathrm{P}}| \leq r^2.
\end{equation}

Define
\[
a_{k+1} := a_k + \tilde{a} r^{2k}, \quad
b_{k+1} := b_k + \tilde{b}, \quad
c_{k+1} := c_k + \tilde{c} r^k, \quad
M_{k+1} := M_k + \tilde{M}.
\]
Then $\mathrm{P}_{k+1}$ is well defined, and from \eqref{estbmo1} we obtain
\[
\sup_{\mathrm{Q}^{+}_{r^{k+1}}}|w - \mathrm{P}_{k+1}| \leq r^{2(k+1)}.
\]
The remaining estimates follow from Lemma \ref{quadraticaprox}, completing the induction step and thus proving the construction.

To finish the proof, let $s \in \left(0, \frac{1}{2} \right)$ and choose $k$ such that $0 < r^{k+1} < s \leq r^k$. Then,
\begin{equation}\label{estimativa2}
\begin{array}{rcl}
  \frac{
\|(w_t - b_k)\chi_{\mathrm{Q}^{+}_{s}}\|_{L^{\Upsilon}_{\omega}(\mathrm{Q}^{+}_{1/2})}
+ \|(D^2 w - M_k)\chi_{\mathrm{Q}^{+}_{s}}\|_{L^{\Upsilon}_{\omega}(\mathrm{Q}^{+}_{1/2})}
}{
\|\chi_{\mathrm{Q}^{+}_{s}}\|_{L^{\Upsilon}_{\omega}(\mathrm{Q}^{+}_{1/2})}
} & \leq & \frac{
\|(w_t - b_k)\chi_{\mathrm{Q}^{+}_{s}}\|_{L^{\Upsilon}_{\omega}(\mathrm{Q}^{+}_{1/2})}
+ \|(D^2 w - M_k)\chi_{\mathrm{Q}^{+}_{s}}\|_{L^{\Upsilon}_{\omega}(\mathrm{Q}^{+}_{1/2})}
}{
\|\chi_{\mathrm{Q}^{+}_{s}}\|_{L^{\Upsilon}_{\omega}(\mathrm{Q}^{+}_{1/2})}
} \\
   & \leq  & \mathrm{C}\left(
\frac{\|(w_t - b_k)\chi_{\mathrm{Q}^{+}_{r^k}}\|_{L^{\Upsilon}_{\omega}}}{\|\chi_{\mathrm{Q}^{+}_{r^k}}\|_{L^{\Upsilon}_{\omega}}}
+
\frac{\|(D^2 w - M_k)\chi_{\mathrm{Q}^{+}_{r^k}}\|_{L^{\Upsilon}_{\omega}}}{\|\chi_{\mathrm{Q}^{+}_{r^k}}\|_{L^{\Upsilon}_{\omega}}}
\right) \\
& \leq & \mathrm{C}\left(
\|(w_k)_t\|_{L^{\Upsilon}_{\omega}(\mathrm{Q}^{+}_{r^k})}
+ \|D^2 w_k\|_{L^{\Upsilon}_{\omega}(\mathrm{Q}^{+}_{r^k})}
\right)\\
& \leq & \mathrm{C}< \infty
\end{array}
\end{equation}
where we have used Lemma \ref{mergulhoorliczlebesgue} and Proposition \ref{Prop4.2}.

Therefore, from \eqref{estimativa2}, we deduce
\[
\frac{
\|(w_t - (w_t)_{\mathrm{Q}^{+}_{s}})\chi_{\mathrm{Q}^{+}_{s}}\|_{L^{\Upsilon}_{\omega}}
+ \|(D^2 w - (D^2 w)_{\mathrm{Q}^{+}_{s}})\chi_{\mathrm{Q}^{+}_{s}}\|_{L^{\Upsilon}_{\omega}}
}{
\|\chi_{\mathrm{Q}^{+}_{s}}\|_{L^{\Upsilon}_{\omega}}
}
\leq 2\mathrm{C}.
\]
This implies that
\[
\|w_t\|_{L^{\Upsilon}_{\omega}-BMO(\mathrm{Q}^{+}_{1/2})}
+ \|D^2 w\|_{L^{\Upsilon}_{\omega}-BMO(\mathrm{Q}^{+}_{1/2})}
\leq \mathrm{C} < \infty,
\]
thereby completing the proof.
\end{proof}

\section{Variable Exponent Morrey Estimates}\label{Section6}

We conclude this work by presenting the proof of Theorem \ref{T2}. The strategy for establishing this result is based on applying the estimates obtained in Theorem \ref{T1} in the particular case of weighted Lebesgue spaces. These estimates are then utilized to derive the desired conclusion through the following extrapolation result (see \cite[Theorem 2.21]{CUW} for further details):

\begin{lemma}[\bf Weighted Variable Exponent Extrapolation]\label{extrapolation}
Let $U \subset \mathbb{R}^{n+1}$ be a bounded domain. Suppose that for some $p \geq 1$ and for every $\omega \in \mathfrak{A}_{1}$, the following inequality holds:
\begin{eqnarray*}
\int_{U}|f(x,t)|^{p}\omega(x,t)\,dx\,dt \leq \mathrm{C} \int_{U}|g(x,t)|^{p}\omega(x,t)\,dx\,dt,
\end{eqnarray*}
where $f, g: U \longrightarrow \mathbb{R}$ are measurable functions, and $\mathrm{C} > 0$ is a constant. If $\varsigma$ is a log-H\"{o}lder continuous function satisfying $n+2 < \varsigma_{1} \leq \varsigma(x,t) \leq \varsigma_{2} < \infty$ for all $(x,t) \in U$, with $\varsigma_{1} > p$, then the following estimate holds:
\begin{eqnarray*}
\|f\|_{L^{\varsigma(\cdot)}_{\omega}(U)} \leq \mathrm{C}(n, \varsigma_{1}, \varsigma_{2}, \mathrm{C}_{\varsigma}, U) \|g\|_{L^{\varsigma(\cdot)}_{\omega}(U)},
\end{eqnarray*}
where $\mathrm{C}_{\varsigma} > 0$ is the constant appearing in Remark \ref{remark}.
\end{lemma}

We are now in a position to present the proof of Theorem \ref{T2}.

\begin{proof}[\bf Proof of Theorem \ref{T2}]
We begin, without loss of generality, by normalizing and assuming that $\|f\|_{L^{\varsigma(\cdot),\varrho(\cdot)}(\Omega_{\mathrm{T}})} = 1$. Let $\omega \in \mathfrak{A}_{1}$ be an arbitrary weight and define $p := \frac{\varsigma_{1}+n+2}{2} \in (n+2,\varsigma_{1})$. By applying Theorem \ref{T1} in conjunction with Lemma \ref{extrapolation}, we obtain the estimate
\begin{eqnarray}\label{est1teo110}
\|u\|_{L^{\varsigma(\cdot)}_{\omega}(\Omega_{\mathrm{T}})} \leq \mathrm{C} \|f\|_{L^{\varsigma(\cdot)}_{\omega}(\Omega_{\mathrm{T}})}.
\end{eqnarray}

We note that inequality \eqref{est1teo110} ensures the validity of the theorem in the case $\varrho \equiv 0$, and thus we will focus on the case where $\varrho \not\equiv 0$.

Our goal now is to establish the estimate
\begin{eqnarray*}
\|D^{2}u\|_{L^{\varsigma(\cdot),\varrho(\cdot)}(\Omega_{\mathrm{T}})} \leq \mathrm{C} \|f\|_{L^{\varsigma(\cdot),\varrho(\cdot)}(\Omega_{\mathrm{T}})},
\end{eqnarray*}
as the corresponding estimates for $u$, $u_t$, and $Du$ follow by analogous arguments. To this end, observe that $\varrho_0 > 0$ since $\varrho \not\equiv 0$.

Extend $f$ by zero outside $\Omega_{\mathrm{T}}$ and fix an arbitrary point $(x_0, t_0) \in \Omega_{\mathrm{T}}$ and radius $r > 0$. According to \cite[Proposition 2]{CoiRoc}, for every $\iota \in (0,1)$, the function $\omega = \left(\mathcal{M}(\chi_{\mathrm{Q}_r(x_0, t_0)})\right)^{\iota}$ belongs to the Muckenhoupt class $\mathfrak{A}_{1}$. In particular, for all $\iota \in \left(\frac{\varrho_0}{n+2}, 1\right) \subset (0,1)$, we obtain from \eqref{est1teo110} the estimate
\begin{eqnarray}\label{est2teo110}
\int_{\Omega_{\mathrm{T}}(x_0, t_0; r)} |D^{2}u|^{\varsigma(x,t)} \, dx\,dt 
&=& \int_{\Omega_{\mathrm{T}}} |D^{2}u|^{\varsigma(x,t)} \chi_{\mathrm{Q}_r(x_0, t_0)}(x,t)\, dx\,dt \nonumber\\
&=& \int_{\Omega_{\mathrm{T}}} |D^{2}u|^{\varsigma(x,t)} \left(\chi_{\mathrm{Q}_r(x_0, t_0)}(x,t)\right)^{\iota} dx\,dt \nonumber\\
&\leq& \int_{\Omega_{\mathrm{T}}} |D^{2}u|^{\varsigma(x,t)} \omega(x,t) \, dx\,dt \nonumber\\
&\leq& \mathrm{C} \int_{\Omega_{\mathrm{T}}} |f|^{\varsigma(x,t)} \omega(x,t) \, dx\,dt \nonumber\\
&=& \mathrm{C} \int_{\mathbb{R}^{n+1}} |f|^{\varsigma(x,t)} \omega(x,t) \, dx\,dt,
\end{eqnarray}
where, in the third line, we have used that $\left(\chi_{\mathrm{Q}_r(x_0, t_0)}\right)^{\iota} \leq \left(\mathcal{M}(\chi_{\mathrm{Q}_r(x_0, t_0)})\right)^{\iota}$ almost everywhere in $\Omega_{\mathrm{T}}$, and in the final equality that $f$ vanishes outside $\Omega_{\mathrm{T}}$.

We now partition $\mathbb{R}^{n+1}$ into dyadic parabolic cubes to obtain the disjoint union
\begin{eqnarray*}
\mathbb{R}^{n+1} = \mathrm{Q}_{2r}(x_0, t_0) \cup \left(\bigcup_{k \geq 1} \left( \mathrm{Q}_{2^{k+1}r}(x_0, t_0) \setminus \mathrm{Q}_{2^k r}(x_0, t_0) \right) \right).
\end{eqnarray*}
Substituting this decomposition into the last integral in \eqref{est2teo110}, we derive
\begin{eqnarray}\label{est3teo110}
\int_{\Omega_{\mathrm{T}}(x_0, t_0; r)} |D^{2}u|^{\varsigma(x,t)} \, dx\,dt 
&\leq& \mathrm{C} \Bigg( \underbrace{\int_{\mathrm{Q}_{2r}(x_0, t_0)} |f|^{\varsigma(x,t)} \omega(x,t) \, dx\,dt}_{\defeq A_0} \nonumber\\
&&+ \sum_{k=1}^{\infty} \underbrace{\int_{\mathrm{Q}_{2^{k+1}r}(x_0, t_0) \setminus \mathrm{Q}_{2^k r}(x_0, t_0)} |f|^{\varsigma(x,t)} \omega(x,t) \, dx\,dt}_{\defeq A_k} \Bigg).
\end{eqnarray}

We now proceed to estimate each of the integrals appearing on the right-hand side of inequality \eqref{est3teo110}:

\begin{itemize}
\item[\checkmark] \textbf{Estimate of $A_{0}$.}\\
By the definition of $\omega$, it follows that $\omega(x,t)\leq 1$ for almost every $(x,t)\in \mathbb{R}^{n+1}$. Moreover, by the assumption imposed on the exponent $\varrho$, we have $2^{\varrho(x_{0},t_{0})}\leq 2^{\varrho_{0}}<2^{n+2}$. These facts together ensure that
\begin{eqnarray}\label{estA0}
A_{0}&\leq& \int_{\mathrm{Q}_{2r}(x_{0},t_{0})}|f(x,t)|^{\varsigma(x,t)}\,dxdt\nonumber\\
&=&2^{n+2}\cdot\frac{1}{2^{n+2}}\int_{\Omega_{\mathrm{T}}(x_{0},t_{0};2r)}|f(x,t)|^{\varsigma(x,t)}\,dxdt\nonumber\\
&<&2^{n+2}\cdot\frac{1}{2^{\varrho(x_{0},t_{0})}}(2r)^{\varrho(x_{0},t_{0})}\rho_{\varsigma(\cdot),\varrho(\cdot)}(f)\nonumber\\
&\leq&2^{n+2}\cdot\frac{1}{2^{\varrho(x_{0},t_{0})}}(2r)^{\varrho(x_{0},t_{0})}\|f\|_{L^{\varsigma(\cdot),\varrho(\cdot)}(\Omega_{\mathrm{T}})}\nonumber\\
&\leq& 2^{n+2}r^{\varrho(x_{0},t_{0})}\|f\|_{L^{\varsigma(\cdot),\varrho(\cdot)}(\Omega_{\mathrm{T}})},
\end{eqnarray}
where the penultimate inequality follows from the modular-unit ball property of the norm.

\item[\checkmark] \textbf{Estimate of $A_{k}$ for all $k\in\mathbb{N}$.}\\
We begin by observing that
\begin{eqnarray}\label{est4teo110}
\intav{\mathrm{Q}_{\rho}(y,s)}\chi_{\mathrm{Q}_{r}(x_{0},t_{0})}(x,t)\,dxdt &=& \frac{|\mathrm{Q}_{r}(x_{0},t_{0})\cap \mathrm{Q}_{\rho}(y,s)|}{|\mathrm{Q}_{\rho}(y,s)|}\nonumber\\
&\leq& \frac{|\mathrm{Q}_{r}(x_{0},t_{0})|}{|\mathrm{Q}_{\rho}(y,s)|} = \left(\frac{r}{\rho}\right)^{n+2},
\end{eqnarray}
for almost every $(y,s)\in \Omega_{\mathrm{T}}$ and for all $\rho>0$. Now, taking $\rho>(2^{k+1}-1)r$ and $(y,s)\in \mathrm{Q}_{2^{k+1}r}(x_{0},t_{0})\setminus \mathrm{Q}_{2^{k}r}(x_{0},t_{0})$, from \eqref{est4teo110} we deduce that
\begin{eqnarray*}
0<\intav{\mathrm{Q}_{\rho}(y,s)}\chi_{\mathrm{Q}_{r}(x_{0},t_{0})}(x,t)\,dxdt = \left(\frac{r}{\rho}\right)^{n+2} &\leq& \frac{1}{(2^{k+1}-1)^{n+2}}\\
&\leq& \frac{1}{2^{(k-1)(n+2)}}, \quad \forall k\in\mathbb{N},
\end{eqnarray*}
since $2^{k+1}-1 \geq 2^{k-1}$ for all $k\in\mathbb{N}$. 

On the other hand, if $0<\rho \leq (2^{k+1}-1)r$, then $\mathrm{Q}_{r}(x_{0},t_{0})\cap \mathrm{Q}_{\rho}(y,s) = \emptyset$, since $\mathrm{B}_{r}(x_{0})\cap \mathrm{B}_{\rho}(y) = \emptyset$. Hence, from both cases, we conclude that
\begin{eqnarray}\label{est5teo110}
\left(\mathcal{M}(\chi_{\mathrm{Q}_{r}(x_{0},t_{0})})(y,s)\right)^{\iota} = \left(\sup_{\rho>0}\intav{\mathrm{Q}_{\rho}(y,s)}\chi_{\mathrm{Q}_{r}(x_{0},t_{0})}(x,t)\,dxdt\right)^{\iota} \leq \frac{1}{2^{\iota(k-1)(n+2)}},
\end{eqnarray}
for all $(y,s)\in \mathrm{Q}_{2^{k+1}r}(x_{0},t_{0})\setminus \mathrm{Q}_{2^{k}r}(x_{0},t_{0})$. Therefore, from \eqref{est5teo110} it follows that
\begin{eqnarray}\label{est6teo110}
A_{k} &\leq& \frac{1}{2^{\iota(k-1)(n+2)}}\int_{\mathrm{Q}_{2^{k+1}r}(x_{0},t_{0})\setminus\mathrm{Q}_{2^{k}r}(x_{0},t_{0})}|f(x,t)|^{\varsigma(x,t)}\,dxdt\nonumber\\
&\leq& \frac{1}{2^{\iota(k-1)(n+2)}}\int_{\Omega_{\mathrm{T}}(x_{0},t_{0};2^{k+1}r)}|f(x,t)|^{\varsigma(x,t)}\,dxdt\nonumber\\
&\leq& \frac{1}{2^{\iota(k-1)(n+2)}}(2^{k+1}r)^{\varrho(x_{0},t_{0})}\|f\|_{L^{\varsigma(\cdot),\varrho(\cdot)}(\Omega_{\mathrm{T}})}\nonumber\\
&=& 2^{(n+2)\iota + \varrho(x_{0},t_{0})} \cdot 2^{k(\varrho(x_{0},t_{0}) - (n+2)\iota)} r^{\varrho(x_{0},t_{0})} \|f\|_{L^{\varsigma(\cdot),\varrho(\cdot)}(\Omega_{\mathrm{T}})}\nonumber\\
&\leq& 2^{2(n+2)} \cdot 2^{k(\varrho(x_{0},t_{0}) - (n+2)\iota)} r^{\varrho(x_{0},t_{0})} \|f\|_{L^{\varsigma(\cdot),\varrho(\cdot)}(\Omega_{\mathrm{T}})},
\end{eqnarray}
where the last inequality follows from the assumption $\varrho(x_{0},t_{0}) < n+2$ and the estimate $2^{(n+2)\iota + \varrho(x_{0},t_{0})} \leq 2^{2(n+2)}$.
\end{itemize}

Consequently, from estimates \eqref{estA0} and \eqref{est6teo110}, we deduce that
\begin{eqnarray}\label{est7teo110}
\int_{\Omega_{\mathrm{T}}}|f(x,t)|^{\varsigma(x,t)}\,dxdt &\leq& 4^{n+2}r^{\varrho(x_{0},t_{0})}\|f\|_{L^{\varsigma(\cdot),\varrho(\cdot)}(\Omega_{\mathrm{T}})} \left(1+\sum_{k=1}^{\infty}2^{(\varrho(x_{0},t_{0})-(n+2)\iota)k}\right)\nonumber\\
&=& 4^{n+2}r^{\varrho(x_{0},t_{0})} \|f\|_{L^{\varsigma(\cdot),\varrho(\cdot)}(\Omega_{\mathrm{T}})} \sum_{k=0}^{\infty}2^{(\varrho(x_{0},t_{0})-(n+2)\iota)k} \nonumber\\
&\stackrel{\varrho(x_{0},t_{0})\leq\varrho_{0}}{\leq}& 4^{n+2}r^{\varrho(x_{0},t_{0})} \|f\|_{L^{\varsigma(\cdot),\varrho(\cdot)}(\Omega_{\mathrm{T}})} \sum_{k=0}^{\infty}2^{(\varrho_{0}-(n+2)\iota)k} \nonumber\\
&=& 4^{n+2}r^{\varrho(x_{0},t_{0})} \|f\|_{L^{\varsigma(\cdot),\varrho(\cdot)}(\Omega_{\mathrm{T}})} \sum_{k=0}^{\infty} \frac{1}{2^{((n+2)\iota-\varrho_{0})k}} \nonumber\\
&\stackrel{\iota > \frac{\varrho_{0}}{n+2}}{=}& \mathrm{C}'' r^{\varrho(x_{0},t_{0})} \|f\|_{L^{\varsigma(\cdot),\varrho(\cdot)}(\Omega_{\mathrm{T}})},
\end{eqnarray}
where $\mathrm{C}'' = 4^{n+2} \frac{2^{(n+2)\iota-\varrho_{0}}}{2^{(n+2)\iota-\varrho_{0}}-1} > 0$. Hence, from estimates \eqref{est3teo110} and \eqref{est7teo110}, we obtain that
\begin{eqnarray}\label{est8teo110}
\frac{1}{r^{\varrho(x_{0},t_{0})}} \int_{\Omega_{\mathrm{T}}(x_{0},t_{0};r)} |D^{2}u|^{\varsigma(x,t)}\,dxdt &\leq& \frac{\mathrm{C}}{r^{\varrho(x_{0},t_{0})}} r^{\varrho(x_{0},t_{0})} \|f\|_{L^{\varsigma(\cdot),\varrho(\cdot)}(\Omega_{\mathrm{T}})} \nonumber\\
&=& \mathrm{C} \|f\|_{L^{\varsigma(\cdot),\varrho(\cdot)}(\Omega_{\mathrm{T}})}.
\end{eqnarray}

Taking the supremum in \eqref{est8teo110} over all $(x_{0},t_{0})\in \Omega_{\mathrm{T}}$ and $r > 0$, we obtain
\begin{eqnarray}
\rho_{\varsigma(\cdot),\varrho(\cdot)}(|D^{2}u|) \leq \mathrm{C} \|f\|_{L^{\varsigma(\cdot),\varrho(\cdot)}(\Omega_{\mathrm{T}})} = \mathrm{C} < \infty,
\end{eqnarray}
which implies that $D^{2}u \in L^{\varsigma(\cdot),\varrho(\cdot)}(\Omega_{\mathrm{T}})$ and
\begin{eqnarray*}
\rho_{\varsigma(\cdot),\varrho(\cdot)}\left(\frac{|D^{2}u|}{\mathrm{C}}\right) \leq 1 \quad \Longrightarrow \quad \left\| \frac{D^{2}u}{\mathrm{C}} \right\|_{L^{\varsigma(\cdot),\varrho(\cdot)}(\Omega_{\mathrm{T}})} \leq 1 = \|f\|_{L^{\varsigma(\cdot),\varrho(\cdot)}(\Omega_{\mathrm{T}})}
\end{eqnarray*}
by the norm-modular unit ball property. More precisely,
\begin{eqnarray*}
\|D^{2}u\|_{L^{\varsigma(\cdot),\varrho(\cdot)}(\Omega_{\mathrm{T}})} \leq \mathrm{C} \|f\|_{L^{\varsigma(\cdot),\varrho(\cdot)}(\Omega_{\mathrm{T}})}.
\end{eqnarray*}
This concludes the proof of the desired estimate.
\end{proof}

As a consequence of Theorem \ref{T2}, we obtain a variable exponent H\"{o}lder continuity of the gradient for viscosity solutions of \eqref{1.1} in the case where $\gamma = g = 0$. 

Recall that, given a continuous function $\alpha : \overline{\Omega_{\mathrm{T}}} \to [0,+\infty)$, the \textit{variable exponent H\"{o}lder space} $C^{0,\alpha(\cdot)}(\overline{\Omega_{\mathrm{T}}})$ is defined as the set of all functions $u : \overline{\Omega_{\mathrm{T}}} \to \mathbb{R}$ such that
\begin{eqnarray*}
[u]_{\alpha(\cdot),\overline{\Omega_{\mathrm{T}}}} := \sup_{\substack{(x,t),(y,s)\in\overline{\Omega_{\mathrm{T}}}\\(x,t)\neq (y,s)}} \frac{|u(x,t) - u(y,s)|}{d_{p}((x,t),(y,s))^{\alpha(x,t)}} < \infty,
\end{eqnarray*}
where the associated norm is given by
\[
\|u\|_{C^{0,\alpha(\cdot)}(\overline{\Omega_{\mathrm{T}}})} := \|u\|_{L^{\infty}(\overline{\Omega_{\mathrm{T}}})} + [u]_{\alpha(\cdot),\overline{\Omega_{\mathrm{T}}}}.
\]

The variable exponent H\"{o}lder continuity of the gradient for classical $W^{2,\varsigma(\cdot),\varrho(\cdot)}$ solutions to \eqref{1.1} follows from the corollary below and an application of the Campanato-type theorem in the parabolic setting (see Theorem \ref{campanato} in Appendix A). The proof proceeds along similar lines to the argument in \cite[Corollary 3.1]{Tang} (see also \cite[Corollary 6.1]{ZZF}).

\begin{corollary}\label{cor5.2}
Under the assumptions of Theorem \ref{T2}, let $u \in W^{2,\varsigma(\cdot),\varrho(\cdot)}(\Omega_{\mathrm{T}})$ be an $L^{\varsigma_{1}}$-viscosity solution to \eqref{1.1}, with $\gamma = g = 0$. Suppose that the exponent functions $\varsigma$ and $\varrho$ satisfy $\varsigma(\cdot) + \varrho(\cdot) > n + 2$. Then, the gradient $Du$ belongs to the variable exponent H\"{o}lder space 
\[
Du \in C^{0,1 - \frac{n+2 - \varrho(\cdot)}{\varsigma(\cdot)}}(\overline{\Omega_{\mathrm{T}}}).
\]
\end{corollary}

\subsection*{Appendix A: Variable Exponent Campanato Spaces - Parabolic Setting}

This Appendix presents a version of Campanato's theorem adapted to the parabolic context, providing the foundation for the optimal regularity result stated in Corollary~\ref{cor5.2}. The development presented here is inspired by the ideas of Fan in~\cite{Fan}.

We denote by $\mathcal{M}(\Omega_{\mathrm{T}})$ the space of all measurable functions on $\Omega_{\mathrm{T}}$, identifying two functions as equivalent if they differ only on a set of measure zero.

We recall the definition of the variable exponent Lebesgue space $L^{\varsigma(\cdot)}(\Omega_{\mathrm{T}})$, associated with a measurable function $\varsigma: \Omega_{\mathrm{T}} \to [1,+\infty)$, defined as
\begin{eqnarray*}
L^{\varsigma(\cdot)}(\Omega_{\mathrm{T}}) = \left\{ u \in \mathcal{M}(\Omega_{\mathrm{T}}) \; ; \; \exists c > 0 \text{ such that } \rho_{\varsigma(\cdot),\Omega_{\mathrm{T}}}\left(\frac{u}{c}\right) \defeq \int_{\Omega_{\mathrm{T}}} \left|\frac{u(x,t)}{c}\right|^{\varsigma(x,t)} dx\,dt < \infty \right\},
\end{eqnarray*}
and equipped with the \textit{Luxemburg norm}
\begin{eqnarray*}
\|u\|_{L^{\varsigma(\cdot)}(\Omega_{\mathrm{T}})} = \inf \left\{ c > 0 \; ; \; \rho_{\varsigma(\cdot),\Omega_{\mathrm{T}}}\left( \frac{u}{c} \right) \leq 1 \right\}.
\end{eqnarray*}

Associated with the function $\varsigma$, we define its conjugate exponent $\varsigma^{\sharp}$ at each point $(x,t) \in \Omega_{\mathrm{T}}$ by
\begin{eqnarray*}
\varsigma^{\sharp}(x,t) =
\begin{cases}
+\infty, & \text{if } \varsigma(x,t) = 1, \\
\frac{\varsigma(x,t)}{\varsigma(x,t)-1}, & \text{otherwise}.
\end{cases}
\end{eqnarray*}
Observe that the usual conjugation relation holds:
\begin{eqnarray*}
\frac{1}{\varsigma^{\sharp}(x,t)} + \frac{1}{\varsigma(x,t)} = 1, \quad \forall (x,t) \in \Omega_{\mathrm{T}}.
\end{eqnarray*}

Following \cite[Proposition 2.2]{Fan}, we have the following estimate for the $L^{\varsigma(\cdot)}$ norm of the characteristic function $\chi_{\Omega_{\mathrm{T}}}$:
\begin{eqnarray*}
\|\chi_{\Omega_{\mathrm{T}}}\|_{L^{\varsigma(\cdot)}(\Omega_{\mathrm{T}})} \leq \max\left\{|\Omega_{\mathrm{T}}|^{\frac{1}{\varsigma_{-}}}, |\Omega_{\mathrm{T}}|^{\frac{1}{\varsigma_{+}}} \right\},
\end{eqnarray*}
where $\varsigma_{-} = \displaystyle\operatornamewithlimits{essinf}_{\Omega_{\mathrm{T}}} \varsigma(\cdot)$ and $\varsigma_{+} = \displaystyle\operatornamewithlimits{esssup}_{\Omega_{\mathrm{T}}} \varsigma(\cdot)$ denote the essential infimum and supremum, respectively. Furthermore, if both $\varsigma_{-}$ and $\varsigma_{+}$ are attained in $\overline{\Omega_{\mathrm{T}}}$, then there exists a point $(x_{0},t_{0}) \in \overline{\Omega_{\mathrm{T}}}$ such that
\begin{eqnarray*}
\|\chi_{\Omega_{\mathrm{T}}}\|_{L^{\varsigma(\cdot)}(\Omega_{\mathrm{T}})} \leq |\Omega_{\mathrm{T}}|^{\frac{1}{\varsigma(x_{0},t_{0})}}.
\end{eqnarray*}

For the remainder of this Appendix, we assume that $\varsigma$ is a continuous Log-H\"{o}lder function and that the domain $\Omega_{\mathrm{T}}$ has no cusps. More precisely, there exists a constant $\mathrm{C}_0 > 0$ such that
\begin{eqnarray*}
|\Omega_{\mathrm{T}}(x,t;r)| \geq \mathrm{C}_0 |\mathrm{Q}_{r}(x,t)|, \quad \forall (x,t) \in \overline{\Omega_{\mathrm{T}}} \text{ and } 0 < r \leq \operatorname{diam}(\Omega_{\mathrm{T}}).
\end{eqnarray*}
Note that these conditions are satisfied in Corollary~\ref{cor5.2}, since $\Omega$ is assumed to be a $C^{2,\alpha}$ domain.

We now recall the definition of the variable exponent Campanato space:

\begin{definition}
Let $\varsigma:\Omega_{\mathrm{T}} \to [1,+\infty)$ and $\varrho:\Omega_{\mathrm{T}} \to [0,+\infty)$ be measurable functions. The \textit{variable exponent Campanato space} $\mathfrak{L}^{\varsigma(\cdot),\varrho(\cdot)}(\Omega_{\mathrm{T}})$ is defined as
\begin{eqnarray*}
\mathfrak{L}^{\varsigma(\cdot),\varrho(\cdot)}(\Omega_{\mathrm{T}}) = \left\{ u \in L^{\varsigma(\cdot)}(\Omega_{\mathrm{T}}) : \sup_{\genfrac{}{}{0pt}{}{(x_{0},t_{0})\in \Omega_{\mathrm{T}}}{r>0}} r^{\frac{-\varrho(x_{0},t_{0})}{\varsigma(x_{0},t_{0})}} \left\| u - u_{\Omega_{\mathrm{T}}(x_{0},t_{0};r)} \right\|_{L^{\varsigma(\cdot)}(\Omega_{\mathrm{T}}(x_{0},t_{0};r))} < \infty \right\}
\end{eqnarray*}
equipped with the norm
\begin{eqnarray*}
\|u\|_{\mathfrak{L}^{\varsigma(\cdot),\varrho(\cdot)}(\Omega_{\mathrm{T}})} = \|u\|_{L^{\varsigma(\cdot)}(\Omega_{\mathrm{T}})} + [u]_{\mathfrak{L}^{\varsigma(\cdot),\varrho(\cdot)}(\Omega_{\mathrm{T}})},
\end{eqnarray*}
where the seminorm is given by
\begin{eqnarray*}
[u]_{\mathfrak{L}^{\varsigma(\cdot),\varrho(\cdot)}(\Omega_{\mathrm{T}})} \defeq \sup_{\genfrac{}{}{0pt}{}{(x_{0},t_{0})\in \Omega_{\mathrm{T}}}{r>0}} r^{\frac{-\varrho(x_{0},t_{0})}{\varsigma(x_{0},t_{0})}} \left\| u - u_{\Omega_{\mathrm{T}}(x_{0},t_{0};r)} \right\|_{L^{\varsigma(\cdot)}(\Omega_{\mathrm{T}}(x_{0},t_{0};r))}.
\end{eqnarray*}
\end{definition}

\begin{remark}
As established in \cite[Corollary 4.1]{Fan}, if both $\varsigma$ and $\varrho$ are Log-H\"{o}lder continuous functions, then there exists a continuous embedding of the variable exponent Morrey space $L^{\varsigma(\cdot),\varrho(\cdot)}(\Omega_{\mathrm{T}})$ into the Campanato space $\mathfrak{L}^{\varsigma(\cdot),\varrho(\cdot)}(\Omega_{\mathrm{T}})$.
\end{remark}

For simplicity, we henceforth adopt the notation $\Omega_{\mathrm{T}}^{r} \defeq \Omega_{\mathrm{T}}(x,t;r)$ for $r > 0$ and $(x,t) \in \Omega_{\mathrm{T}}$.

\begin{lemma}\label{LA1}
Let $u \in \mathfrak{L}^{\varsigma(\cdot),\varrho(\cdot)}(\Omega_{\mathrm{T}})$. Then, there exists a constant $\mathrm{C} = \mathrm{C}(n,\varsigma_{+},\mathrm{C}_{\varsigma},\operatorname{diam}(\Omega_{\mathrm{T}}))$ such that for every $(x_{0},t_{0}) \in \Omega_{\mathrm{T}}$ and $0 < r < s \leq \operatorname{diam}(\Omega_{\mathrm{T}})$, the following inequality holds:
\begin{eqnarray*}
|u_{s} - u_{r}| \leq \mathrm{C} |\Omega_{\mathrm{T}}^{r}|^{\tau(x_{0},t_{0})} \left( \frac{s}{r} \right)^{\frac{\varrho(x_{0},t_{0})}{\varsigma(x_{0},t_{0})}} [u]_{\mathfrak{L}^{\varsigma(\cdot),\varrho(\cdot)}(\Omega_{\mathrm{T}})},
\end{eqnarray*}
where 
\begin{eqnarray*}
\tau(x,t) = \frac{\varrho(x,t) - (n+2)}{(n+2)\varsigma(x,t)}, \quad \forall (x,t) \in \Omega_{\mathrm{T}}.
\end{eqnarray*}
\end{lemma}

\begin{proof}
The proof follows the same strategy as in \cite[Lemma 4.1]{Fan}, with a minor modification. Namely, due to the definition of $\Omega_{\mathrm{T}}^{r}$, we have that $|\Omega_{\mathrm{T}}^{r}| = r^{n+2}|\Omega_{\mathrm{T}}|$ for all $r > 0$.
\end{proof}

As a consequence of Lemma \ref{LA1}, the estimate simplifies significantly depending on the sign of $\tau(x_{0},t_{0})$:

\begin{lemma}
Assume that the hypotheses of Lemma \ref{LA1} hold. Let $(x_{0},t_{0})\in\Omega_{\mathrm{T}}$ be a fixed point and let $0<\rho<R\leq\operatorname{diam}(\Omega_{\mathrm{T}})$. Then, there exist two positive constants $c_{1}$ and $c_{2}$ depending only on $\tau(x_{0},t_{0})$, $\mathrm{C}_{\varsigma}$, $\varsigma_{+}$, $\operatorname{diam}(\Omega_{\mathrm{T}})$, and the geometric condition $\mathrm{C}_0$ (which ensures that $\Omega$ has no cusps), such that the following estimates hold:
\begin{itemize}
\item[i.] $|u_{R}-u_{\rho}|\leq c_{1}|\Omega_{\mathrm{T}}^{R}|^{\tau(x_{0},t_{0})}[u]_{\mathfrak{L}^{\varsigma(\cdot),\varrho(\cdot)}(\Omega_{\mathrm{T}})}$ when $\tau(x_{0},t_{0})>0$;
\item[ii.] $|u_{R}-u_{\rho}|\leq c_{2}|\Omega_{\mathrm{T}}^{\rho}|^{\tau(x_{0},t_{0})}[u]_{\mathfrak{L}^{\varsigma(\cdot),\varrho(\cdot)}(\Omega_{\mathrm{T}})}$ when $\tau(x_{0},t_{0})<0$.
\end{itemize}
\end{lemma}

\begin{proof}
The idea of the proof is to apply Lemma \ref{LA1} iteratively to derive the desired estimates. To this end, for each $k\geq 0$, define $r_{k}=2^{-k}R$. Fixing $k\geq 1$, we apply Lemma \ref{LA1} with $r=r_{k}$ and $s=r_{k-1}$ (note that $0<r<s\leq R\leq \operatorname{diam}(\Omega_{\mathrm{T}})$) to obtain
\begin{eqnarray}\label{eq1}
|u_{r_{k-1}} - u_{r_{k}}| &\leq& \mathrm{C}|\Omega_{\mathrm{T}}^{r_{k}}|^{\tau(x_{0},t_{0})} \underbrace{\left(\frac{r_{k-1}}{r_{k}}\right)^{\frac{\varrho(x_{0},t_{0})}{\varsigma(x_{0},t_{0})}}}_{=2^{\frac{\varrho(x_{0},t_{0})}{\varsigma(x_{0},t_{0})}}} [u]_{\mathfrak{L}^{\varsigma(\cdot),\varrho(\cdot)}(\Omega_{\mathrm{T}})} \nonumber \\
&\leq& 2^{\frac{\varrho_{+}}{\varsigma_{-}}} \mathrm{C} |\Omega_{\mathrm{T}}^{r_{k}}|^{\tau(x_{0},t_{0})} [u]_{\mathfrak{L}^{\varsigma(\cdot),\varrho(\cdot)}(\Omega_{\mathrm{T}})}.
\end{eqnarray}

Next, we transition from the discrete to the continuous setting. Let $k\in\mathbb{N}$ be such that $r_{k} \leq \rho < r_{k-1}$. Then, by the triangle inequality, we obtain
\begin{eqnarray}\label{eq2}
|u_{R} - u_{\rho}| &\leq& |u_{R} - u_{r_{k-1}}| + |u_{r_{k-1}} - u_{\rho}| \nonumber\\
&\leq& \sum_{j=1}^{k-1} |u_{r_{j-1}} - u_{r_{j}}| + |u_{r_{k-1}} - u_{\rho}| \nonumber\\
&\stackrel{\eqref{eq1}}{\leq}& 2^{\frac{\varrho_{+}}{\varsigma_{-}}} \mathrm{C} [u]_{\mathfrak{L}^{\varsigma(\cdot),\varrho(\cdot)}(\Omega_{\mathrm{T}})} \sum_{j=1}^{k-1} |\Omega_{\mathrm{T}}^{r_{j}}|^{\tau(x_{0},t_{0})} + |u_{r_{k-1}} - u_{\rho}| \nonumber\\
&=:& \mathrm{I} + \mathrm{II}.
\end{eqnarray}

We now estimate each term in \eqref{eq2}. For the term $\mathrm{I}$, since $\Omega_{\mathrm{T}}$ has no cusps, the following inequality holds for all $0<r'<r''\leq\operatorname{diam}(\Omega_{\mathrm{T}})$:
\begin{eqnarray}\label{eq3}
\mathrm{C}_0 \left(\frac{r''}{r'}\right)^{n+2} = \frac{\mathrm{C}_0|\mathrm{Q}_{r''}(x_{0},t_{0})|}{|\mathrm{Q}_{r'}(x_{0},t_{0})|} \leq \frac{|\Omega_{\mathrm{T}}^{r''}|}{|\Omega_{\mathrm{T}}^{r'}|} \leq \frac{|\mathrm{Q}_{r''}(x_{0},t_{0})|}{\mathrm{C}_0|\mathrm{Q}_{r'}(x_{0},t_{0})|} = \mathrm{C}_0^{-1} \left(\frac{r''}{r'}\right)^{n+2}.
\end{eqnarray}

We now prove the first inequality. Indeed, assume that $\tau(x_{0},t_{0}) > 0$. Then, from \eqref{eq3}, it follows that
\begin{eqnarray*}
\sum_{j=1}^{k-1} |\Omega_{\mathrm{T}}^{r_{j}}| &\leq& \sum_{j=1}^{k-1} \left( \mathrm{C}_0^{-1} |\Omega_{\mathrm{T}}^{R}| \left( \frac{r_{j}}{R} \right)^{n+2} \right)^{\tau(x_{0},t_{0})} \\
&\leq& \mathrm{C}_0^{-\tau(x_{0},t_{0})} |\Omega_{\mathrm{T}}^{R}|^{\tau(x_{0},t_{0})} \sum_{j=1}^{k-1} \frac{1}{2^{(n+2)\tau(x_{0},t_{0})j}} \\
&\leq& \frac{\mathrm{C}_0^{-\tau(x_{0},t_{0})}}{2^{(n+2)\tau(x_{0},t_{0})} - 1} |\Omega_{\mathrm{T}}^{R}|^{\tau(x_{0},t_{0})}.
\end{eqnarray*}

Therefore,
\begin{eqnarray}\label{eq4}
\mathrm{I} \leq \frac{2^{\frac{\varrho_{+}}{\varsigma_{-}}} \mathrm{C} \mathrm{C}_0^{-\tau(x_{0},t_{0})}}{2^{(n+2)\tau(x_{0},t_{0})} - 1} |\Omega_{\mathrm{T}}^{R}|^{\tau(x_{0},t_{0})} [u]_{\mathfrak{L}^{\varsigma(\cdot),\varrho(\cdot)}(\Omega_{\mathrm{T}})}.
\end{eqnarray}

For the term $\mathrm{II}$, we apply Lemma \ref{LA1} once more (since $\rho < r_{k-1}$) to obtain
\begin{eqnarray}\label{eq5}
\mathrm{II} &\leq& \mathrm{C} |\Omega_{\mathrm{T}}^{\rho}|^{\tau(x_{0},t_{0})} \underbrace{\left(\frac{r_{k-1}}{\rho}\right)^{\frac{\varrho(x_{0},t_{0})}{\varsigma(x_{0},t_{0})}}}_{\leq 2^{\frac{\varrho(x_{0},t_{0})}{\varsigma(x_{0},t_{0})}}} [u]_{\mathfrak{L}^{\varsigma(\cdot),\varrho(\cdot)}(\Omega_{\mathrm{T}})} \nonumber\\
&\leq& 2^{\frac{\varrho_{+}}{\varsigma_{-}}} \mathrm{C} |\Omega_{\mathrm{T}}^{\rho}|^{\tau(x_{0},t_{0})} [u]_{\mathfrak{L}^{\varsigma(\cdot),\varrho(\cdot)}(\Omega_{\mathrm{T}})} \nonumber\\
&\stackrel{\tau(x_{0},t_{0})>0}{\leq}& 2^{\frac{\varrho_{+}}{\varsigma_{-}}} \mathrm{C} |\Omega_{\mathrm{T}}^{R}|^{\tau(x_{0},t_{0})} [u]_{\mathfrak{L}^{\varsigma(\cdot),\varrho(\cdot)}(\Omega_{\mathrm{T}})},
\end{eqnarray}
since $\rho < R$. Hence, combining \eqref{eq4} and \eqref{eq5}, we establish item (i) of the lemma with
\begin{eqnarray*}
c_{1} = 2^{\frac{\varrho_{+}}{\varsigma_{-}}} \mathrm{C} \frac{2^{(n+2)\tau(x_{0},t_{0})} - 1 + \mathrm{C}_0^{-\tau(x_{0},t_{0})}}{2^{(n+2)\tau(x_{0},t_{0})} - 1} > 0.
\end{eqnarray*}

Finally, the case $\tau(x_{0},t_{0}) < 0$ can be addressed in a completely analogous manner (cf. \cite[Lemma 4.2]{Fan}).
\end{proof}

\begin{theorem}[\bf Campanato-type Theorem - Parabolic Case]\label{campanato}
Under the same assumptions as above, suppose that $\varrho(x,t)>n+2$ for all $(x,t)\in\Omega_{\mathrm{T}}$. Then, the spaces $\mathfrak{L}^{\varsigma(\cdot),\varrho(\cdot)}(\Omega_{\mathrm{T}})$ and $C^{0,\alpha(\cdot)}(\overline{\Omega_{\mathrm{T}}})$ are isomorphic, i.e.,
\[
\mathfrak{L}^{\varsigma(\cdot),\varrho(\cdot)}(\Omega_{\mathrm{T}}) \cong C^{0,\alpha(\cdot)}(\overline{\Omega_{\mathrm{T}}}),
\]
where
\begin{equation*}
\alpha(x,t) = \frac{\varrho(x,t)-(n+2)}{\varsigma(x,t)}, \quad \forall (x,t) \in \overline{\Omega_{\mathrm{T}}}.
\end{equation*}
\end{theorem}

\begin{proof}
We first prove the inclusion $C^{0,\alpha(\cdot)}(\overline{\Omega_{\mathrm{T}}})\subset \mathfrak{L}^{\varsigma(\cdot),\varrho(\cdot)}(\Omega_{\mathrm{T}})$. 

Let $u \in C^{0,\alpha(\cdot)}(\overline{\Omega_{\mathrm{T}}})$, and fix a point $(x_{0},t_{0}) \in \Omega_{\mathrm{T}}$ and a radius $0<r<\mathrm{diam}(\Omega_{\mathrm{T}})$. For any $(x,t) \in \Omega_{\mathrm{T}}^{r}$, we estimate
\begin{align*}
|u(x,t) - u_{r}| 
&\leq \frac{1}{|\Omega_{\mathrm{T}}^{r}|} \int_{\Omega_{\mathrm{T}}^{r}} |u(x,t) - u(y,s)| \, dyds \\
&\leq \frac{1}{|\Omega_{\mathrm{T}}^{r}|} \int_{\Omega_{\mathrm{T}}^{r}} [u]_{\alpha(\cdot),\overline{\Omega_{\mathrm{T}}}} \, d_{p}((x,t),(y,s))^{\alpha(x,t)} \, dyds \\
&\leq \frac{1}{|\Omega_{\mathrm{T}}^{r}|} [u]_{\alpha(\cdot),\overline{\Omega_{\mathrm{T}}}} \, \max\{2r,r^{2}\}^{\alpha(x,t)} |\Omega_{\mathrm{T}}^{r}| \\
&= [u]_{\alpha(\cdot),\overline{\Omega_{\mathrm{T}}}} \, \max\{2r,r^{2}\}^{\alpha(x,t)} \\
&\leq \mathrm{C}'\left(\mathrm{diam}(\Omega_{\mathrm{T}}), [u]_{\alpha(\cdot),\overline{\Omega_{\mathrm{T}}}} \right) \max\{2r,r^{2}\}^{\alpha(x_{0},t_{0})}.
\end{align*}

Therefore, we obtain
\begin{align*}
r^{-\frac{\varrho(x_{0},t_{0})}{\varsigma(x_{0},t_{0})}} \|u - u_{r}\|_{L^{\varsigma(\cdot),\varrho(\cdot)}(\Omega_{\mathrm{T}}^{r})}
&\leq r^{-\frac{\varrho(x_{0},t_{0})}{\varsigma(x_{0},t_{0})}} \, \mathrm{C} \, \|\chi_{\Omega_{\mathrm{T}}^{r}}\|_{L^{\varsigma(\cdot),\varrho(\cdot)}(\Omega_{\mathrm{T}}^{r})} \, \max\{2r,r^{2}\}^{\alpha(x_{0},t_{0})} \\
&\leq \mathrm{C} \, r^{-\frac{\varrho(x_{0},t_{0})}{\varsigma(x_{0},t_{0})}} \max\{2r,r^{2}\}^{\alpha(x_{0},t_{0})} \, r^{\frac{n+2}{\varrho(x_{0},t_{0})}} \\
&= \mathrm{C} \, r^{-\alpha(x_{0},t_{0})} \max\{2r,r^{2}\}^{\alpha(x_{0},t_{0})} \\
&= \mathrm{C} \, \max\{2,r\}^{\alpha(x_{0},t_{0})} \\
&\leq \mathrm{C}''\left(\mathrm{diam}(\Omega_{\mathrm{T}}), |\Omega_{\mathrm{T}}|, [u]_{\alpha(\cdot),\overline{\Omega_{\mathrm{T}}}} \right),
\end{align*}
where the constants $\mathrm{C}', \ \mathrm{C}'' > 0$ are independent of the point $(x_{0},t_{0}) \in \Omega_{\mathrm{T}}$ and the radius $0<r<\mathrm{diam}(\Omega_{\mathrm{T}})$. This shows that $u \in \mathfrak{L}^{\varsigma(\cdot),\varrho(\cdot)}(\Omega_{\mathrm{T}})$.

The converse inclusion follows analogously to the argument presented in \cite[Theorem 4.3]{Fan}. For brevity, we omit the details here.
\end{proof}

\subsection*{Acknowledgments}

\hspace{0.4cm} CAPES-Brazil partially supported J.S. Bessa under Grant No. 88887.482068/2020-00. FAPESP-Brazil has supported J.S. Bessa under Grant No. 2023/18447-3. J.V. da Silva has received partial support from CNPq-Brazil under Grant No. 307131/2022-0, FAEPEX-UNICAMP (Project No. 2441/23, Special Calls - PIND - Individual Projects, 03/2023), and Chamada CNPq/MCTI No. 10/2023 - Faixa B - Consolidated Research Groups under Grant No. 420014/2023-3. G.C. Ricarte has been partially supported by CNPq-Brazil under Grants No. 304239/2021-6. We would like to thank the anonymous Referee for the careful reading and suggestions throughout the paper.

\end{document}